\documentclass[11pt,reqno]{amsart}
\usepackage[foot]{amsaddr}
 
\usepackage[left=2.5cm, right=2.5cm, top=2.5cm, bottom=3cm]{geometry}
\usepackage[pagebackref, linkcolor=green, urlcolor=cyan]{hyperref}

\usepackage{graphicx}
\usepackage{amsfonts,amsmath, amssymb,amsthm,amscd}

\numberwithin{equation}{section}

\usepackage{bm}
\usepackage{bbm}
\usepackage{mathrsfs}
\usepackage{yfonts}
\usepackage{verbatim}
\usepackage{xcolor}
\usepackage{upgreek}
\usepackage{relsize}
\usepackage{mathdots}
\usepackage{mathtools}
\usepackage{tikz} 
\usepackage{caption}
\usepackage{subcaption}

\usepackage{romannum}

\newtheorem{theorem}{Theorem}[section]

\newtheorem{proposition}[theorem]{Proposition}
\newtheorem{corollary}[theorem]{Corollary}

\theoremstyle{definition}
\newtheorem{remark}[theorem]{Remark}

\theoremstyle{definition}

\theoremstyle{definition}
\newtheorem{definition}[theorem]{Definition}

\theoremstyle{definition}

\theoremstyle{definition}

\theoremstyle{definition}

\newcommand{\bsb}[1]{\boldsymbol{#1}}
\newcommand{\ggl}[1]{\mathfrak{g} \mathfrak{l}_{#1}}
\newcommand{\gso}[1]{\mathfrak{s} \mathfrak{o}_{#1}}
\newcommand{\ci}[0]{\mathcal{I}}
\renewcommand{\geq}{\geqslant}

\renewcommand{\leq}{\leqslant}

\newcommand{\rflat}{\raisebox{4pt}{\begin{tikzpicture}[scale=0.1] \draw[thick] (0,0) -- (3,0) -- (2,1) -- (1,1) -- (0,0);  \end{tikzpicture}}}
\newcommand{\symflat}{\begin{tikzpicture}[scale=0.07] \draw[thick] (0,0) -- (3,0) -- (2,1) -- (1,1) -- (0,0); \draw[thick] (0,0) -- (0,3) -- (1,2) -- (1,1) -- (0,0); 
\end{tikzpicture}}
\newcommand{\symtrapz}{\begin{tikzpicture}[scale=0.1] \draw[thick] (0,0) -- (2,0) -- (2,1) -- (1,1) -- (0,0); \draw[thick] (0,0) -- (0,2) -- (1,2) -- (1,1) -- (0,0); 
\end{tikzpicture}}

\newcommand{\personalcirc}{\begin{tikzpicture}[scale=0.75]
	 \draw   circle(2pt);
	\end{tikzpicture}}

\title{An identity in distribution between full-space  and half-space log-gamma polymers}

\author[G. Barraquand]{Guillaume Barraquand}
\address{G. Barraquand,
Laboratoire de Physique de l'{\'e}cole normale sup\'erieure, ENS, Universit{\'e} PSL, CNRS, Sorbonne Universit{\'e}, Universit{\'e} de Paris, Paris, France}
\email{guillaume.barraquand@ens.fr}
\author[S. Wang]{Shouda Wang}
\address{Shouda Wang,  \'Ecole Polytechnique, Institut Polytechnique de Paris,  Route de Saclay, F-91120 Palaiseau Cedex, France and 
Laboratoire de Physique de l'{\'e}cole normale sup\'erieure, ENS, Universit{\'e} PSL, CNRS, Sorbonne Universit{\'e}, Universit{\'e} de Paris, Paris, France} \email{shoudawang@princeton.edu}

\begin{document} 
\pagenumbering{arabic}

\renewcommand{\thesubfigure}{\Alph{subfigure}}
\maketitle 
\begin{abstract}
    We prove an identity in distribution between two kinds of partition functions for the log-gamma directed polymer model: (1) the point-to-point partition function in a quadrant, (2) the point-to-line partition function in an octant. 
    As an application, we prove that the point-to-line free energy of the log-gamma polymer in an octant obeys a phase transition depending on the strength of the noise along the boundary. This transition of (de)pinning by randomness was first predicted in physics by Kardar in 1985 and proved rigorously for zero temperature models by Baik and Rains in 2001. While it is expected to arise universally for models in the Kardar-Parisi-Zhang universality class, this is the first positive temperature model for which this transition can be rigorously established.
\end{abstract}
\setcounter{tocdepth}{1}
\tableofcontents

 \section{Introduction and main results} 
\subsection{Preface}
 Models of directed polymers in $1+1$ dimensions confined to a half-space are expected to obey a phase transition depending on the strength of the noise along the boundary. This phenomenon was predicted in physics by Kardar in \cite{kardar1985depinning} and termed depinning by quenched randomness. A more precise understanding of this transition came from the study of specific exactly solvable models by  Baik and Rains \cite{BaikRains,BaikRains2}. When the strength of the disorder on the boundary is below a certain threshold, the asymptotic  behaviour of the free energy is the same as without a boundary: fluctuations are governed by  a one third scaling exponent and statistics related to random matrix theory, as usually for models in the Kardar-Parisi-Zhang (KPZ) universality class.   Above this threshold, free energy fluctuations are Gaussian distributed on a square root scale, as if fluctuations were only determined by the noise along the boundary.   
 
 \medskip 
 
 In order to establish these results, \cite{BaikRains} applied the RSK correspondence on random matrices satisfying various symmetries. For each symmetry type, they computed the exact distribution of the RSK shape in terms of Schur functions. This enabled them to   study in particular an exactly solvable model of last passage percolation, i.e., a directed polymer at zero temperature, where the  symmetry imposed on the weights of the model is equivalent to confining polymer paths in a half-space (an octant of the $\mathbb Z^2$ lattice). The distribution of the free energy can then be characterized by determinantal and Pfaffian formulas, which can be asymptotically analyzed \cite{BaikRains2}. 
 
 \medskip
 
 All models studied in \cite{BaikRains, BaikRains2} (applications of \cite{BaikRains, BaikRains2} are reviewed in \cite{baik1999symmetrized}) can be considered as free-fermionic models. Positive temperature models are however typically not free-fermionic and arguably more difficult to study. Some of the results of Baik and Rains have been lifted to the positive temperature setting in \cite{OSZ, BisiZygouras, HalfSpaceMac, bisi2020geometric}, using the geometric RSK correspondence instead of the usual one, and Whittaker functions instead of Schur functions. These works provide concise exact formulas for observables of interest of various models, but none of these formulas seem to be readily amenable for asymptotic analysis -- see e.g. the discussion in \cite[Sections 1.4 and 8]{HalfSpaceMac}. Thus, a proof of a depinning phase transition for a half-space polymer model at positive temperature, or any model in the KPZ class that is not free-fermionic,  has remained out of reach. 
 
 \medskip 
 
 In this paper, we consider the point-to-line partition function (and more generally the point-to-half-line partition function) for a log-gamma polymer model in an octant. We prove an identity in distribution between those  partition functions and point-to-point partition functions in a quadrant. Relying on known asymptotics for the latter \cite{BCR13, KQ18, BCD}, we prove that the point-to-line partition function in an octant obeys a phase transition, depending on the parameter of weights along the boundary. \footnote{We remark that earlier progresses towards a proof of this phase transition \cite{OSZ, HalfSpaceMac} focused on the point-to-point partition function, which is expected to obey a similar phase transition with slightly different statistics. See predictions based on non-rigorous asymptotics in \cite{HalfSpaceMac}.} The identity in distribution is valid for inhomogenous weights distributions, depending on two families of inhomogeneity parameters. The proof is based on applying the geometric RSK algorithm to symmetric random matrices with inverse-gamma distributed weights, and computing the distribution of observables of interest as integrals involving $\ggl{n}(\mathbb R)$-Whittaker functions and their orthogonal analogues. This generalizes a similar computation in \cite{BisiZygouras}. The identity in distribution between the two partition functions may finally be recognized after transforming formulas using the spectral theory associated with Whittaker functions and comparing the Laplace transform formula for the point-to-line partition function to the one for the point-to-point partition function \cite{COSZ}.

 \subsection{The log-gamma polymer} 
 Let us define the log-gamma directed polymer model, first introduced in \cite{Seppa}. 

\begin{definition}[Partition function of log-gamma polymers]
Let $\mathcal{I}\subset \mathbb{N}\times\mathbb N$ be an index set. Let $(\theta_{i,j})_{(i,j)\in\mathcal{I}}$ be a parameter array with positive entries $\theta_{i,j}>0$ for all $(i,j)\in\mathcal{I}$.
We define the weight array of the log-gamma polymer with index set $\mathcal{I}$ and parameter array $(\theta_{i,j})$ as a family of independent inverse-gamma random variables $(W_{i,j})_{(i,j)\in\mathcal{I}}$ such that  
$W_{i,j}\sim \operatorname{Gamma}^{-1}(\theta_{i,j})$ (equivalently $W_{i,j}^{-1}\sim \operatorname{Gamma}(\theta_{i,j})$) for $(i,j)\in\mathcal{I}$. More precisely, the density function of $W_{i,j}$ is given by 
$$
f_{W_{i,j}}(x)=\frac{1}{\Gamma(\theta_{i,j})}x^{-\theta_{i,j}-1}e^{-\frac{1}{x}} ,\quad \text{for}\quad x>0.
$$
Suppose that $(1,1)$ and $(n,N)\in\ci$.
We define the (point-to-point) partition function of the log-gamma polymer via 
\begin{equation}\label{defZ}
 {Z(n,N)} = \sum_{\pi:(1,1)\to(n,N)}\prod_{(i,j)\in\pi}W_{i,j}.
\end{equation} 
where the sum is over all up-right paths from $(1,1)$ to $(n,N)$ confined in $\ci$. See Figure \ref{fullspaceloggamma}. We are particularly interested in the following two types of domains of $\mathbb Z^2$ and their corresponding index set: 
\begin{enumerate}
    \item[(\romannum 1)] Rectangular domains, with  index set $\ci=\{(i,j)\mid 1\leq i\leq n,1\leq j\leq m\}$ for some $n,m$, will be used to define \textit{full-space point-to-point partition functions}  as \eqref{defZ}. 
    \item[(\romannum 2)] Trapezoidal domains, where in particular the paths are confined inside an octant, with  index set $ \ci =\{(i,j)\mid 1\leq i\leq n, i\leq j\leq 2n+m-i+1 \}$ for some $n,m$, will be used to define \textit{trapezoidal point-to-line partition functions}. 
\end{enumerate}
\end{definition}
\begin{definition}\label{def:ptl}
 For trapezoidal domains with index set  $ \ci =\{(i,j)\mid 1\leq i\leq n, i\leq j\leq 2n+m-i+1 \}$,
 we define the trapezoidal point-to-line partition function via
\begin{equation}\label{defZflat}
Z^{ \rflat}(n;m):=\sum_{  \substack{1\leq k\leq n \\ \pi :(1,1)\to(k,2n-k+m+1)}  } \prod_{(i, j) \in \pi} W_{i,j}=\sum_{  \substack{1\leq k\leq n } } Z(k,2n-k+m+1) ,
\end{equation} 
where the superscript  \begin{tikzpicture}[scale=0.1] \draw[thick] (0,0) -- (3,0) -- (2,1) -- (1,1) -- (0,0);  \end{tikzpicture} is used to distinguish it from point-to-point partition functions.
 \end{definition}
 When $m=0$, the point-to-line partition function corresponds to the sum of weights of all paths of length $2n-1$ required to remain inside an octant of the $\mathbb Z^2$ lattice (see Figure \ref{Fig1B}). For arbitrary $m$, the point-to-line partition function can be thought of as a sum of products of weights over all up-right paths starting at the point $(1,1)$ and ending at any of the points on the right border of the trapezoid, as shown in Figure \ref{halfspaceparam}.

 \begin{figure}
 \centering
 \begin{subfigure}{.44\textwidth}
     \centering\begin{tikzpicture}[scale=0.6] 
     
		\draw[->,thick, gray] (0,0) -- (10, 0);
		\draw (-.5, 4) node{{\footnotesize $i$}};
		\draw (7, -.5) node{{\footnotesize $j$}};
		\draw[ ->,thick,   gray] (0,0) -- (0, 5.5); 
		\draw (-.5, -.2) node{{\footnotesize $(1,1)$}};
		\draw (9.6, 5.3) node{{\footnotesize $(n,N)$}};
		\clip (0,0) -- (9.5,0) -- (9.5,5.5) -- (0, 5.5) -- (0,0);
		\draw[dotted, gray] (0,0) grid (10,6);
		
		\draw[ultra thick] (0,0) -- (1,0) -- (1,1) -- (1,2) -- (2,2) -- (2,3) -- (2,4) -- (3,4) -- (4,4) -- (5,4) -- (6,4) -- (7,4) -- (8,4) -- (8,5) -- (9,5);
		\end{tikzpicture}
			\caption{$\ci=\{(i,j)\mid 1\leq i\leq 6,1\leq j\leq 10\}$, $(n,N)=(6,10)$.} 
 \end{subfigure} 
 \hfill
 \begin{subfigure}{.5\textwidth}
     \centering\begin{tikzpicture}[scale=0.6]
     
		\draw[->, thick, gray] (0,0) -- (12, 0);
		\draw (-.5, 4) node{{\footnotesize $i$}};
		\draw (7, -.5) node{{\footnotesize $j$}};
		\draw[->,thick,   gray] (0,0) -- (0, 5.5); 
		\draw[thick,   gray] (0,0) -- (5, 5); 
		\draw (-.5, -.2) node{{\footnotesize $(1,1)$}};
		\draw (7.9, 4.6) node{{\footnotesize $(n,N)$}};
		
		 \foreach \n in {0,1,2,3,4,5}\filldraw [black] (11-\n,\n) circle(0.07);
		\clip (0,0) -- (11,0) -- (5.5, 5.5) -- (0,0);
		\draw[dotted, gray] (0,0) grid (11,6);
		
		\draw[ultra thick] (0,0) -- (1,0) -- (1,1) --(3,1) -- (3,3) -- (5,3) -- (5,4) -- (7,4);
		\end{tikzpicture}
		
	\caption{$\ci=\{(i,j)\mid 1\leq i\leq 6, i\leq j\leq 13-i\}$, $(n,N)=(5,8)$. } 
	\label{Fig1B}
	\end{subfigure}
	\caption{(A). A full-space log-gamma polymer; (B). A half-space log-gamma polymer  (paths are restricted to an octant of the $\mathbb Z^2$ lattice). Thick points in (B) are the endpoints of paths that appear in the sum in the definition of point-to-line partition functions \eqref{defZflat}. In both figures, the thick line represents an up-right path. We have drawn the point $(n,N)$ with $n$ being the vertical coordinate and $N$ being the horizontal one in order to be consistent with usual matrix conventions below.}
	\label{fullspaceloggamma}
 \end{figure}
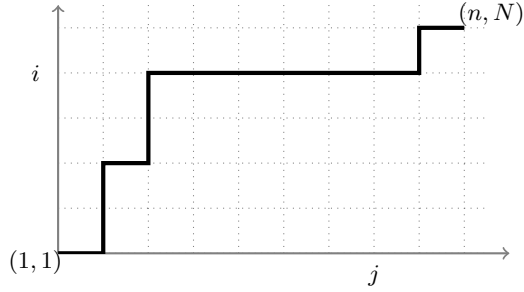
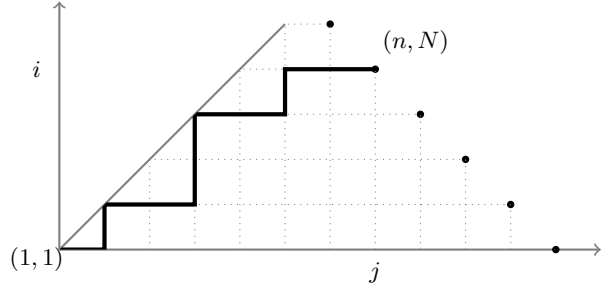

  \subsection{An identity in distribution} 
  Our main result is an identity in distribution between the full-space point-to-point log-gamma polymer partition function and the trapezoidal point-to-line log-gamma polymer partition function. The identity   holds when weights depend on the same two families of inhomogeneity parameters. We make this parametrization explicit now. 
  
 \begin{definition}\label{nota} 
Suppose that $n\geq 1,m\geq 0$ are two   integers  and that $\alpha_{\circ}$, $\bsb{\alpha}=(\alpha_1,\dots,\alpha_n)$ and  $\bsb{\beta}=(\beta_1,\dots,\beta_m)$ satisfy  $\alpha_{i}+\alpha_{\circ}>0$ for all $1\leq i\leq n$,
 $\alpha_{i}+\alpha_j>0$ for all $1\leq i, j\leq n$ and $\alpha_i+\beta_k>0$ for all $1\leq i\leq n, 1\leq k\leq m$.

 Let $Z(n,n+m+1)$ be  the full-space  point-to-point log-gamma polymer partition function, defined by \eqref{defZ} with  the following choice of parameters    (see also Figure \ref{fullspaceparam})
\begin{equation} \label{paramFull}
 W_{i,j}\sim\left\{\begin{array}{ll}
 \operatorname{Gamma}^{-1}(\alpha_{i}+\alpha_{\circ}) ,&  j=1, \\
  \operatorname{Gamma}^{-1}(\alpha_{i}+\alpha_{j-1}) ,& 2\leq j \leq n+1, \\
\operatorname{Gamma}^{-1}\left(\alpha_{i}+\beta_{j-n-1}\right) ,& n+2\leq j\leq n+m+1 .
\end{array}\right. 
\end{equation}
 We recall that the random variables $ W_{i,j}$ are assumed to be independent.

Let $Z^{\rflat}(n;m)$ be the trapezoidal point-to-line partition function  of the log-gamma polymer defined by \eqref{defZflat} with   the following parametrization  (see  Figure \ref{halfspaceparam})
\begin{equation}\label{paramHal}
 W_{i, j} 
\sim\left\{\begin{array}{ll}
\operatorname{Gamma}^{-1}\left(\alpha_{i}+\alpha_{\circ}\right), & 1 \leq i=j \leq n, \\
\operatorname{Gamma}^{-1}\left(\alpha_{i}+\alpha_{j}\right), & 1 \leq i<j \leq n, \\
\operatorname{Gamma}^{-1}\left(\alpha_i+\beta_{j-n}\right), & 1\leq i\leq n, n<j\leq n+m, \\
\operatorname{Gamma}^{-1}\left(\alpha_{i}+\alpha_{2 n+m-j+1}\right), & 1 \leq i \leq n, n+m<j \leq 2 n+m-i+1.
\end{array}\right.
\end{equation} 
Again,   the random variables $ W_{i, j}$ are assumed to be independent. 

\end{definition} 
 
 \begin{figure}
 \centering
 \begin{subfigure}{.44\textwidth}
     \centering\begin{tikzpicture}[scale=0.6]
		\draw (-.5, -.4) node{{\footnotesize $(1,1)$}}; 
		\draw (10.8, 5.8) node{{\footnotesize $(n,n+m+1)$}};
		\draw[->] (9.2,    3.1) node[anchor=south]{{ $   \alpha_i + \beta_k$}} to[bend left] (8,2); 
		\draw[->] (3.2,   3.1) node[anchor=south]{{ $   \alpha_i + \alpha_j$}} to[bend left] (2,2); 
		\draw[->] (-.6,    3.2) node[anchor=south]{{ $   \alpha_i + \alpha_{\circ}$}} to[bend right] (0,2);
		\clip (-0.5,-0.5) -- (9.5,-0.5) -- (9.5,5.5) -- (-0.5, 5.5) -- (-0.5,-0.5);
		\draw[dotted, gray] (0,0) grid (10,6);
		
		
		
 \coordinate (1) at (0,0); 
 \coordinate (2) at (0,1);
 \coordinate (3) at (0,2);
 \coordinate (4) at (0,3);
 \coordinate (5) at (0,4);
 \coordinate (6) at (0,5);
 \foreach \n in {1,2,3,4,5,6} \fill [green] (\n)
   circle (2pt) node [below] {}; 
   \coordinate (1) at (1,0);
 \coordinate (2) at (1,1);
 \coordinate (3) at (1,2);
 \coordinate (4) at (1,3);
 \coordinate (5) at (1,4);
 \coordinate (6) at (1,5);
 \foreach \n in {1,2,3,4,5,6} \fill [red] (\n)
   circle (2pt) node [below] {}; 
    \coordinate (1) at (6,0);
 \coordinate (2) at (6,1);
 \coordinate (3) at (6,2);
 \coordinate (4) at (6,3);
 \coordinate (5) at (6,4);
 \coordinate (6) at (6,5);
 \foreach \n in {1,2,3,4,5,6} \fill [red] (\n)
   circle (2pt) node [below] {}; 

    \coordinate (1) at (2,0);
 \coordinate (2) at (2,1);
 \coordinate (3) at (2,2);
 \coordinate (4) at (2,3);
 \coordinate (5) at (2,4);
 \coordinate (6) at (2,5);
 \foreach \n in {1,2,3,4,5,6} \fill [red] (\n)
   circle (2pt) node [below] {};

    \coordinate (1) at (3,0);
 \coordinate (2) at (3,1);
 \coordinate (3) at (3,2);
 \coordinate (4) at (3,3);
 \coordinate (5) at (3,4);
 \coordinate (6) at (3,5);
 \foreach \n in {1,2,3,4,5,6} \fill [red] (\n)
   circle (2pt) node [below] {};

    \coordinate (1) at (4,0);
 \coordinate (2) at (4,1);
 \coordinate (3) at (4,2);
 \coordinate (4) at (4,3);
 \coordinate (5) at (4,4);
 \coordinate (6) at (4,5);
 \foreach \n in {1,2,3,4,5,6} \fill [red] (\n)
   circle (2pt) node [below] {}; 
   
    \coordinate (1) at (5,0);
 \coordinate (2) at (5,1);
 \coordinate (3) at (5,2);
 \coordinate (4) at (5,3);
 \coordinate (5) at (5,4);
 \coordinate (6) at (5,5);
 \foreach \n in {1,2,3,4,5,6} \fill [red] (\n)
   circle (2pt) node [below] {}; 
   
    \coordinate (1) at (7,0);
 \coordinate (2) at (7,1);
 \coordinate (3) at (7,2);
 \coordinate (4) at (7,3);
 \coordinate (5) at (7,4);
 \coordinate (6) at (7,5);
 \foreach \n in {1,2,3,4,5,6} \fill [blue] (\n)
   circle (2pt) node [below] {}; 
   
    \coordinate (1) at (8,0);
 \coordinate (2) at (8,1);
 \coordinate (3) at (8,2);
 \coordinate (4) at (8,3);
 \coordinate (5) at (8,4);
 \coordinate (6) at (8,5);
 \foreach \n in {1,2,3,4,5,6} \fill [blue] (\n)
   circle (2pt) node [below] {};
   
    \coordinate (1) at (9,0);
 \coordinate (2) at (9,1);
 \coordinate (3) at (9,2);
 \coordinate (4) at (9,3);
 \coordinate (5) at (9,4);
 \coordinate (6) at (9,5);
 \foreach \n in {1,2,3,4,5,6} \fill [blue] (\n)
   circle (2pt) node [below] {};  
		\end{tikzpicture}
			\caption{}
	\label{fullspaceparam}
 \end{subfigure} 
 \begin{subfigure}{.55\textwidth}
     \centering\begin{tikzpicture}[scale=0.6]
		\draw[->] (11.8,    3.1) node[anchor=south]{{ $   \alpha_i + \alpha_{j}$}} to[bend left] (11,2); 
		\draw[->] (10.2,    4.6) node[anchor=south]{{ $   \alpha_i + \beta_k$}} to[bend left] (7,2); 
		\draw[->] (4.3,    3) node[anchor=south]{{ $   \alpha_i + \alpha_j$}} to[bend left] (3,2); 
		\draw[->] (.6,    2.6) node[anchor=south]{{ $   \alpha_i + \alpha_{\circ}$}} to[bend right] (2,2);

		
		\fill [white] (0,-0.3) circle (2pt) node [below] {};
 \coordinate (1) at (0,0);
 \coordinate (2) at (1,1);
 \coordinate (3) at (2,2);
 \coordinate (4) at (3,3);
 \coordinate (5) at (4,4);
 \coordinate (6) at (5,5);
 \foreach \n in {1,2,3,4,5,6} \fill [green] (\n)
   circle (2pt) node [below] {};

   \coordinate (1) at (1,0); 
 \foreach \n in {1} \fill [red] (\n)
   circle (2pt) node [below] {};

    \coordinate (1) at (2,0);
 \coordinate (2) at (2,1); 
 \foreach \n in {1,2} \fill [red] (\n)
   circle (2pt) node [below] {};

    \coordinate (1) at (3,0);
 \coordinate (2) at (3,1);
 \coordinate (3) at (3,2); 
 \foreach \n in {1,2,3} \fill [red] (\n)
   circle (2pt) node [below] {};

    \coordinate (1) at (4,0);
 \coordinate (2) at (4,1);
 \coordinate (3) at (4,2);
 \coordinate (4) at (4,3); 
 \foreach \n in {1,2,3,4} \fill [red] (\n)
   circle (2pt) node [below] {}; 
   
    \coordinate (1) at (5,0);
 \coordinate (2) at (5,1);
 \coordinate (3) at (5,2);
 \coordinate (4) at (5,3);
 \coordinate (5) at (5,4); 
 \foreach \n in {1,2,3,4,5} \fill [red] (\n)
   circle (2pt) node [below] {};

    \coordinate (1) at (6,0);
 \coordinate (2) at (6,1);
 \coordinate (3) at (6,2);
 \coordinate (4) at (6,3);
 \coordinate (5) at (6,4); 
 \coordinate (6) at (6,5); 
 \foreach \n in {1,2,3,4,5,6} \fill [blue] (\n)
   circle (2pt) node [below] {}; 
    \coordinate (1) at (7,0);
 \coordinate (2) at (7,1);
 \coordinate (3) at (7,2);
 \coordinate (4) at (7,3);
 \coordinate (5) at (7,4);
 \coordinate (6) at (7,5);
 \foreach \n in {1,2,3,4,5,6} \fill [blue] (\n)
   circle (2pt) node [below] {}; 
   
    \coordinate (1) at (8,0);
 \coordinate (2) at (8,1);
 \coordinate (3) at (8,2);
 \coordinate (4) at (8,3);
 \coordinate (5) at (8,4);
 \coordinate (6) at (8,5);
 \foreach \n in {1,2,3,4,5,6} \fill [blue] (\n)
   circle (2pt) node [below] {};
   
    \coordinate (1) at (9,0);
 \coordinate (2) at (9,1);
 \coordinate (3) at (9,2);
 \coordinate (4) at (9,3);
 \coordinate (5) at (9,4);
 \coordinate (6) at (9,5);
 \foreach \n in {1,2,3,4,5,6} \fill [red] (\n)
   circle (2pt) node [below] {}; 
   
    \coordinate (1) at (10,0);
 \coordinate (2) at (10,1);
 \coordinate (3) at (10,2);
 \coordinate (4) at (10,3);
 \coordinate (5) at (10,4); 
 \foreach \n in {1,2,3,4,5} \fill [red] (\n)
   circle (2pt) node [below] {}; 
   
    \coordinate (1) at (11,0);
 \coordinate (2) at (11,1);
 \coordinate (3) at (11,2);
 \coordinate (4) at (11,3); 
 \foreach \n in {1,2,3,4} \fill [red] (\n)
   circle (2pt) node [below] {};

    \coordinate (1) at (12,0);
 \coordinate (2) at (12,1);
 \coordinate (3) at (12,2); 
 \foreach \n in {1,2,3} \fill [red] (\n)
   circle (2pt) node [below] {};

    \coordinate (1) at (13,0);
 \coordinate (2) at (13,1); 
 \foreach \n in {1,2} \fill [red] (\n)
   circle (2pt) node [below] {};

    \coordinate (1) at (14,0); 
 \foreach \n in {1} \fill [red] (\n)
   circle (2pt) node [below] {}; 
   
   \clip (0,0)--(14,0)--(9,5)--(5,5)--(0,0);
		\draw[dotted, gray] (0,0) grid (14.2,5.3);
		
		\end{tikzpicture}
		
	\caption{}
	\label{halfspaceparam}
	\end{subfigure}
	\caption{(A) The parametrization in (\ref{paramFull}). Rows and columns are indexed from $1$ to $n$ and $n+m+1$, respectively. (B) The parametrization in (\ref{paramHal}). Rows and columns are indexed from $1$ to $n$ and $2n+m$, respectively. In both figures, $n=6,m=3$, and $i=3, j=4, k=2$.}
 \end{figure}
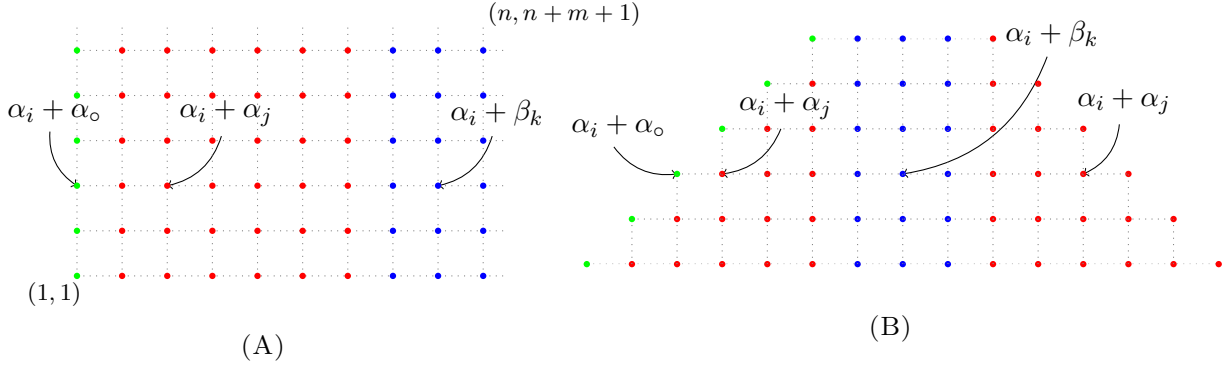 
  
\begin{theorem}\label{main} 
Assume that parameters $n,m,\alpha_{\circ },\bsb \alpha, \bsb \beta$ and the partition functions $Z(n,n+m+1)$, $ Z^{\rflat}(n;m)$ are   as in Definition \ref{nota}. 
   Then the following equality in distribution holds,
   \begin{equation}
      Z^{\rflat}(n;m)\overset{(d)}{=\joinrel =} Z(n,n+m+1).
      \label{eq:identityindistribution}
   \end{equation}
\end{theorem}
  
 \begin{remark} In the case $m=0$, we obtain an identity in law between the point-to-point partition function of a full-space log-gamma polymer and the point-to-line partition function of a the log-gamma polymer confined in an octant.
  \end{remark}  
  \begin{remark}
  There exist several other remarkable identities in distribution involving partition functions of the log-gamma polymer. 
  \begin{itemize} 
 \item  The so-called replica partition function \cite{bisi2020geometric} has the same distribution as the point-to-point partition function in an octant \cite{OSZ}. 
 \item Two point-to-point partition functions in an octant have the same distribution when parameters along the first row and the boundary are switched \cite{HalfSpaceMac}. 
 \item  Certain joint distributions of point-to-point partition functions satisfy a shift invariance property \cite{dauvergne2020hidden},  similar to identities discovered in  \cite{borodin2019shift} for other exactly solvable models. 
 \end{itemize}
  \end{remark}
In order to prove Theorem \ref{main}, we will show that the Laplace transforms of both sides of \eqref{eq:identityindistribution} are equal. As we will prove in the sequel, the Laplace transform of the trapezoidal point-to-line partition function is given, for $r>0$,  by
\begin{align}\nonumber
\mathbb{E}\left[e^{-rZ^{\rflat}(n;m)}\right]= \frac{r^{\sum_{i=1}^{n}\alpha_{i}}}{ \prod_{ i=1}^n \Gamma(\alpha_i+\alpha_\circ) \prod_{ i=1}^n\prod_{j=1}^{n} \Gamma\left(\alpha_{i}+\alpha_{j} \right)\prod_{ i=1}^n\prod_{k=1}^m\Gamma(\alpha_i+\beta_k)} & \\
\times \int_{\mathbb{R}_{>0}^{n}} \mathcal{T}_{\alpha_{\circ},\bsb{\beta};r}(\bsb{x}) &
\Psi_{\boldsymbol{\alpha}}^{\mathfrak{s} \mathfrak{o}_{2 n+1}}(\boldsymbol{x})  \prod_{i=1}^{n}  \frac{{d} x_{i}}{x_{i}},
    \label{eq:Laplacetrapezoidal}
\end{align}
where the orthogonal Whittaker functions  $\Psi^{\gso{2n+1}}_{\bsb \alpha}$  and the function  $\mathcal{T}_{\alpha_{\circ},\bsb{\beta};r}$ are defined in Section \ref{sec:whi}.
In the special case $m=0$ and $\alpha_{\circ}=0$, the formula \eqref{eq:Laplacetrapezoidal} reduces to \cite[Theorem 3.13]{BisiZygouras}.
In the general case, the proof is given in Section \ref{sec_proof} where, following the approach used in \cite{BisiZygouras},  we compute the joint distribution of the image of  $(W_{i,j})$ under the geometric RSK map. Then we insert this joint distribution into the definition of the Laplace transform of $Z^{\rflat}(n;m)$ to obatin \eqref{eq:Laplacetrapezoidal}.

\medskip 
Regarding the full-space point-to-point partition function, its  Laplace transform  was computed in \cite{COSZ, OSZ} as 
\begin{align}\nonumber
  \mathbb{E}\left[e^{-r Z(n,n+m+1)}\right]
  =\frac{1}{\prod_{i=1}^n\Gamma(\alpha_i+\alpha_\circ)\prod_{ i=1}^n\prod_{j=1}^n\Gamma(\alpha_i+\alpha_j)\prod_{ i=1}^n\prod_{k=1}^m\Gamma(\alpha_i+\beta_k)} &\\
  \times \int_{\mathbb R_{>0}^n} e^{-rx_1} \Psi^{\ggl n}_{\alpha_{\circ}\sqcup \bsb{\alpha}\sqcup \bsb{\beta};1}(\bsb x) & \Psi^{\ggl n}_{\bsb \alpha}(\bsb x)\prod_{i=1}^n \frac{dx_i}{x_i},
    \label{eq:Laplacepointtopoint}
\end{align}
where $\alpha_{\circ}\sqcup \bsb{\alpha}\sqcup \bsb{\beta}$ denotes the concatenation of the three arrays of parameters and the function  $\Psi^{\ggl n}_{\alpha_{\circ}\sqcup \bsb{\alpha}\sqcup \bsb{\beta};1} $, defined in \eqref{eq7.5},  is a generalization of $\ggl n$-Whittaker functions.
At first sight, it is far from obvious that \eqref{eq:Laplacetrapezoidal} and \eqref{eq:Laplacepointtopoint} are equal. It turns out that  the Plancherel theorem   associated with $\ggl{n}$-Whittaker functions (see Section \ref{sec:whi}) can be applied to  both  formulas (see Corollary  \ref{LapTransFla} and Proposition \ref{lapTrans}),
and one thus transforms \eqref{eq:Laplacetrapezoidal} and \eqref{eq:Laplacepointtopoint} into the following contour integral formula (see Corollary \ref{LapTransFla} for details)
\begin{align*}
     \mathbb{E}&\left[e^{-r Z^{\rflat}(n;m)}\right]=\mathbb{E}\left[e^{-r Z(n,n+m+1)}\right]
      =\\
  &\int_{(\mu+\mathrm{i} \mathbb{R})^{n}} r^{\sum_{i=1}^{n}\left(\alpha_{i}-\lambda_{i}\right)} \prod_{1 \leq i, j \leq n} \Gamma\left(\lambda_{i}-\alpha_{j} \right) 
  \prod_{i=1}^{n}\frac{\Gamma\left(\lambda_{i}+{\alpha}_{\circ} \right)}{\Gamma\left(\alpha_{i}+{\alpha}_{\circ}\right)} \prod_{ i=1}^n\prod_{j=1}^{n} \frac{\Gamma\left(\lambda_{i}+{\alpha}_{j} \right)}{\Gamma\left(\alpha_{i}+{\alpha}_{j}\right)}  \prod_{ i=1}^n\prod_{k=1}^m\frac{\Gamma(\lambda_i+\beta_k)}{\Gamma(\alpha_i+\beta_k)}
   s_{n}(\bsb{\lambda}) d \bsb{\lambda}.
\end{align*} 
However, the Plancherel theorem can only be applied  when $m$ is sufficiently large, otherwise it is not clear how to write the integrand in \eqref{eq:Laplacetrapezoidal} as a product of  two functions in $L^2(\mathbb{R}_{>0}^n)$, 
as required to apply the Plancherel theorem. In particular, the above argument does  not work for the point-to-line partition function of polymers confined in an octant, since in this case $m=0$. This is the reason why we have introduced in Definition \ref{nota} the trapezoidal domains with arbitrary $m$. We will show that for large enough but fixed  $m$, the integrand in \eqref{eq:Laplacetrapezoidal} is the product of two functions in $L^2(\mathbb{R}_{>0}^n)$. Then, once the identity in distribution
\eqref{eq:identityindistribution} is established for large enough $m$, a probabilistic argument allows to prove the result for arbitrary $m\geqslant 0$. This probabilistic argument is given in Section \ref{sec4}.

\begin{remark}
The equality between \eqref{eq:Laplacetrapezoidal} and  \eqref{eq:Laplacepointtopoint} (which holds for any $m\geqslant 0$) is a sign of -- seemingly unknown -- properties of Whittaker functions. For $m=0$, the identity reads
\begin{equation}
r^{\sum_{i=1}^{n}\alpha_{i}} \int_{\mathbb{R}_{>0}^{n}} e^{-r/x_n} \left(\prod_{i=1}^n(x_i/r)^{(-1)^{i}}\right)^{\alpha_{\circ}} 
\Psi_{\boldsymbol{\alpha}}^{\mathfrak{s} \mathfrak{o}_{2 n+1}}(\boldsymbol{x})  \prod_{i=1}^{n}  \tfrac{{d} x_{i}}{x_{i}} = \int_{\mathbb R_{>0}^n} e^{-rx_1} \Psi^{\ggl n}_{\alpha_{\circ}\sqcup \bsb{\alpha};1}(\bsb x) \Psi^{\ggl n}_{\bsb \alpha}(\bsb x)\prod_{i=1}^n \tfrac{dx_i}{x_i}.
\label{eq:identityWhittaker}
\end{equation}
An analogous identity holds in the zero temperature limit, where one considers geometric last passage percolation instead of the log-gamma polymer. This identity was obtained in \cite[(7.59), (7.60)]{BaikRains} 
(see also \cite[Section 4]{BisiZygouras} and \cite{bisi2016goe}). Integrals in \eqref{eq:Laplacetrapezoidal} and \eqref{eq:Laplacepointtopoint}  become  sums over integer partitions in \cite{BaikRains} and Whittaker functions become Schur functions. Another combinatorial proof of (a special case) of this identity at the Schur function level  was proposed  in \cite{bisi2019transition}, using generalizations of bounded Littlewood identities from \cite{stembridge1990nonintersecting, macdonald1995symmetric, okada1998applications, krattenthaler1998identities}. Bounded Littlewood identities also appear in \cite{betea2018combinatorics} in a very similar context. It would be interesting to investigate whether the properties of Schur functions used in \cite{betea2018combinatorics, bisi2019transition} may be generalized to  Macdonald polynomials and Whittaker functions (see \cite{rains2021bounded} for partial answers).
\end{remark}

 \subsection{Asymptotics}\label{intro:asymp} Many asymptotic results have been obtained for partition functions of full-space directed polymer models \cite{amir2011probability, sasamoto2010exact, calabrese2010free, dotsenko2010replica, MacdoProcesses, borodin2014free, borodin2015height, imamura2017free, AKQ14}, and in particular for the log-gamma polymer \cite{BCR13,KQ18,BCD}. Yet not many rigorous asymptotic results exist for models in half-space. Theorem \ref{main} bridges asymptotics in the full-space case to the half-space case (and more generally, the trapezoidal case). Before stating our asymptotic result, let us fix some notations. 
  \begin{definition}\label{notation2}
  Suppose   $\theta,\delta>0$ are    fixed constants. Suppose $n$ and $m\geq n$ are integers that tend to infinity in such a way that $\frac{n}{m}>\delta$ holds. We denote $\frac{m}{n}$ by $p$ in the sequel.  We assume that $n$ is the variable that tends to infinity and $m,p$ are merely  functions of $n$. 
  Let $\theta_0>0$ be a  positive real number which may depend on $n$. 
  
For each $n$, consider   the full space log-gamma polymer   defined via (see Figure \ref{img_paramGUE})
\begin{equation}\label{paramGUE}
    W_{i,j}^{-1}\sim\operatorname{Gamma}(\theta_0) \text{ for } j=1,\quad W_{i,j}^{-1}\sim\operatorname{Gamma}(\theta) \text{ for } j\geq 2,
\end{equation}
on the domain  $\{(i,j)\mid 1\leq i\leq n,\,1\leq j\leq m\}$. We denote its point-to-point partition function by $Z(n,m)$ (as in \eqref{defZ}). 
Consider also the  trapezoidal  log-gamma polymer defined via  (see Figure \ref{img_paramTrapGUE})
\begin{equation} \label{paramTrapGUE}
         W_{i,j}^{-1}\sim\operatorname{Gamma}(\theta_0 )\,\text{ if } \, 1\leq i=j\leq n,   \quad
           W_{i,j}^{-1}\sim\operatorname{Gamma}(\theta ) \,  \text{ otherwise}, 
 \end{equation}
 on the domain $\{(i,j)\mid 1\leq i\leq n,\,i\leq j\leq n+m-i\}$.  We denote the point-to-line partition function of this trapezoidal log-gamma polymer by $Z^{\rflat}(n;m-n-1)$  (as in \eqref{defZflat}). Note that $m$ in this definition corresponds to $n+m+1$ in Definitions \ref{nota} and \ref{def:ptl}. We make this  change of notation  for the simplicity of discussions about Theorem \ref{main2} and it only concerns Section  \ref{intro:asymp} and Section \ref{sec:asymptotics}. 
 \end{definition}

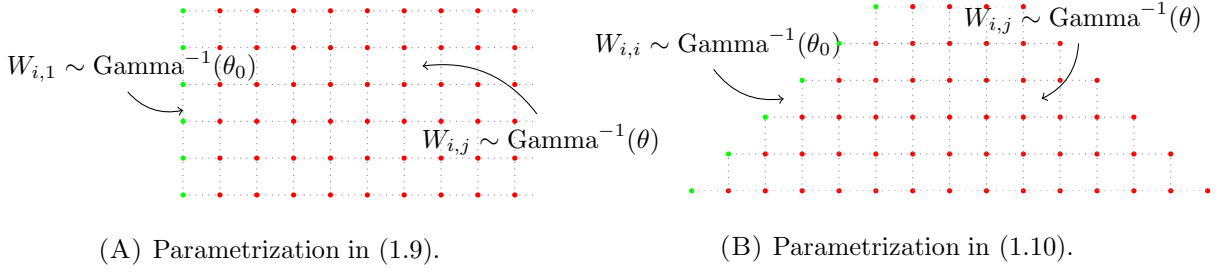
\begin{figure}
\centering
\begin{subfigure}{.45\textwidth}
\begin{center}
\begin{tikzpicture}[scale=0.49]

		\draw[->] (9.6,    2.2) node[anchor=north]{\small{ $ W_{i,j}\sim \operatorname{Gamma}^{-1}(  \theta)$}} to[bend right] (6.5,3.5); 
 
		\draw[->] (-1.5,    2.8) node[anchor=south]{\small{ $ W_{i,1}\sim \operatorname{Gamma}^{-1}(  \theta_0)$}} to[bend right] (0,2.3);

		\clip (-0.5,-0.5) -- (9.5,-0.5) -- (9.5,5.5) -- (-0.5, 5.5) -- (-0.5,-0.5);
		\draw[dotted, gray] (0,0) grid (10,5);
	
 \coordinate (1) at (0,0); 
 \coordinate (2) at (0,1);
 \coordinate (3) at (0,2);
 \coordinate (4) at (0,3);
 \coordinate (5) at (0,4);
 \coordinate (6) at (0,5);
 \foreach \n in {1,2,3,4,5,6} \fill [green] (\n)
   circle (2pt) node [below] {}; 
   \coordinate (1) at (1,0);
 \coordinate (2) at (1,1);
 \coordinate (3) at (1,2);
 \coordinate (4) at (1,3);
 \coordinate (5) at (1,4);
 \coordinate (6) at (1,5);
 \foreach \n in {1,2,3,4,5,6} \fill [red] (\n)
   circle (2pt) node [below] {}; 
    \coordinate (1) at (6,0);
 \coordinate (2) at (6,1);
 \coordinate (3) at (6,2);
 \coordinate (4) at (6,3);
 \coordinate (5) at (6,4);
 \coordinate (6) at (6,5);
 \foreach \n in {1,2,3,4,5,6} \fill [red] (\n)
   circle (2pt) node [below] {}; 

    \coordinate (1) at (2,0);
 \coordinate (2) at (2,1);
 \coordinate (3) at (2,2);
 \coordinate (4) at (2,3);
 \coordinate (5) at (2,4);
 \coordinate (6) at (2,5);
 \foreach \n in {1,2,3,4,5,6} \fill [red] (\n)
   circle (2pt) node [below] {};

    \coordinate (1) at (3,0);
 \coordinate (2) at (3,1);
 \coordinate (3) at (3,2);
 \coordinate (4) at (3,3);
 \coordinate (5) at (3,4);
 \coordinate (6) at (3,5);
 \foreach \n in {1,2,3,4,5,6} \fill [red] (\n)
   circle (2pt) node [below] {};

    \coordinate (1) at (4,0);
 \coordinate (2) at (4,1);
 \coordinate (3) at (4,2);
 \coordinate (4) at (4,3);
 \coordinate (5) at (4,4);
 \coordinate (6) at (4,5);
 \foreach \n in {1,2,3,4,5,6} \fill [red] (\n)
   circle (2pt) node [below] {}; 
   
    \coordinate (1) at (5,0);
 \coordinate (2) at (5,1);
 \coordinate (3) at (5,2);
 \coordinate (4) at (5,3);
 \coordinate (5) at (5,4);
 \coordinate (6) at (5,5);
 \foreach \n in {1,2,3,4,5,6} \fill [red] (\n)
   circle (2pt) node [below] {}; 
   
    \coordinate (1) at (7,0);
 \coordinate (2) at (7,1);
 \coordinate (3) at (7,2);
 \coordinate (4) at (7,3);
 \coordinate (5) at (7,4);
 \coordinate (6) at (7,5);
 \foreach \n in {1,2,3,4,5,6} \fill [red] (\n)
   circle (2pt) node [below] {}; 
   
    \coordinate (1) at (8,0);
 \coordinate (2) at (8,1);
 \coordinate (3) at (8,2);
 \coordinate (4) at (8,3);
 \coordinate (5) at (8,4);
 \coordinate (6) at (8,5);
 \foreach \n in {1,2,3,4,5,6} \fill [red] (\n)
   circle (2pt) node [below] {};
   
    \coordinate (1) at (9,0);
 \coordinate (2) at (9,1);
 \coordinate (3) at (9,2);
 \coordinate (4) at (9,3);
 \coordinate (5) at (9,4);
 \coordinate (6) at (9,5);
 \foreach \n in {1,2,3,4,5,6} \fill [red] (\n)
   circle (2pt) node [below] {};  
		\end{tikzpicture} 
\end{center}
		\caption{Parametrization in (\ref{paramGUE}).}
		\label{img_paramGUE}
\end{subfigure}
\begin{subfigure}{.54\textwidth}
\begin{center} 

    \begin{tikzpicture}[scale=0.49]
		 \draw[->] (10.5,    4) node[anchor=south]{\small{ $W_{i,j}\sim \operatorname{Gamma}^{-1}(\theta)$}} to[bend left] (9.5,2.5);  
		 
		\draw[->] (.6,    3.3) node[anchor=south]{\small{ $ W_{i,i}\sim \operatorname{Gamma}^{-1}( \theta_0)$}} to[bend right] (2.5,2.5);

	
 \coordinate (1) at (0,0);
 \coordinate (2) at (1,1);
 \coordinate (3) at (2,2);
 \coordinate (4) at (3,3);
 \coordinate (5) at (4,4);
 \coordinate (6) at (5,5);
 \foreach \n in {1,2,3,4,5,6} \fill [green] (\n)
   circle (2pt) node [below] {}; 
   
   \coordinate (1) at (1,0); 
 \foreach \n in {1} \fill [red] (\n)
   circle (2pt) node [below] {};

    \coordinate (1) at (2,0);
 \coordinate (2) at (2,1); 
 \foreach \n in {1,2} \fill [red] (\n)
   circle (2pt) node [below] {};

    \coordinate (1) at (3,0);
 \coordinate (2) at (3,1);
 \coordinate (3) at (3,2); 
 \foreach \n in {1,2,3} \fill [red] (\n)
   circle (2pt) node [below] {};

    \coordinate (1) at (4,0);
 \coordinate (2) at (4,1);
 \coordinate (3) at (4,2);
 \coordinate (4) at (4,3); 
 \foreach \n in {1,2,3,4} \fill [red] (\n)
   circle (2pt) node [below] {}; 
   
    \coordinate (1) at (5,0);
 \coordinate (2) at (5,1);
 \coordinate (3) at (5,2);
 \coordinate (4) at (5,3);
 \coordinate (5) at (5,4); 
 \foreach \n in {1,2,3,4,5} \fill [red] (\n)
   circle (2pt) node [below] {};

    \coordinate (1) at (6,0);
 \coordinate (2) at (6,1);
 \coordinate (3) at (6,2);
 \coordinate (4) at (6,3);
 \coordinate (5) at (6,4); 
 \coordinate (6) at (6,5); 
 \foreach \n in {1,2,3,4,5,6} \fill [red] (\n)
   circle (2pt) node [below] {}; 
    \coordinate (1) at (7,0);
 \coordinate (2) at (7,1);
 \coordinate (3) at (7,2);
 \coordinate (4) at (7,3);
 \coordinate (5) at (7,4);
 \coordinate (6) at (7,5);
 \foreach \n in {1,2,3,4,5,6} \fill [red] (\n)
   circle (2pt) node [below] {}; 
   
    \coordinate (1) at (8,0);
 \coordinate (2) at (8,1);
 \coordinate (3) at (8,2);
 \coordinate (4) at (8,3);
 \coordinate (5) at (8,4);
 \coordinate (6) at (8,5);
 \foreach \n in {1,2,3,4,5,6} \fill [red] (\n)
   circle (2pt) node [below] {};
   
    \coordinate (1) at (9,0);
 \coordinate (2) at (9,1);
 \coordinate (3) at (9,2);
 \coordinate (4) at (9,3);
 \coordinate (5) at (9,4);
 \coordinate (6) at (9,5);
 \foreach \n in {1,2,3,4,5,6} \fill [red] (\n)
   circle (2pt) node [below] {}; 
   
    \coordinate (1) at (10,0);
 \coordinate (2) at (10,1);
 \coordinate (3) at (10,2);
 \coordinate (4) at (10,3);
 \coordinate (5) at (10,4); 
 \foreach \n in {1,2,3,4,5} \fill [red] (\n)
   circle (2pt) node [below] {}; 
   
    \coordinate (1) at (11,0);
 \coordinate (2) at (11,1);
 \coordinate (3) at (11,2);
 \coordinate (4) at (11,3); 
 \foreach \n in {1,2,3,4} \fill [red] (\n)
   circle (2pt) node [below] {};

    \coordinate (1) at (12,0);
 \coordinate (2) at (12,1);
 \coordinate (3) at (12,2); 
 \foreach \n in {1,2,3} \fill [red] (\n)
   circle (2pt) node [below] {};

    \coordinate (1) at (13,0);
 \coordinate (2) at (13,1); 
 \foreach \n in {1,2} \fill [red] (\n)
   circle (2pt) node [below] {};

    \coordinate (1) at (14,0); \foreach \n in {1} \fill [red] (\n) circle (2pt) node [below] {}; 
   
   \clip (0,0)--(14,0)--(9,5)--(5,5)--(0,0);
		\draw[dotted, gray] (0,0) grid (14.2,5.3);
		
		\end{tikzpicture}
      \end{center}	
     \caption{Parametrization in (\ref{paramTrapGUE}).}
	\label{img_paramTrapGUE}
\end{subfigure}
\caption{Log-gamma polymers with homogeneous weights in the bulk and perturbations on the boundary.}
\label{fig}
\end{figure}

Now we need to introduce certain quantities appearing in the limit theorem below. Let $\psi$ denote the digamma function defined on $\mathbb R_{>0}$ by 
\[\psi(z)=\partial_z \log \Gamma(z) = -\gamma+\sum_{n=0}^{\infty}\left(\frac{1}{n+1}-\frac{1}{n+z}\right). \]
Here $\gamma$ is the Euler–Mascheroni constant.
Its derivative $\psi^{\prime}(z)=\sum_{n=0}^{\infty}\frac{1}{(n+z)^2}$ is a strictly decreasing function.
\begin{definition}\label{notation3}
Fix $\theta>0$.
Define  $\theta_c$ to be the unique solution to $\psi^{\prime}(\theta_c)-p\psi^{\prime } (\theta-\theta_c)=0$. It is well-defined since $x\mapsto\psi^{\prime}(x)-p\psi^{\prime}(\theta-x)$ is a bijection from $(0,\theta)$ to $\mathbb{R}$. 
Define
\begin{equation*} 
f_{\theta,p}:= -\psi(\theta_c)-p\psi(\theta-\theta_c),
\,\, \text{ and }
\bar f_{\theta,p}:= -\psi(\theta_0)-p\psi(\theta-\theta_0).
\end{equation*}  
Define also 
\begin{equation*} 
    \sigma_{\theta,p}:=\left(
       \frac{  - \psi^{\prime\prime} (\theta_c)-p\psi^{\prime \prime}(\theta-\theta_c) }{2}\right)^{\frac{1}{3}}. 
\end{equation*}
Since $p$ depends on $n$, the quantities  $\theta_c,f_{\theta,p},\bar f_{\theta,p}, \sigma_{\theta,p}$ all depend on $n$ as well.
\end{definition}

\begin{theorem}\label{main2}
Assume that parameters $n,m,p,\theta,\theta_0$ and the partition function $Z^{\rflat}(n;m-n-1)$  are as in   Definition \ref{notation2}. 
\begin{enumerate}
    \item  If   $\liminf_{n\rightarrow\infty} n^{1/3}(\theta_0-\theta_c)=+\infty $, then for all $t$,
    \begin{equation}
       \lim_{n\rightarrow\infty} \mathbb{P}\left( \frac{\log Z^{\rflat}(n;m-n-1)-nf_{\theta,p}  }{  n^{1 / 3} \sigma_{\theta,p}} \leq t\right)= F_{\mathrm{GUE}}(t),
       \label{eq:convergenceGUE}
    \end{equation} 
    with $F_{\mathrm{GUE}}$ being the GUE Tracy-Widom distribution \cite{TW94},  defined in \eqref{eq:defGUE}. 
    \item  If for some $y\in\mathbb{R}$, $\lim_{n\rightarrow\infty} n^{1/3}(\theta_0-\theta_c) \sigma_{\theta,p}=y   $, then for all $t$,
    \begin{equation*}
       \lim_{n\rightarrow\infty} \mathbb{P}\left( \frac{\log Z^{\rflat}(n;m-n-1)-nf_{\theta,p}}{  
       n^{1 / 3} \sigma_{\theta,p}}  \leq t\right)= F_{\mathrm{BBP} ; -y}(t),  
    \end{equation*} 
    with $F_{\mathrm{BBP} ; -y}$ being  the Baik-Ben Arous-P\'ech\'e  distribution \cite{BBP05}, defined in \eqref{eq:defBBP}.
    \item   If for some $\alpha\in(2/3,1]$, the limit $\lim_{n\rightarrow\infty}m^{-\alpha}\left({n\psi^{\prime}(\theta_0)-m\psi^{\prime}(\theta-\theta_0)}\right)$ exists and is positive,
    then for all $t$,
    \begin{equation}
       \lim_{n\rightarrow\infty} \mathbb{P}\left( \frac{\log Z^{\rflat}(n;m-n-1)- n\bar f_{\theta,p}}{ n^{1 / 2} \sqrt{\psi^{\prime} (\theta_0)-p\psi^{\prime }(\theta-\theta_0)}} \leq t\right)=\Phi(t),
       \label{eq:convergenceGaussian}
    \end{equation}
    where $\Phi$ is the standard Gaussian distribution function.
\end{enumerate}
\end{theorem} 
\begin{remark}
In part (3) of Theorem \ref{main2}, the hypothesis  in the case $\alpha\in(2/3,1)$ is equivalent to that   the limit $\lim_{n\rightarrow\infty}\frac{n}{m^\alpha}(\theta_0-\theta_c)$ exists and is negative.
\end{remark}
\begin{remark}
If $m$ and $n$ go to infinity in such a way that $p=\frac{m}{n}$ is fixed, then $\theta_c$ is also a fixed constant. Theorem \ref{main2} implies in particular that when $\theta_0$ is fixed with $\theta_0>\theta_c$, then \eqref{eq:convergenceGUE} holds, and when $\theta_0$ is fixed with  $\theta_0<\theta_c$, then  \eqref{eq:convergenceGaussian} holds.
\end{remark}
We prove Theorem \ref{main2} in Section \ref{sec:asymptotics}. The part (2) of Theorem \ref{main2} is a direct consequence of \cite[Theorem~1.7]{BCD}, stated below in Proposition \ref{asymptoticsGUE}.
The part (1)   is deduced from \cite[Theorem~1.2, Theorem~1.7]{BCD} via a coupling method.
 The part (3) is obtained using estimates from \cite{Seppa} about stationary and homogeneous log-gamma polymer models.
 
 \medskip 
 
  An explanation for why the point-to-point partition function undergoes a phase transition under a perturbation of one column is provided in \cite{BBP05} (in the zero temperature limit). Recall that an inverse-gamma variable has large mean value when its parameter is small. For small enough values of  $\theta_0$, paths stay a macroscopic distance along the first column with large probability under the polymer measure, while for larger values of $\theta_0$, paths would not be influenced by the different weights along the first column. We refer to \cite[Section 8.1]{HalfSpaceMac} and \cite[Section 1.5]{BCD} for a more detailed exposition of this argument in the context of the log-gamma polymer. In the half-space setting, a precise qualitative understanding of the measure on polymer paths, for various values of  $\theta_0$, is more difficult to obtain. Nevertheless,  we expect that for a fixed $\theta_0<\theta_c$, most paths under the polymer measure will hit the boundary a $\mathcal O(n)$ number of times, while when $\theta_0>\theta_c$, the geometry of paths should be qualitatively similar as in the full-space case.

 \subsection{KPZ equation}
 The Kardar-Parisi-Zhang (KPZ) equation \cite{KPZ} is the stochastic PDE 
 \begin{equation}
 \partial_t h(t,x) = \tfrac 1 2 \partial_{xx}h(t,x)+\tfrac 1 2 \left(\partial_x h(t,x)\right)^2 +\xi(t,x),\;\; x\in \mathbb R,\, t>0,
     \label{eq:KPZ}
 \end{equation}
 where $\xi(t,x)$ is a space-time Gaussian white noise. By definition, we say that $h$ is a  solution (in the Cole-Hopf sense) of \eqref{eq:KPZ} if $h(t,x)=\log \mathcal Z(t,x)$ where $\mathcal Z$ is a solution of the multiplicative noise stochastic heat equation (mSHE) 
 \begin{equation}
     \partial_t \mathcal Z(t,x)= \tfrac{1}{2} \partial_{x,x} \mathcal Z(t,x) +\mathcal Z(t,x)\xi(t,x),\;\; x\in \mathbb R, \, t>0, 
     \label{eq:SHE}
 \end{equation}
 which can be given a precise meaning using Ito calculus, see for example the review  \cite{corwin2012kardar}. We also consider the KPZ equation on the half-line $\mathbb R_{\geq 0}$. The solution depends on a boundary parameter which we will denote by $A\in \mathbb R$.  As is the case for the full-line KPZ equation, a solution of the KPZ equation on a half-line is defined as the logarithm of the mSHE on $\mathbb R_{\geq 0}$ with Robin type boundary condition
 \begin{equation}
     \begin{cases}
      \partial_t \mathcal Z_A(t,x)= \tfrac{1}{2} \partial_{x,x} \mathcal Z_A(t,x) +\mathcal Z_A(t,x)\xi(t,x),\;\; x\in \mathbb R_{\geq 0}, \, t>0, \\
      \partial_x \mathcal Z_A(t,x) \big\vert_{x=0} =A \mathcal Z_A(t,0). 
      \end{cases}
      \label{eq:halfspaceSHE}
 \end{equation}
Since the solution $\mathcal Z_A(t,x)$ is not differentiable in $x$,  the boundary condition $\partial_x \mathcal Z_A(t,x) \big\vert_{x=0} =A \mathcal Z_A(t,0)$ cannot be enforced on the solution itself. It must be enforced on the half-line heat kernel used to define mild solutions, we refer to \cite[Def. 2.5]{corwin2018open} for details. 

\medskip 
It is predicted in the physics paper \cite{BarraquandLeDoussal} that for all 
$t>0$, $X\geq 0$ and $A\in \mathbb R$, the identity in distribution 
\begin{equation}
    \mathcal Z_A(t,0)  = 2 \mathcal Z(t,-X)
    \label{eq:identityKPZ}
\end{equation}
holds, where,  in the LHS,  $\mathcal Z_A$ is the solution to \eqref{eq:halfspaceSHE} with initial data $\mathcal Z_A(0,x) = \mathbf{1}_{x\geqslant X}$, and in the RHS, $\mathcal Z$ is the solution to \eqref{eq:SHE} with initial data $\mathcal Z(0,x) = \mathbf{1}_{x\geq 0}\,e^{B_x-(A+\frac 1 2 )x}$, where $B_x$ is a standard Brownian motion independent of the white noise $\xi$. This prediction was based on the fact that moments of both sides of \eqref{eq:identityKPZ} match. However, this observation does not constitute a proof because moments of a solution  to the mSHE do not uniquely determine its distribution. 

\medskip 

We expect that \eqref{eq:identityKPZ} can be proved as a scaling limit of our main result \eqref{eq:identityindistribution}. Indeed, consider the log-gamma partition functions $Z(n,n+m+1)$ and  $Z^{\rflat}(n;m)$ with parameters chosen as in \eqref{paramGUE} and \eqref{paramTrapGUE} respectively. Let us use the scalings  \begin{equation}
    \theta=2\sqrt{N},\; \theta_0=\sqrt{N}+A+\tfrac 1 2,\; n=tN,\; m=X\sqrt{N}.
    \label{eq:scalings}
\end{equation}
Then we expect to find a sequence $C_N$ such that under the scalings \eqref{eq:scalings}, for all $t>0, X\geqslant 0$ and $A\in \mathbb R$, we have the weak convergences as $N\to+\infty$ 
\begin{align}
   C_N Z(n,n+m+1) &\Longrightarrow 2 \mathcal Z(t,-X),\label{eq:convergenceKPZfull}\\
    C_N Z^{\rflat}(n;m) &\Longrightarrow  \mathcal Z_A(t,0),\label{eq:convergenceKPZhalf}
\end{align}
where the initial conditions are the same as those required for  \eqref{eq:identityKPZ}.
The convergence of full-space directed polymer partition functions to the mSHE with delta initial data was first proved in \cite{AKQ14}. The proof was extended to cover the particular setting of the log-gamma polymer in \cite{corwin2017intermediate}. It seems however that the convergence stated in \eqref{eq:convergenceKPZfull} is not proved in the existing  literature, although the same initial condition was considered in \cite{parekh2019positive} in a slightly different setting. Regarding half-space directed  polymers, a general convergence result of the partition function of discrete polymer models to solutions of \eqref{eq:halfspaceSHE} was proved in \cite{Wu}. The result of \cite{Wu} applies to the log-gamma polymer, but only deals with point-to-point partition functions. It seems that significant extra work is needed to prove the convergence of the point-to-line partition function.   

\medskip 

Thus, modulo the proof of convergences \eqref{eq:convergenceKPZfull} and \eqref{eq:convergenceKPZhalf}, our Theorem \ref{main} implies \eqref{eq:identityKPZ}. Proofs of these convergences are outside the scope of the present paper, but we hope that proving \eqref{eq:identityKPZ} may be a motivation to extend the existing proofs of convergence for full-space and half-space polymers to a more general setting, covering in particular \eqref{eq:convergenceKPZfull} and \eqref{eq:convergenceKPZhalf}.



 \subsection{Outline}
  In Section \ref{sec:rsk}, we state and prove a number of prerequisites about the geometric RSK correspondence and Whittaker functions. In Section \ref{sec_proof}, we prove two formulas for the Laplace transform of the point-to-line partition function of trapezoidal log-gamma polymers. In Section \ref{sec4}, we prove the main identity in distribution (Theorem \ref{main}).  In Section \ref{sec:asymptotics}, we prove our asymptotic result (Theorem \ref{main2}).

 \subsection*{Acknowledgments}
G.B. thanks Pierre Le Doussal and Nikos Zygouras for useful discussions, and Xuan Wu for explanations related to \cite{Wu}. We also thank an anonymous referee for his/her constructive and detailed comments.

\section{Geometric RSK correspondence and Whittaker functions } \label{sec:rsk}
This section provides preliminary results useful in the following sections. In particular we introduce the geometric RSK correspondence (we will denote it by $\operatorname{gRSK}$) on polygonal arrays (in particular, symmetric polygonal arrays) and Whittaker functions. We refer the reader to Elia Bisi's PhD thesis \cite{bisi2018random} for a more detailed introduction to these subjects. We will also introduce the function $\mathcal T_{\alpha_{\circ}, \bsb\beta; r}$ which appears in \eqref{eq:Laplacetrapezoidal} and state some of its properties. 

\medskip 
It was found in \cite{OSZ,COSZ}  that   Whittaker functions  arise naturally as we consider gRSK on (rectangular) domains (see also \cite{o2012directed} for an earlier occurrence of Whittaker functions in a directed polymer context). The partition functions for log-gamma polymers on such domains may be written as integrals involving  Whittaker functions.
For log-gamma polymers on non-rectangular domains, a generalization of gRSK  to polygonal arrays \cite{NZ17} may be used 
to derive, for example, the Laplace transform of the partition function, which is the approach adopted by  \cite{BisiZygouras} to treat half-space log-gamma polymers.

\subsection{From half-space to symmetrized  polymer partition functions} \label{sec:halftosymmetric}

    Before introducing gRSK, we point out that gRSK does not seem to apply well to trapezoidal domains, that we are interested in. However, a trapezoidal log-gamma polymer can always be converted to an equivalent    symmetric log-gamma polymer, in that they have identical partition functions. And it is convenient to apply gRSK to symmetric domains \cite{OSZ,BisiZygouras}. 
    
    Some notations   come in handy. For  any domain  $\mathcal{I}\subset\mathbb{N}\times\mathbb{N}$, we define its \emph{transpose} as $\mathcal{I}^t:=\{(i,j)\mid (j,i)\in\mathcal{I}\}$. 
    A domain $\mathcal{I}$ is symmetric if $\mathcal{I}=\mathcal{I}^t$.
    An array  $\bsb{w}$ defined on  a symmetric domain $\mathcal{I}$ is called \textit{symmetric} if ${w}_{i,j}={w}_{j,i}$ for all $i,j$. We may now state     the aforementioned equivalence between trapezoidal log-gamma polymers and symmetric log-gamma polymers.

Recall from (\ref{defZflat})  that  the  point-to-line partition function on a trapezoidal log-gamma polymer  is given by 
\begin{equation*}
Z^{ \rflat}(n;m):=\sum_{  \substack{1\leq k\leq n \\ \pi :(1,1)\to(k,2n-k+m+1)}  } \prod_{(i, j) \in \pi} W_{i,j},
\end{equation*}
where the sum is over up-right paths confined in a trapezoidal domain $\mathcal{I}=\{(i,j)\mid 1\leq i\leq n, i\leq j\leq 2n+m-i+1\}$ (cf. Figure \ref{halfspaceparam}).
Now,  let  $\widetilde{W}_{i,j}$'s be   random variables  given by  
\begin{equation*} 
\widetilde{W}_{i, j} :\overset{}{=}\left\{\begin{array}{ll}
{W}_{i, j}/2 ,&   i=j  , \\
{W}_{i, j} ,&  i<j  , \\
{W}_{j, i} ,&   i >j .
\end{array}\right.
\end{equation*}

We see from this definition that $ (\widetilde{W}_{i, j})$ is defined on a  symmetric domain (that is, $ \widetilde{W}_{i, j}$ exists if and only if $ \widetilde{W}_{j,i}$ exists) and   $ (\widetilde{W}_{i, j})_{i\leq j}$ is a family of independent random variables. We call this the \emph{symmetrized} polymer
and define  the \textit{symmetrized partition function} as
\[
Z^{\symflat}(n;m):=\sum_{\substack{1\leq k\leq n \\ \pi:(1,1)\to(k,2n-k+m+1)}} \prod_{(i,j)\in\pi}\widetilde{W}_{i,j}+\sum_{\substack{1\leq k\leq n \\ \pi:(1,1)\to(2n-k+m+1,k)}} \prod_{(i,j)\in\pi}\widetilde{W}_{i,j},
\]
where  the sums are over  up-right paths $\pi$ in the $\symflat$--shape domain $\mathcal{I}\cup \mathcal{I}^t$. Since $(\widetilde{W}_{i,j})$ is symmetric,  a path $\pi$ in $\mathcal{I}\cup\mathcal{I}^t$   can always be converted, by performing   reflections with respect to the diagonal $\{i=j\}$,   to a path $\pi^{\prime}$ in $\mathcal{I}$ with the same weight, i.e.,  $\prod_{(i,j)\in\pi}\widetilde{W}_{i,j}=\prod_{(i,j)\in\pi^{\prime}}\widetilde{W}_{i,j}$. Conversely, for a fixed $\pi^{\prime}$, the number of such $\pi$'s is   $2^{|\{i\mid (i,i)\in \pi^{\prime}\}|}$. Hence
\begin{align}
   \nonumber  Z^{\symflat}(n;m) 
    &= \sum_{  \substack{1\leq k\leq n \\ \pi^{\prime} :(1,1)\to(k,2n-k+m+1)}  } 
    \prod_{(i, j) \in \pi^{\prime}} 2^{|\{i\mid (i,i)\in \pi^{\prime}\}|}\widetilde{W}_{i, j} \\
    &\overset{}{=}  \sum_{  \substack{1\leq k\leq n \\ \pi^{\prime} :(1,1)\to(k,2n-k+m+1)}  } \prod_{(i, j) \in \pi^{\prime}}  {W}_{i, j} = Z^{\rflat}(n;m). \label{eq3.1}
\end{align}
This equivalence allows us to  study $Z^{\symflat}(n;m)$ in lieu of studying $Z^{\rflat}(n;m)$, which will be the starting point of our proof for Theorem \ref{main}.

   \subsection{Geometric RSK correspondence}
   The geometric RSK correspondence was introduced by Kirillov \cite{Kri01} as the geometric lifting of the RSK correspondence. Noumi and Yamada proposed an alternative definition, purely written in terms of matrix algebra \cite{NY04}. It is shown in \cite{OSZ} that gRSK can be constructed using local moves and that this is equivalent to \cite{NY04}.  In the following we present this construction via local moves.
   
  \subsubsection{Geometric RSK correspondence on polygonal arrays}
We start with some definitions. 
 \begin{definition}
  A \emph{polygonal domain} $\mathcal{I}$ is a finite subset of $ \mathbb{N}\times\mathbb{N}$   such  that    we have $(i,j)\in\mathcal{I}$   whenever $(i+1,j)\in\mathcal{I}$ or  $(i,j+1)\in\mathcal{I}$.
A \emph{polygonal array} on $\mathcal{I}$ is   an array $(w_{p,q})_{ (p,q)\in \mathcal{I}}\in\mathbb{R}_{>0}^{|\mathcal{I}|}$ with positive entries   indexed by a polygonal domain $\mathcal{I}$. 
For notational convenience, we   freely write $\bsb{w}$, $\bsb{w}_{\mathcal{I}}$, $(w_{p,q})$, or $(w_{p,q})_{\mathcal{I}}$  to refer to the same polygonal array when there is no ambiguity.

An index  $(i,j)\in\mathcal{I}$  is called a  \emph{border index} if $(i+1,j+1)$ does not belong to $\mathcal{I}$. 
If none  of $(i+1,j)$, $(i,j+1)$, $(i+1,j+1)$   belongs to $\mathcal{I}$, then $(i,j)$ is called an \emph{outer index}.
\end{definition}

The gRSK map  is a birational map between polygonal arrays    on a same polygonal domain.   The following presentation of gRSK, following \cite{NZ17, OSZ}, is based on local moves.
Local moves are    birational maps between polygonal arrays that are localized on an entry and its neighbors. We will use two families  of local moves $(a_{i,j})$ and $(b_{i,j})$.
\begin{definition}
Let $(w_{i,j})$ be a polygonal array on $\mathcal{I}$. For an index $(i,j)\in\mathcal{I}$, the local move $a_{i,j}$ replaces $w_{i,j}$ with $$w_{i, j}\left(w_{i-1, j}+w_{i, j-1}\right),$$ with the convention $w_{i,0}=w_{0,j}=0$ and $w_{0,1}+w_{1,0}=1$. 
For a non-border index $(i,j)\in\mathcal{I}$, the local move $b_{i,j}$ replaces $w_{i,j}$ with $$ {w_{i, j}^{-1}}\left(w_{i-1, j}+w_{i, j-1}\right)\left( {w_{i+1, j}}^{-1}+ {w_{i, j+1}}^{-1}\right)^{-1}.$$
For $i\leq j$  ($i\geq j$ resp.), let
\[\varrho_{i, j}:=  a_{i, j} \circ  b_{i-1, j-1} \circ \dots \circ b_{1,j-i+1} \quad \left(\varrho_{i, j}:=  a_{i, j} \circ  b_{i-1, j-1} \circ \dots \circ b_{i-j+1,1} \text{ resp.}\right).\]
\end{definition}
Now we are ready to give the definition of  gRSK, which is done by induction.
Let $\bsb{w}$ be a polygonal array   on $\mathcal{I}$,   let $\mathcal{I}_{\text{out}}$ be the set of outer indices of $\mathcal{I}$ and let $\mathcal{I}^\circ : =\mathcal{I}\setminus \mathcal{I}_{\text{out}}$. The induction starts from $\operatorname{gRSK}(\varnothing):=\varnothing$. 
By introducing the notations $\boldsymbol{w}^{\circ}=\left\{w_{i, j}:(i, j) \in \mathcal{I}^{\circ}\right\}, \boldsymbol{w}^{\text {out }}=\left\{w_{i, j}:(i, j) \in \mathcal{I}_{\text {out}}\right\}$, the induction relation reads
\begin{equation}\label{eq19}
   \operatorname{gRSK}(\boldsymbol{w}):=\bigcirc_{(i, j) \in \mathcal{I}_{\text {out }}} \varrho_{i, j}\left(\operatorname{gRSK}\left(\boldsymbol{w}^{\circ}\right) \sqcup \boldsymbol{w}^{\text {out }}\right),
\end{equation}
where $\bigcirc$ denotes compositions of $\varrho_{i, j}$'s and $ \sqcup $ denotes union of two arrays.  Note that the order in which the $\varrho_{i, j}$'s for $(i, j) \in \mathcal{I}_{\text {out }}$ are composed is irrelevant (see \cite{NZ17}).

\begin{proposition}[{\cite[Proposition 2.6, 2.7]{NZ17}}]\label{gRSKprop} 

Let $\bsb{t} = \operatorname{gRSK}(\bsb{w})$. 
\begin{itemize}
\item 
The Jacobian of the gRSK  in log-coordinates 
\begin{equation*}
    \left(\log w_{i, j}:(i, j) \in \mathcal{I} \right) \mapsto\left(\log t_{i, j}:(i, j) \in \mathcal{I} \right)
\end{equation*}
has absolute value 1.
    \item With the convention   $t_{0,j}=t_{i,0}=0$ for all $i,j$, we have
\begin{equation*} 
    \sum_{(i,j)\in\mathcal{I}}\frac{1}{w_{i,j}}  = \frac{1}{t_{1,1}}+\sum_{(i, j) \in \mathcal{I}} \frac{t_{i-1, j}+t_{i, j-1}}{t_{i, j}}.
\end{equation*} 

\item  
For any border index $(n,m)$,
\[ t_{n,m}=\sum_{\pi:(1,1)\rightarrow (n,m)}\prod_{(i,j)\in\pi} w_{i,j},
\]
where the sum is over up-right paths  from $(1,1)$ to $(n,m)$.
\item  
For any border index $(n,m)$,
\[\prod_{i=1}^n\prod_{j=1}^mw_{i,j}=\tau_{m-n},
\]
where  $\tau_{q}: = \prod_{j-i=q}t_{i,j}$.
In particular, for an outer index $(n,m)$, 
\[  \prod_{i=1}^n  w_{i,m}=\frac{\tau_{m-n}}{\tau_{m-n-1}} \quad \text{and} \quad
\prod_{j=1}^m  w_{n,j}=\frac{\tau_{m-n}}{\tau_{m-n+1}}
\]
hold.
\end{itemize}
\end{proposition} 
Note that $t_{n,m}$, $(n,m)$ being a border index, can thus be regarded as a point-to-point partition function for a polymer with weight array $\bsb{w}$.

\subsubsection{Symmetric geometric RSK correspondence}
As motivated by Subsection \ref{sec:halftosymmetric}, we introduce gRSK for symmetric arrays, first considered in \cite{OSZ}. Suppose  $\mathcal{I}\subset \{(i,j)\mid i\leq j\}$  and  that $\mathcal{I}\cup \mathcal{I}^t$ is a polygonal domain. 
Then a symmetric  array  $\bsb{w}$   on $\mathcal{I}\cup \ci^t$ only has $|\ci |$ free variables (say, the restriction of $\bsb{w}$ to the quadrant $\{(i,j)\mid i\leq j\}$). Let $\bsb{t}:=\operatorname{gRSK}(\bsb{w})$, then $\bsb{t}$ also is symmetric  (see Proposition \ref{jacob} below) and thus has $|\ci |$ free variables. We may thus regard gRSK in this symmetric setting as a birational map from $\mathbb{R}_{>0}^{|\ci|}$ to  $\mathbb{R}_{>0}^{|\ci|}$.

\begin{proposition}[\cite{OSZ,BisiZygouras}]\label{jacob}
If $\bsb{w}$ is symmetric, then  $\bsb{t}=\operatorname{gRSK}(\bsb{w})$ is also symmetric. In addition,
the Jacobian of   
\begin{equation*} 
    \left(\log w_{i, j}:  i \leq j\right) \mapsto\left(\log t_{i, j}: i \leq j\right)
\end{equation*}
has absolute value 1.  
\end{proposition}
  
  The following proposition provides a formula  for the product of diagonal entries $\prod_{i}w_{i,i}$ in terms of $t_{j,j}$'s. This was first proved in {\cite[Lemma 5.1]{OSZ}} for cases where the underlying polygonal domain is  $\{(i,j)\mid 1\leq i,j \leq n\}$ for some $n$. We provide a proof here  for general symmetric polygonal arrays\footnote{As pointed out by an anonymous referee, Proposition \ref{prop2} can also be deduced from  \cite[Lemma 5.1]{OSZ} using the fact that the diagonal entries of the gRSK image of a polygonal array $\bsb w $ coincide exactly with the diagonal entries of the gRSK image of the square subarray $\bsb w^{sq} = (\bsb w_{ij})_{1\leq i,j\leq \ell}$. In fact, this implies that Proposition \ref{prop2} is still true under the weaker assumption that $\bsb w^{sq}$ is symmetric.}.
  
  We introduce the \emph{diagonal length} $\ell$ of a polygonal array $\bsb{w}$ on $\ci$ as $\ell:=\operatorname{card}\{i\mid (i,i)\in \mathcal{I}\}$.
\begin{proposition}\label{prop2} Suppose       $\bsb{w}$ is a  symmetric polygonal array. Let $\bsb{t}=\operatorname{gRSK}(\bsb{w})$, then
   \[
   4^{\lfloor n / 2\rfloor} \prod_{i=1}^{n} w_{i, i} =\prod_{j=1}^{n}t_{j,j}^{(-1)^{n-j}},
   \]
   where $n$ is the {diagonal length} of $\bsb{w}$.
\end{proposition}
\begin{proof} 
We prove this by induction on $n$.
Observe that  when $n=1$, we have $\bsb{t}=\bsb{w}$ and the formula follows.

Suppose that the formula holds for any symmetric  polygonal array of diagonal length $n-1$, we now show that it also holds for any symmetric polygonal array $\bsb{w}$ of diagonal length $n$.
 
By (\ref{eq19}), if $(n,n)$ is not an outer index, then $t^{\circ}_{j,j}=t_{j,j}$ for all $j=1,\dots,n$, where $t^{\circ}$ denotes $\operatorname{gRSK}(\bsb{w}^{\circ})$. Since $w^{\circ}_{i,i}=w_{i,i}$ for all $1\leq i\leq n$, the problem is thus reduced to proving the desired formula for $\bsb{w}^{\circ}$, which is $\bsb{w}$ deprived of the entries on its outer indices. After repeating this process for a finite time, the reduction stops at some ${{{{\bsb{w}^{\circ}}^{.}}^{.}}^{.}}^{\circ}$, of which $(n,n)$ is an outer index (see Figure \ref{img_rdc}).

Assume now without loss of generality that    $(n,n)$ is an outer index of $\bsb{w}$. Then the  diagonal length of $\bsb{w}^\circ$ is $n-1$ and the induction hypothesis yields 
\begin{equation*}
    4^{\lfloor \tfrac{n-1}{2}\rfloor} \prod_{i=1}^{n-1} w_{i i} =\prod_{j=1}^{n-1}{t^{\circ}_{j,j}}^{(-1)^{n-1-j}}.
\end{equation*}
By the hypothesis that $\bsb{w}$ is symmetric,    $\bsb{t}^{\circ}=\operatorname{gRSK}(\bsb{w}^{\circ})$ is symmetric.  Hence we may rewrite (\ref{eq19}) as
\begin{align*}
    t_{1,1}&=\frac{1}{t^{\circ}_{1,1}}\left(\frac{1}{t^{\circ}_{1,2}} +\frac{1}{t^{\circ}_{2,1}}\right)^{-1} =   \frac{t^{\circ}_{1,2}}{2t^{\circ}_{1,1}}, \\
    t_{j,j}&=\frac{1}{t^{\circ}_{j,j}} t^{\circ}_{j-1,j}t^{\circ}_{j,j+1},\quad \text{for}  \quad 2\leq j\leq n-1,\\
    t_{n,n}&= 2w_{n,n}t_{n-1,n}^{\circ}.
\end{align*}
Therefore,
\begin{align*}
    \prod_{j=1}^n(t_{j,j})^{(-1)^{n-j}}&=  t_{n,n}{}t_{n-1,n-1}^{-1}{}t_{n-2,n-2} \cdots t_{1,1}^{(-1)^{n-1}} \\ 
    &= 2^{1+(-1)^n}w_{n,n}\prod_{j=1}^{n-1}(t_{j,j})^{(-1)^{n-1-j}}\\
     &= 2^{1+(-1)^n}w_{n,n} 4^{\lfloor \tfrac{n-1}{2} \rfloor}\prod_{i=1}^{n-1}w_{i,i}\\
     &=4^{\lfloor \tfrac{n}{2} \rfloor}\prod_{i=1}^{n}w_{i,i},
\end{align*}
which completes the proof.

\end{proof}  
  
 \begin{figure}
     \centering\begin{tikzpicture}[scale=0.7] 
		\draw[->] (3,1.5) node[anchor=south]{} to[bend left] (4.5,1.5);  
		\draw[->] (8,1.5) node[anchor=south]{} to[bend left] (9.5,1.5);  
		  
 \coordinate (1) at (0,0);
 \coordinate (2) at (1,0);
 \coordinate (3) at (2,0);
 \coordinate (4) at (3,0);
 \coordinate (5) at (0,1);
 \coordinate (6) at (1,1);
 \coordinate (7) at (2,1);
 \coordinate (8) at (3,1);
 \coordinate (9) at (0,2);
 \coordinate (10) at (1,2);
 \coordinate (11) at (0,3);
 \coordinate (12) at (1,3);
 \foreach \n in {1,2,3,4,5,6,7,8,9,10,11,12} \fill [black] (\n)
   circle (2pt) node [below] {};

 \coordinate (1) at (5,0);
 \coordinate (2) at (6,0);
 \coordinate (3) at (7,0);
 \coordinate (4) at (8,0);
 \coordinate (5) at (5,1);
 \coordinate (6) at (6,1);
 \coordinate (7) at (7,1);
 \coordinate (9) at (5,2);
 \coordinate (10) at (6,2);
 \coordinate (11) at (5,3);
 \foreach \n in {1,2,3,4,5,6,7,9,10,11} \fill [black] (\n)
   circle (2pt) node [below] {};  
   
 \coordinate (1) at (10,0);
 \coordinate (2) at (11,0);
 \coordinate (3) at (12,0);
 \coordinate (5) at (10,1);
 \coordinate (6) at (11,1);
 \coordinate (9) at (10,2);
 \foreach \n in {1,2,3,5,6,9} \fill [black] (\n)
   circle (2pt) node [below] {};

		\end{tikzpicture}
		
	\caption{Reduction from $\bsb{w}$ to $\bsb{w}^{\circ }$, and then to $\bsb{w}^{\circ \circ}$.}
	\label{img_rdc}
 \end{figure}

  \subsection{Whittaker functions} \label{sec:whi}
  We use  $\mathrm{i}=\sqrt{-1}$ to denote the imaginary unit and reserve $i$ for use as indices.
 \subsubsection{$\mathfrak{gl}_n$-Whittaker functions} 
 We define $\ggl{n}$-Whittaker functions  through Givental's   integral formula   \cite{Giv97}.
  Fix a parameter sequence $\bsb{\alpha}=(\alpha_1,\dots, \alpha_n)\in \mathbb
  C^n$. Consider a  triangular array $\bsb{z}=(z_{i,j}: 1\leq i\leq n, 1\leq j\leq i)$ (Figure \ref{pat}). 
We   define the \emph{type} of such an array as 
\[
\operatorname{type}(\boldsymbol{z})_{i}:=\frac{\prod_{j=1}^{i} z_{i, j}}{\prod_{j=1}^{i-1} z_{i-1, j}} \quad \text { for } \quad i=1, \ldots, n,
\]
with the convention $\prod_{j=1}^0z_{0,j}=1$.
Denote 
$$ \operatorname{type}(\boldsymbol{z})^{\boldsymbol{\alpha}}:= \prod_{i=1}^{n} \operatorname{type}(\bsb{z})_i^{{\alpha}_i}.$$
  The $\mathfrak{gl}_n$-Whittaker function is defined via an integral on   triangular arrays with fixed bottom row:
\begin{equation}\label{eq7}
\Psi_{\boldsymbol{\alpha}}^{\mathfrak{g l}_{n}}(\boldsymbol{x}):=
\int_{\mathbb{R}_{>0}^{n(n-1)/2}}\operatorname{type}(\boldsymbol{z})^{-\boldsymbol{\alpha}}
\exp \left(-\sum_{i=1}^{n} \sum_{j=1}^{i}\frac{z_{i+1, j+1}+z_{i-1,j}}{z_{i, j}} \right) \prod_{\substack{1 \leq i<n \\ 1 \leq j \leq i}} \frac{dz_{i, j}}{z_{i, j}},
\end{equation}
where $z_{n,j}:=x_j$ for $1\leq j\leq n$ and by convention $z_{i,j}=0$ if $(i,j)$ is out of the range $1\leq i\leq n,1\leq j\leq i$. 
 A direct consequence of this definition is that for any $c\in\mathbb{C}$,
\begin{equation}\label{eq:whitransprop}
    \Psi_{\boldsymbol{\alpha}+c}^{\mathfrak{g l}_{n}}(\boldsymbol{x})=\left(\prod_{i=1}^{n} x_{i}\right)^{-c} \Psi_{\boldsymbol{\alpha}}^{\mathfrak{g l}_{n}}(\boldsymbol{x}).
\end{equation}

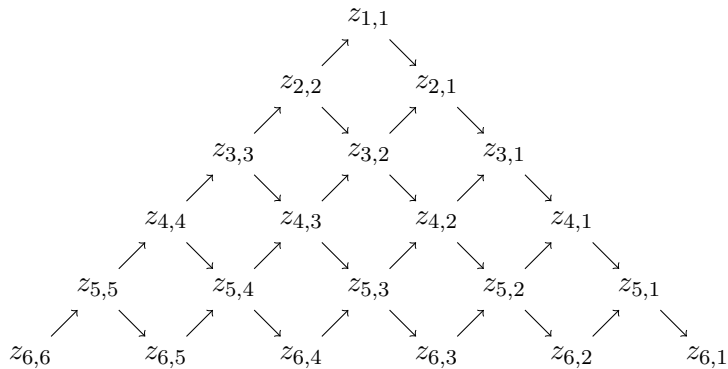
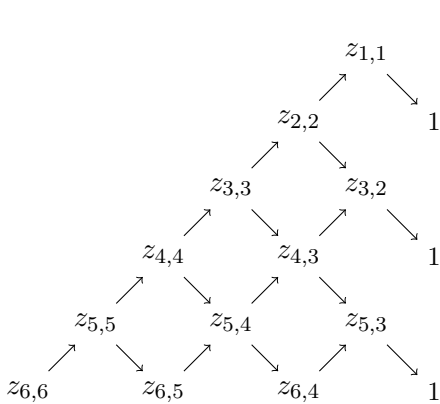
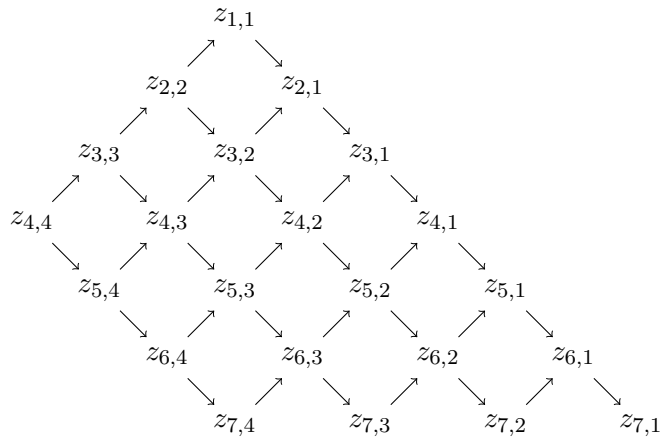
\begin{figure}
\centering
\begin{subfigure}{.85\textwidth}
\begin{center}
\begin{tikzpicture}[scale=0.9]
\node (z66) at (-2,-2) {$z_{6,6}$};
\node (z65) at (0,-2) {$z_{6,5}$};
\node (z64) at (2,-2) {$z_{6,4}$};
\node (z63) at (4,-2) {$z_{6,3}$};
\node (z62) at (6,-2) {$z_{6,2}$};
\node (z61) at (8,-2) {$z_{6,1}$};
\node (z55) at (-1, -1) {$z_{5,5}$};
\node (z54) at (1,-1) {$z_{5,4}$};
\node (z53) at (3,-1) {$z_{5,3}$};
\node (z52) at (5,-1) {$z_{5,2}$};
\node (z51) at (7,-1) {$z_{5,1}$};
\node (z44) at (0, 0) {$z_{4,4}$};
\node (z43) at (2, 0) {$z_{4,3}$};
\node (z42) at (4, 0) {$z_{4,2}$};
\node (z41) at (6, 0) {$z_{4,1}$};
\node (z33) at (1,1) {$z_{3,3}$};
\node (z32) at (3,1) {$z_{3,2}$};
\node (z31) at (5,1) {$z_{3,1}$};
\node (z22) at (2,2) {$z_{2,2}$};
\node (z21) at (4,2) {$z_{2,1}$};
\node (z11) at (3,3) {$z_{1,1}$}; 
\draw[-> ]  (z44) edge (z33)  (z33) edge (z22) edge (z43) (z43) edge  (z32) (z22) edge (z11) edge (z32) (z32) edge (z21) edge (z42) (z42) edge (z31) (z11)  edge (z21)  (z21) edge (z31)(z31) edge (z41);
\foreach \n in {1,2,3,4}
{
\pgfmathtruncatemacro{\y}{\n +1}
\draw[->] (z4\n) edge(z5\n) ; 
\draw[->] (z5\y) edge (z4\n);
}
\foreach \n in {1,2,3,4,5}
{
\pgfmathtruncatemacro{\y}{\n +1}
\draw[->] (z5\n) edge(z6\n) ; 
\draw[->] (z6\y) edge (z5\n);
}
\end{tikzpicture}
\end{center}
\caption{$\mathfrak{gl}_n$ case, $n=6$. }
\label{pat}
\end{subfigure}

\begin{subfigure}{.4\textwidth}
\begin{center} 
\begin{tikzpicture}[scale=0.9]
\node (z63) at (-2,-2) {$z_{6,6}$}; 
\node (z62) at (0,-2) {$z_{6,5}$};
\node (z61) at (2,-2) {$z_{6,4}$}; 
\node (z53) at (-1, -1) {$z_{5,5}$}; 
\node (z52) at (1, -1) {$z_{5,4}$};
\node (z51) at (3, -1) {$z_{5,3}$}; 
\node (z42) at (0, 0) {$z_{4,4}$};
\node (z41) at (2, 0) {$z_{4,3}$}; 
\node (z32) at (1,1) {$z_{3,3}$};
\node (z31) at (3,1) {$z_{3,2}$}; 
\node (z21) at (2,2) {$z_{2,2}$}; 
\node (z11) at (3,3) {$z_{1,1}$}; 

\node (e1) at (4,2) {\small $1$};
\node (e2) at (4,0) {\small $1$};
\node (e3) at (4,-2) {\small $1$};
\draw[-> ] (z63) edge(z53) (z53) edge(z42) edge(z62) (z62) edge(z52) (z52) edge(z41) edge (z61) (z61) edge (z51)  (z42) edge(z52) edge (z32)  (z32) edge (z21) edge (z41) (z41) edge  (z31) edge(z51) (z21) edge (z11) edge (z31) (z11) edge (e1) (z31) edge (e2) (z51) edge(e3) ;
\end{tikzpicture}
\end{center}
\caption{ $\mathfrak{so}_{2n+1}$ case, $n=3$.}
\label{pat_so}
\end{subfigure}
\begin{subfigure}{.58\textwidth}
\begin{center}
\begin{tikzpicture}[scale=0.9]
\node (z74) at (3, -3) {$z_{7,4}$};
\node (z73) at (5, -3) {$z_{7,3}$};
\node (z72) at (7, -3) {$z_{7,2}$};
\node (z71) at (9, -3) {$z_{7,1}$};
\node (z64) at (2, -2) {$z_{6,4}$};
\node (z63) at (4, -2) {$z_{6,3}$};
\node (z62) at (6, -2) {$z_{6,2}$};
\node (z61) at (8, -2) {$z_{6,1}$};
\node (z54) at (1, -1) {$z_{5,4}$};
\node (z53) at (3, -1) {$z_{5,3}$};
\node (z52) at (5, -1) {$z_{5,2}$};
\node (z51) at (7, -1) {$z_{5,1}$};
\node (z44) at (0, 0) {$z_{4,4}$};
\node (z43) at (2, 0) {$z_{4,3}$};
\node (z42) at (4, 0) {$z_{4,2}$};
\node (z41) at (6, 0) {$z_{4,1}$};
\node (z33) at (1,1) {$z_{3,3}$};
\node (z32) at (3,1) {$z_{3,2}$};
\node (z31) at (5,1) {$z_{3,1}$};
\node (z22) at (2,2) {$z_{2,2}$};
\node (z21) at (4,2) {$z_{2,1}$};
\node (z11) at (3,3) {$z_{1,1}$}; 
\draw[-> ] (z44) edge (z33) edge (z54) (z54) edge (z43) edge(z64)   (z33) edge (z22) edge (z43) (z43) edge (z53) edge (z32) (z53) edge (z63) edge (z42) (z64) edge(z53) edge(z74) (z74) edge (z63) (z63) edge (z52) edge(z73) (z73) edge (z62)  (z52) edge (z41) edge (z62) (z62) edge (z51) edge(z72) (z72) edge (z61) (z61) edge (z71) (z51) edge (z61) (z22) edge (z11) edge (z32) (z32) edge (z21) edge (z42) (z42) edge (z52) edge (z31) (z11)  edge (z21)  (z21) edge (z31)(z31) edge (z41) (z41) edge (z51) ;
\end{tikzpicture}
\end{center}
\caption{generalized $\mathfrak{gl}_n$ case, $n=4,m=7$. }
\label{pat2}
\end{subfigure}

\caption{Tableau of integral variables $z_{i,j}$ (including the fixed bottom row) for $\mathfrak{gl}_n$ Whittaker functions,  $\mathfrak{so}_{2n+1}$ Whittaker functions, and     generalized $\mathfrak{gl}_n$-Whittaker functions;  an arrow $a\rightarrow b $ means that  the term $\frac{a}{b}$ appears in the exponential  term in \eqref{eq7}, \eqref{eq7.1}, or \eqref{eq7.5}, respectively for the three cases.}
\label{fig:test}
\end{figure}
  
Now we introduce a generalization of $\ggl n$-Whittaker functions, following  \cite{OSZ}.
This generalized Whittaker function will not be used in the sequel, but completes our discussion for \eqref{eq:Laplacepointtopoint}.
Suppose  $m\geq n$. Let $\bsb\beta\in\mathbb C^m$ and let $r\in \mathbb C$ be such that  $\Re(r)>0$. 
Consider an array $\bsb{z}=(z_{i,j}: 1\leq i\leq m, 1\leq j\leq i\wedge n)$ (Figure \ref{pat2}).
  We   define the {type} of such an array as 
\[
\operatorname{type}(\boldsymbol{z})_{i}:=\frac{\prod_{j=1}^{i\wedge n} z_{i, j}}{\prod_{j=1}^{(i-1)\wedge n} z_{i-1, j}} \quad \text { for } \quad i=1, \ldots, m.
\]
Denote as above
$$ \operatorname{type}(\boldsymbol{z})^{\boldsymbol{\beta}}:= \prod_{i=1}^{m} \operatorname{type}(\bsb{z})_i^{{\beta}_i}.$$
  The generalized  $\mathfrak{gl}_n$-Whittaker function is also defined  via an integral on  arrays with fixed bottom row:
\begin{equation}\label{eq7.5} \Psi_{\boldsymbol{\beta};r}^{\mathfrak{g l}_{n}}(\boldsymbol{x}):=
\int_{\mathbb{R}_{>0}^{mn-n(n+1)/2}}\operatorname{type}(\boldsymbol{z})^{-\boldsymbol{\beta}}
\exp \left(-\frac{r}{z_{n,n}}-\sum_{i=1}^{m} \sum_{j=1}^{i\wedge n}\frac{z_{i+1, j+1}+z_{i-1,j}}{z_{i, j}}\right) \prod_{\substack{1 \leq i<m\\ 1 \leq j \leq i\wedge n}} \frac{dz_{i, j}}{z_{i, j}},
\end{equation}
where $z_{m,j}:=x_j$ for $1\leq j\leq n$ and  by convention $z_{i,j}=0$ if $(i,j)$ is out
 of the range $1\leq i\leq m, 1\leq j\leq i\wedge n$.  
  
  \medskip
  For a function $g\in L^{2}\left(\mathbb{R}_{>0}^{n}, \prod_{i=1}^{n} {~d} x_{i} / x_{i}\right)$, we define its \emph{Whittaker transform} as
\begin{equation*}
    \widehat{g}(\boldsymbol{\lambda}):=\int_{\mathbb{R}_{>0}^{n}} g(\boldsymbol{x}) \Psi_{\boldsymbol{\lambda}}^{\mathfrak{g} \mathfrak{l}_{n}}(\boldsymbol{x}) \prod_{i=1}^{n} \frac{d x_{i}}{x_{i}}.
\end{equation*}
Let $L_{\mathrm{sym}}^{2}\left((\mathrm{i} \mathbb{R})^{n}, s_{n}(\boldsymbol{\lambda}) d \boldsymbol{\lambda}\right)$ be the space of square integrable functions which are  symmetric with respect to permuting  their variables, equipped with the   \emph{Sklyanin measure} 
\begin{equation}\label{eq63}
    s_{n}(\bsb{\lambda})=\frac{1}{(2 \pi \mathrm{i})^{n} n!} \prod_{j \neq k} \Gamma\left(\lambda_{j}-\lambda_{k}\right)^{-1}.
\end{equation}
It is known  that the Whittaker transformation $g\mapsto \widehat{g}$ 
is   an isometry 
from  $L^{2}(\mathbb{R}_{>0}^{n},  \prod_{i=1}^{n} d x_{i} / x_{i})$  to $L_{\mathrm{sym}}^{2}\left((\mathrm{i} \mathbb{R})^{n}, s_{n}(\boldsymbol{\lambda}) d \boldsymbol{\lambda}\right)$ \cite{STS94,KL01}. 
In particular, the following Plancherel  formula holds.  For all $f_1, f_2 \in L^{2}\left(\mathbb{R}_{>0}^{n}, \prod_{i=1}^{n} d x_{i} / x_{i}\right)$, we have
\begin{equation}\label{plancherel}
    \int_{\mathbb{R}_{>0}^{n}} f_1(\boldsymbol{x}) \overline{f_2(\boldsymbol{x})} \prod_{i=1}^{n} \frac{{d} x_{i}}{x_{i}}=\int_{(\mathrm{i} \mathbb{R})^{n}} \widehat{f_1}(\boldsymbol{\lambda}) \overline{\widehat{f_2}(\boldsymbol{\lambda})} s_{n}(\boldsymbol{\lambda}) {d} \boldsymbol{\lambda}.
\end{equation}

 Now we turn to  the Whittaker transform of a particular function $\mathcal{T}_{\alpha,\bsb{\beta};r}(\bsb{x})$, which  naturally appears in the  Laplace transform of the point-to-line partition function of trapezoidal polymers, and symmetrized polymers in general, see Remark \ref{rem:linktohalfspaceMacdo}.

\begin{definition} \label{defT}
Assume $r>0$, $\alpha\in \mathbb{C}$, $\bsb{\beta}\in \mathbb{C}^m$.
For  $\bsb{x}=(x_1,\dots,x_n)\in \mathbb{R}_{>0}^{n}$, 
consider an array $\bsb z=(z_{i,j}:1\leq i\leq m+1, 1\leq j\leq n)$ with fixed bottom row $z_{m+1,j}=x_j$ for all $1\leq j \leq n$.
Define  the {type} of such an array as 
\[
\operatorname{type}(\boldsymbol{z})_{i}:=\frac{\prod_{j=1}^{n} z_{i+1, j}}{\prod_{j=1}^{n} z_{i, j}} \quad \text { for } \quad i=1, \ldots, m.
\]
Denoting $\operatorname{type}(\bsb z)^{\bsb \beta}=\prod_{i=1}^m \operatorname{type}(\bsb z)_i^{\beta_i}$,
we define
\begin{align}\nonumber
\mathcal{T}_{\alpha,\bsb{\beta};r}(\bsb{x}):= 
 \int_{\mathbb{R}_{>0}^{nm}} &\left(\prod_{i=1}^n (z_{1,i}/r)^{(-1)^{i}}\right)^{\alpha}
\operatorname{type}(\bsb z)^{-\bsb \beta}\\
&\times\exp\left(-\frac{r}{z_{1,n}}-\sum_{i=1}^{m+1}\sum_{j=1}^n \frac{z_{i+1,j+1}+z_{i-1,j}}{z_{i,j}}  \right) \prod_{i=1}^m\prod_{j=1}^n \frac{d z_{i,j}}{z_{i,j}}, 
\label{eq99}
\end{align}
    with the convention that $z_{i,j}=0$ whenever $(i,j)$ is out of the range $1\leq i\leq m+1, 1\leq j\leq n$. We  also define    $ \mathcal{T}_{\alpha, \emptyset;r}(\bsb{x}) = \left(\prod_{i=1}^n(x_i/r)^{(-1)^{i}}\right)^{\alpha} e^{-r/x_n}$. 
\end{definition} 
The well-definedness of   $\mathcal{T}_{\alpha,\bsb{\beta};r}$ will be shown in Proposition \ref{prop:whittransformT}.
See Figure \ref{figT} for a visual representation of  integral variables $(z_{i,j})$. 
 The Whittaker transform of $\mathcal{T}_{\alpha,\bsb{\beta};r}(\bsb{x})$ in the case $\bsb{\beta}=\emptyset$ is known and stated in the following proposition.
\begin{proposition}[{\cite{Stade}, \cite[Corollary 5.4]{OSZ}}]\label{stade}
Suppose $r>0$, $\bsb \alpha \in \mathbb{C}^{n}$, $\alpha_{\circ} \in \mathbb{C}$ satisfy $\Re\left(\alpha_{i}+\alpha_{\circ}\right)>0$ for all $i$ and  $\Re\left(\alpha_{i}+\alpha_j\right)>0$ for   all $i>j$,  then
   \begin{equation*}\int_{\mathbb{R}_{>0}^{n}}  \left(\prod_{i=1}^n(x_i/r)^{(-1)^{i}}\right)^{\alpha_{\circ}} e^{-r/ x_{n}} \Psi_{\boldsymbol{\alpha}}^{\mathfrak{g} \mathfrak{l}_{n}}(\boldsymbol{x}) \prod_{i=1}^{n} \frac{d x_{i}}{x_{i}} =  r^{-\sum_{i=1}^{n} \alpha_{i}} \prod_{i=1}^n \Gamma\left(\alpha_{i}+\alpha_{\circ}\right) \prod_{ i=1}^n\prod_{j=i+1 }^n \Gamma\left(\alpha_{i}+\alpha_{j}\right).
   \end{equation*} 
\end{proposition}

\begin{figure}
\begin{center}
\begin{tikzpicture}[scale=0.9]
\node (z14) at (0,0) {$z_{1,4}$};
\node (z13) at (2,0) {$z_{1,3}$};
\node (z12) at (4,0) {$z_{1,2}$};
\node (z11) at (6,0) {$z_{1,1}$};

\node (z24) at (1,-1) {$z_{2,4}$};
\node (z23) at (3,-1) {$z_{2,3}$};
\node (z22) at (5,-1) {$z_{2,2}$};
\node (z21) at (7,-1) {$z_{2,1}$};

\node (z34) at (2,-2) {$z_{3,4}$};
\node (z33) at (4,-2) {$z_{3,3}$};
\node (z32) at (6,-2) {$z_{3,2}$};
\node (z31) at (8,-2) {$z_{3,1}$};

\node (z44) at (3,-3) {$z_{4,4}$};
\node (z43) at (5,-3) {$z_{4,3}$};
\node (z42) at (7,-3) {$z_{4,2}$};
\node (z41) at (9,-3) {$z_{4,1}$};
\draw[-> ] (z14) edge (z24) (z24) edge (z34) edge(z13) (z34) edge(z44) edge (z23) (z13) edge (z23) (z23) edge (z12) edge(z33) (z33) edge(z22) edge (z43) (z44) edge (z33) (z12) edge (z22) (z22) edge (z11) edge (z32) (z32) edge (z21) edge (z42) (z42) edge (z31) (z43) edge(z32) (z11) edge (z21) (z21) edge (z31) (z31) edge (z41); 
\end{tikzpicture}
\end{center}
\caption{Tableau of integral variables $(z_{i,j})$ with fixed bottom row  $z_{m+1,j}=x_j$ in the definition of $\mathcal{T}_{\alpha,\bsb{\beta};r}(\bsb{x})$, $n=4$, $m=3$. }
\label{figT}
\end{figure}
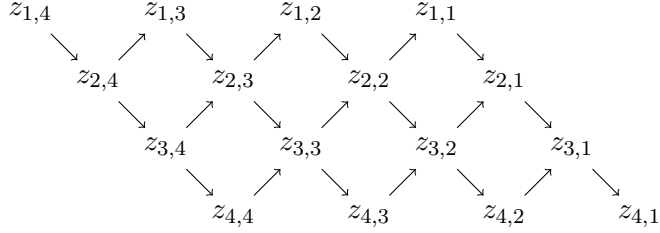

Now we give the Whittaker transform of $
\mathcal{T}_{\alpha,\bsb{\beta};r}(\bsb{x})$. The proof for this formula relies on applying  gRSK
to a symmetric domain not yet introduced, and we postpone the proof to Section \ref{subsec3.2}.
\begin{proposition} \label{prop:whittransformT}
Suppose $r>0$, $\alpha_{\circ}\in\mathbb{C}$, $\bsb{\beta}\in \mathbb{C}^m$, 
then $\mathcal{T}_{\alpha_{\circ},\bsb{\beta};r}(\bsb{x})$ is well-defined for almost  every $\bsb x \in\mathbb R^n_{>0}$. Furthermore, suppose that  $\bsb{\alpha}\in\mathbb{C}^n$ satisfy  $\Re{(\alpha_{\circ}+\alpha_i)}>0$ for all $i\geq 1$, $\Re{(\alpha_i+\alpha_j)}>0$ for all $ 1\leq i<j\leq n$ and $\Re{(\alpha_i+\beta_k)}>0$ for all $ 1\leq i\leq n$, $1\leq k\leq m$. We have
   \begin{equation}\label{eq13}
       \int_{\mathbb{R}^n_{>0}} \mathcal{T}_{\alpha_{\circ},\bsb{\beta};r}(\bsb{x}) \Psi_{\boldsymbol{\alpha}}^{\mathfrak{gl}_n}(\boldsymbol{x}) \prod_{i=1}^n \frac{dx_i}{x_i} = r^{-\sum_{i=1}^n \alpha_i} \prod_{i=1}^n \Gamma(\alpha_i+\alpha_{\circ})\prod_{ i=1}^n\prod_{j=i+1}^n\Gamma(\alpha_i+\alpha_j)\prod_{ i=1}^n\prod_{k=1}^m \Gamma(\alpha_i+\beta_k).
   \end{equation}
\end{proposition}

  \begin{remark}
 Proposition \ref{prop:whittransformT} shows that the function
   $\mathcal T_{\alpha_{\circ},\bsb{\beta};r}(\bsb{x})$ is similar to the function $\mathcal T_{\alpha}^{\tau}(x)$ used in \cite{HalfSpaceMac} to define the half-space Whittaker process, see Definition 6.7 therein. This is not a surprise, as the proof of Proposition \ref{prop:whittransformT} in Section \ref{subsec3.2} involves a symmetrized polymer model equivalent to the half-space polymer model considered in \cite{HalfSpaceMac}. Moreover, one could show by combining Proposition \ref{prop:whittransformT} and \cite[Prop. 6.10]{HalfSpaceMac} that when $r=1$, $\mathcal T_{\alpha_{\circ},\bsb{\beta};r}(\bsb{x})$ is exactly the $\tau\to 0$ limit of the function $\mathcal T_{\alpha}^{\tau}(x)$ defined in \cite{HalfSpaceMac}, provided that parameters are  appropriately matched. Thus, Definition \ref{defT} provides a Givental's type integral representation for the function $\mathcal T$ involved in the definition of half-space Whittaker processes, while it was defined in \cite{HalfSpaceMac} only indirectly through its Whittaker transform. 
  \label{rem:linktohalfspaceMacdo}
  \end{remark}

  \subsubsection{$\mathfrak{so}_{2 n+1}$-Whittaker functions}

We define the $\mathfrak{so}_{2 n+1}$-Whittaker functions also by an integral formula due to \cite{GLO07,GLO08}. The integral is over a half-triangular array (Figure \ref{pat_so}) of depth $2n$,
\[
\boldsymbol{z}=\left(z_{i, j}: 1 \leq i \leq 2 n,  \lfloor i / 2\rfloor+1\leq j\leq i\right).
\]
We   define the type of such an array as 
\[
\operatorname{type}(\boldsymbol{z})_{i}:=\frac{\prod_{j=\lfloor i/2\rfloor +1}^{i} z_{i, j}}{\prod_{j=\lfloor (i-1)/2\rfloor +1}^{i-1} z_{i-1, j}} \quad \text { for } i=1, \ldots, 2 n.
\]
For $\boldsymbol{\alpha}=\left(\alpha_{1}, \ldots, \alpha_{n}\right) \in \mathbb{C}^{n}$, we denote $
\boldsymbol{\alpha}^{\pm}:=\left(-\alpha_{1},\alpha_{1}, -\alpha_{2},\alpha_{2}, \ldots, -\alpha_{n},\alpha_{n}\right)$ and
$$ \operatorname{type}(\boldsymbol{z})^{\bsb{\alpha}^{\pm}}:= \prod_{i=1}^{2n} \operatorname{type}(\bsb{z})_i^{{\alpha}^{\pm}_i}.$$
Now we define the $\gso{2n+1}$-Whittaker function with parameter $\bsb{\alpha}$ via
\begin{equation}\label{eq7.1}
\Psi_{\boldsymbol{\alpha}}^{\mathfrak{s o}_{2 n+1}}(\boldsymbol{x}):=
\int_{ \mathbb{R}_{>0}^{n^2}} \operatorname{type}(\boldsymbol{z})^{\boldsymbol{\alpha}^{\pm}} 
\exp \left(-\sum_{i=1}^{2 n} \sum_{j=\lfloor i/2\rfloor +1}^{i}\frac{z_{i+1, j+1}+z_{i-1,j}}{z_{i, j}}-\sum_{k=1}^n z_{2k-1,k}\right)
\prod_{\substack{1 \leq i<2 n \\
\lfloor i/2 \rfloor<j\leq i}} \frac{{d} z_{i, j}}{z_{i, j}},
\end{equation}
where on the bottom row $z_{2n,j+\lfloor i/2\rfloor}=x_j$ for $1\leq j\leq \lceil i/2 \rceil$ and by convention $z_{i, j}:=0$ whenever $i,j$ is out of the range $1 \leq i \leq 2 n,  \lfloor i / 2\rfloor+1\leq j\leq i$.
We conclude this subsection with the Whittaker transform of $\gso{2n+1}$-Whittaker functions.
\begin{proposition}[{\cite[Lemma 3.14]{BisiZygouras}}]\label{TransSO}
   Suppose $\mu\in\mathbb{C}$ and $\bsb{\alpha}\in\mathbb{C}^n$ are such that $\Re(\mu)>\max_j\left|\Re\left(\alpha_{j}\right)\right|$. Let
\begin{equation*}
    g(\boldsymbol{x}):=\left(\prod_{i=1}^{n} x_{i}\right)^{\mu} \Psi_{\bsb{\alpha}}^{\mathfrak{s} \mathfrak{0}_{2 n+1}}(\boldsymbol{x}) ,
\end{equation*}
then $g\in L^{2}\left(\mathbb{R}_{>0}^{n}, \prod_{i=1}^{n} {~d} x_{i} / x_{i}\right)$ and 
   \begin{equation*}\label{eq5}
    \widehat{g}(\boldsymbol{\lambda}):=\frac{\prod_{1 \leq i, j \leq n} \Gamma\left(\mu-\lambda_{i}+\alpha_{j}\right) \Gamma\left(\mu-\lambda_{i}-\alpha_{j}\right)}{\prod_{1 \leq i<j \leq n} \Gamma\left(2 \mu-\lambda_{i}-\lambda_{j}\right)}.
\end{equation*}
\end{proposition}
\begin{remark}
Several conventions have been used in the literature in order to define Whittaker functions. Our definition for $\Psi_{\boldsymbol{\alpha}}^{\mathfrak{g l}_{n}}(\boldsymbol{x})$ and $\Psi_{\boldsymbol{\beta}; r}^{\mathfrak{g l}_{n}}(\boldsymbol{x})$ are the same as in \cite{OSZ}, but $\Psi_{\boldsymbol{\alpha}}^{\mathfrak{g l}_{n}}(\boldsymbol{x}) $ corresponds to $\Psi_{\boldsymbol{-\alpha}}^{\mathfrak{g l}_{n}}(\boldsymbol{x}) = \Psi_{\boldsymbol{\alpha}}^{\mathfrak{g l}_{n}}(\boldsymbol{x}') $ in \cite{BisiZygouras} (for $ \boldsymbol{x}=(x_1,\dots, x_n)$, we use the notation $\boldsymbol{x}'=(1/x_n, 1/x_{n-1}, \dots, 1/x_1)$). In \cite{GLO07, GLO08, MacdoProcesses, HalfSpaceMac}, Whittaker functions are defined in terms of the variables $\log(x_i)$ and the parameters $\alpha_{\circ}, \alpha_i, \beta_j$ are sometimes multiplied by $\pm\mathrm i$. 

Our definition of orthogonal Whittaker functions is such that $\Psi_{\bsb{\alpha}}^{\mathfrak{s} \mathfrak{0}_{2 n+1}}(\boldsymbol{x})$ corresponds to $ \Psi_{\bsb{\alpha}}^{\mathfrak{s} \mathfrak{0}_{2 n+1}}(\boldsymbol{x}')$ in \cite{BisiZygouras}. In \cite{GLO07, GLO08}, orthogonal Whittaker functions are defined in terms of the variables $\log(x_i)$.
\end{remark}

   \section{Laplace transform of the partition function for trapezoidal log-gamma polymers}\label{sec_proof}
   This section is  devoted to proving an integral formula for the Laplace transform of $Z^{\rflat}(n;m)$ (Theorem \ref{LapTrfmFl}). We then deduce a  contour integral formula (Corollary \ref{LapTransFla})  from Theorem \ref{LapTrfmFl} by applying the Plancherel formula \eqref{plancherel}. This contour integral formula will enable us, in Section \ref{sec4},  to compare the Laplace transform of $Z^{\rflat}(n;m)$ with that of $Z(n,m)$  and directly prove Theorem \ref{main} for the $m\geq n-1$ case.

  In order to apply  the Plancherel formula, the Whittaker transform of $\mathcal{T}_{\alpha_{\circ},\bsb{\beta};r}$ (Proposition \ref{prop:whittransformT}) will be needed. We compute it in Subsection \ref{subsec3.2},  using a similar approach as  in the  proof of Theorem \ref{LapTrfmFl}.

\subsection{Laplace transform formula for the partition function of trapezoidal log-gamma polymers}\label{sec_3.2}
By using symmetric gRSK and $\gso{2n+1}$--Whittaker functions,   we derive the Laplace transform of $Z^{\rflat}(n;m)$, using a generalization of similar arguments in \cite{OSZ,BisiZygouras}.  
\begin{theorem}\label{LapTrfmFl} 
Assume that parameters $n,m,\alpha_{\circ },\bsb \alpha, \bsb \beta$ and the partition function $  Z^{\rflat}(n;m)$ are   as in Definition \ref{nota}. Assume $\mathcal{T}_{\alpha_{\circ},\bsb{\beta};r}(\bsb{x})$ is as    in Definition \ref{defT}.  We have
\begin{align}\nonumber
\mathbb{E} & \left[e^{-rZ^{\rflat}(n;m)}\right]=\\
& \qquad \frac{r^{\sum_{k=1}^{n}\alpha_{k}}}{ \prod_{ i=1}^n \Gamma\left(\alpha_{i}+\alpha_{\circ} \right) \prod_{ i=1}^n\prod_{j=1}^{n} \Gamma\left(\alpha_{i}+\alpha_{j} \right)\prod_{ i=1}^n\prod_{k=1}^m\Gamma(\alpha_i+\beta_k)} 
\int_{\mathbb{R}_{>0}^{n}} \mathcal{T}_{\alpha_{\circ},\bsb{\beta};r}(\bsb{x}) 
\Psi_{\boldsymbol{\alpha}}^{\mathfrak{s} \mathfrak{o}_{2 n+1}}(\boldsymbol{x})  \prod_{i=1}^{n}  \frac{{d} x_{i}}{x_{i}}.
 \label{eq1.5}
\end{align} 
\end{theorem}

\begin{proof}
Let $\mathcal{I}$ be the trapezoidal  index set  
\[\mathcal{I}:=\{(i,j)\mid  1\leq i\leq n,  i\leq j\leq 2n+m-i+1\}.\] 
As explained in (\ref{eq3.1}), $Z^{\rflat}(n;m)$ is equal in distribution to $Z^{\symflat}(n;m)$, whose  weight array   $\bsb{W}$  (note that $ \bsb W$ corresponds to  $\widetilde{\bsb W}$ in the notation of \eqref{eq3.1})  is defined on  $\mathcal{I}\cup\mathcal{I}^t$ via 
\begin{equation}\label{eq8.5}
W_{i,j}=W_{j,i}\, \, \text{and}\,\,
W_{i, j}^{-1} \sim\left\{\begin{array}{ll}
2\operatorname{Gamma}\left(\alpha_{i}+\alpha_{\circ}\right), & 1 \leq i=j \leq n, \\
\operatorname{Gamma}\left(\alpha_{i}+\alpha_{j}\right), & 1 \leq i<j \leq n, \\
\operatorname{Gamma}\left(\alpha_i+\beta_{j-n}\right), & 1\leq i\leq n, n<j\leq n+m, \\
\operatorname{Gamma}\left(\alpha_{i}+\alpha_{2 n+m-j+1}\right) ,& 1 \leq i \leq n,
i\leq 2n+m-j+1\leq n.
\end{array}\right. 
\end{equation}
Now we calculate the Laplace transform of $Z^{\symflat}$.
To this end,   we will derive the joint distribution of $\bsb{T}=\operatorname{gRSK}(\bsb{W})$,   whose outer entries  coincide with  the   point-to-point partition functions, that is,   $T_{k,2n+m-k+1}=Z(k,2n+m-k+1)$   and $T_{2n+m-k+1,k}=Z(2n+m-k+1,k)$ for $k=1,\dots,n$ (see Proposition \ref{gRSKprop}). The Laplace transform of $Z^{\symflat}$ will then be expressed in terms of this joint distribution. 

Let $\bsb{w}$ be a symmetric  polygonal array   on $\mathcal{I}\cup \mathcal{I}^t$.
By (\ref{eq8.5}), the density of ${(W_{i,j})}_{i\leq j}$ at ${(w_{i,j})}_{i\leq j}$ is given by $L_1\times L_2$, where 
\begin{align*}
    L_1:= 
    &\prod_{i=1}^n  \frac{(2w_{i,i})^{-\alpha_i-\alpha_{\circ}}}{\Gamma(\alpha_i+\alpha_{\circ})}
   \prod_{1\leq i<j\leq n}\frac{w_{i,j}^{-\alpha_i-\alpha_j}}{\Gamma(\alpha_i+\alpha_j)} \\
    & \qquad\times \prod_{i=1}^n\prod_{j=n+1}^{n+m} \frac{w_{i,j}^{-\alpha_i-\beta_{j-n}}}{\Gamma(\alpha_i+\beta_{j-n})}
   \prod_{i=1}^n\prod_{j=n+m+1}^{2n+m-i+1}\frac{w_{i,j}^{-\alpha_i-\alpha_{2n+m-j+1}}}{\Gamma(\alpha_i+\alpha_{2n+m-j+1})},   \\
  L_2:= &  \exp\left(-\frac{1}{2}\sum_{i=1}^n\frac{1}{w_{i,i}}-\sum_{\substack{i<j  }}\frac{1}{w_{i,j}}
     \right)
     \prod_{(i,j)\in\mathcal{I}}\frac{1}{w_{i,j}}.
 \end{align*}
 In the above expression  we have written under the summation symbol  $i<j$ instead of   $(i,j)\in\mathcal{I}\cap\{i<j\}$ for simplicity of notation. We will use this convention in several other places in the proof, which should not cause too much confusion.  
 Using the symmetry of  $\bsb{w}$,
 we gather  $w_{i,j}$'s in $L_1$ according to their exponent   to obtain
 \begin{align*}
    L_1={} &C\prod_{i=1}^n  w_{i,i}^{-\alpha_{\circ}} 
     \prod_{i=1}^n \prod_{j= i}^{2n+m-i+1} w_{i,j} ^{-\alpha_i} \prod_{i=1}^n\prod_{j=i+1}^n w_{i,j}^{-\alpha_j}
     \prod_{i=1}^n\prod_{j=n+1}^{n+m} w_{i,j}^{ -\beta_{j-n}}
     \prod_{i=1}^n\prod_{j=n+m+1}^{2n+m-i+1}w_{i,j}^{-\alpha_{2n+m-j+1}}  \\
     ={} &C
     \prod_{i=1}^n  w_{i,i}^{-\alpha_{\circ}} 
     \prod_{i=1}^n\prod_{j= i}^{2n+m-i+1} w_{i,j} ^{-\alpha_i} \prod_{j=1}^n\prod_{i=j+1}^n w_{j,i}^{-\alpha_i}
     \prod_{j=1}^{m}\prod_{i=1}^n w_{i,j+n}^{ -\beta_{j}}
     \prod_{i=1}^n\prod_{j=i}^{n}w_{i,2n+m-j+1}^{-\alpha_{j}}  \\ 
     ={}&C
     \left(\prod_{i=1}^n   w_{i,i}\right)^{-\alpha_{\circ}} 
     \prod_{i=1}^n\left(\prod_{j=1}^{2n+m-i+1}w_{i,j}\right)^{-\alpha_i}
     \prod_{j=1}^{m}\left(\prod_{i=1}^n w_{i,j+n}\right)^{ -\beta_{j}}
     \prod_{j=1}^n\left(\prod_{i=1}^{j}w_{i,2n+m-j+1}\right)^{-\alpha_{j}} ,
 \end{align*}
 where
 \[
 C = \left(\prod_{i=1}^n2^{\alpha_i+\alpha_{\circ}}\Gamma(\alpha_i+\alpha_{\circ}) \prod_{ i=1}^n\prod_{j=1}^n\Gamma(\alpha_i+\alpha_j)\prod_{ i=1}^n\prod_{k=1}^m \Gamma(\alpha_i+\beta_k)\right)^{-1}.
 \]
 Again, by symmetry of $\bsb{w}$ we have
 \[
 L_2 = \exp\left(-\frac{1}{2}\sum_{(i,j)\in\mathcal{I}\cup\mathcal{I}^t}\frac{1}{{w_{i,j}}}      \right)
     \prod_{(i,j)\in\mathcal{I}}\frac{1}{w_{i,j}}.
 \] 
 Now we do a change of variables via gRSK. Let   
 $$\bsb{T}:=\operatorname{gRSK}(\bsb{W})\quad
\text{ and   }\quad
 \bsb{t}:=\operatorname{gRSK}(\bsb{w}).$$
 By Proposition \ref{prop2}, and observing that $n$ is the diagonal length of $\bsb{w}$,
 \begin{equation}\label{eq11.1}
     \prod_{i=1}^n   w_{i,i} = 4^{-\lfloor n/2\rfloor}\prod_{j=1}^{n}t_{j,j}^{(-1)^{n-j}}.
 \end{equation}
 By  Proposition \ref{gRSKprop} , 
 \begin{equation}\label{eq11.2}
     \prod_{j=1}^{2n+m-i+1}w_{i,j} = \frac{\tau_{2n+m-2i+1}}{\tau_{2n+m-2i+2}},\quad
     \prod_{i=1}^jw_{i,2n+m-j+1} = \frac{\tau_{2n+m-2j+1}}{\tau_{2n+m-2j}} \quad \text{and}\quad  \prod_{i=1}^n w_{i,j+n} = \frac{\tau_j}{\tau_{j-1}},
 \end{equation}
 where
 \begin{equation*}
     \tau_k = \prod_{\substack{j-i=k  \\(i,j)\in\mathcal{I}\cup\mathcal{I}^t}} t_{i,j}.
 \end{equation*}
 Using Proposition \ref{gRSKprop}, 
 we convert the sum in the expression of $L_2$  to
 \begin{align}
     \sum_{i,j}w_{i,j}^{-1}
     &= \frac{1}{t_{1,1}}+\sum_{(i,j)\neq (1,1)} \frac{t_{i-1, j}+t_{i, j-1}}{t_{i, j}}  = \frac{1}{t_{1,1}}+\sum_{1<i=j} \frac{t_{i-1, j}+t_{i, j-1}}{t_{i, j}} +2\sum_{i<j} \frac{t_{i-1, j}+t_{i, j-1}}{t_{i, j}} \nonumber \\
    & = \frac{1}{t_{1,1}}+2\sum_{1<i=j} \frac{t_{i-1, j}}{t_{i, j}} +2\sum_{i<j} \frac{t_{i-1, j}+t_{i, j-1}}{t_{i, j}}   = \frac{1}{t_{1,1}}+2\sum_{1<i\leq j} \frac{t_{i-1, j}}{t_{i, j}} +2\sum_{i<j} \frac{t_{i, j-1}}{t_{i, j}}
    \label{eq11.31},
 \end{align}
 where we omitted the constraint $(i,j)\in\mathcal{I}\cup\mathcal{I}^t$ to ease  notation.
 Inserting   (\ref{eq11.1}), (\ref{eq11.2}) (resp. (\ref{eq11.31})) into the expression of $L_1$ (resp. $L_2$),  and recalling that  $(\log w_{i,j}) \mapsto (\log t_{i,j}) $ has Jacobian $\pm 1$ (see Prop. \ref{jacob}),  we obtain the density of ${(T_{i,j})}_{i\leq j}$ at ${(t_{i,j})}_{i\leq j}$ as
 \begin{align*}
     C \left(4^{-\lfloor n/2\rfloor}\prod_{j=1}^{n}t_{j,j}^{(-1)^{n-j}}\right )^{-\alpha_{\circ}} &\prod_{i=1}^n\left(\frac{\tau_{2n+m-2i+1}^2}{\tau_{2n+m-2i+2}\tau_{2n+m-2i}}\right)^{-\alpha_i}  \prod_{j=1}^m\left(\frac{\tau_j}{\tau_{j-1}} \right)^{-\beta_j} \\
     \times&\exp\left(-\frac{1}{2 t_{1,1}}-\sum_{1<i\leq j} \frac{t_{i-1, j}}{t_{i, j}} -\sum_{i<j} \frac{t_{i, j-1}}{t_{i, j}}   \right)
     \prod_{(i,j)\in\mathcal{I}}\frac{1}{t_{i,j}}.
 \end{align*}
 The density of $(2T_{i,j})_{i\leq j}$ at $(t_{i,j})_{i\leq j}$ takes a clearer form 
  \begin{align*}
     C' \left( \prod_{j=1}^{n}t_{j,j}^{(-1)^{n-j}}\right )^{-\alpha_{\circ}} &\prod_{i=1}^n\left(\frac{\tau_{2n+m-2i+1}^2}{\tau_{2n+m-2i+2}\tau_{2n+m-2i}}\right)^{-\alpha_i}  \prod_{j=1}^m\left(\frac{\tau_j}{\tau_{j-1}} \right)^{-\beta_j}\\
     \times&\exp\left(-\frac{1}{ t_{1,1}}-\sum_{1<i\leq j} \frac{t_{i-1, j}}{t_{i, j}} -\sum_{i<j} \frac{t_{i, j-1}}{t_{i, j}}   \right)
     \prod_{(i,j)\in\mathcal{I}}\frac{1}{t_{i,j}}.
 \end{align*}
 where
 \begin{equation*} C '= \left(\prod_{i=1}^n \Gamma(\alpha_i+\alpha_{\circ}) \prod_{ i=1}^n\prod_{j=1}^n\Gamma(\alpha_i+\alpha_j)\prod_{ i=1}^n\prod_{k=1}^m \Gamma(\alpha_i+\beta_k)\right)^{-1} .
 \end{equation*}
 Note that 
  $$Z^{\symflat}(n;m)=\sum_{i=1}^{n}T_{i,2n+m-i+1}+ \sum_{i=1}^{n}T_{2n+m-i+1,i} = 2\sum_{i=1}^{n}T_{i,2n+m-i+1}.$$
 We may therefore calculate its Laplace transform 
 \begin{align*}
     \mathbb{E}&\left[e^{-r Z^{\symflat}(n;m)}\right]  =\mathbb{E}\left[e^{-r \sum_{i=1}^n 2T_{i,2n+m-i+1}}\right] \\
     &= C' \int_{\mathbb{R }_{>0}^{n(n+m+1)}}  \left(  \prod_{j=1}^{n}t_{j,j}^{(-1)^{n-j}} \right )^{-\alpha_{\circ}} \prod_{i=1}^n\left(\frac{\tau_{2n+m-2i+1}^2}{\tau_{2n+m-2i+2}\tau_{2n+m-2i}}\right)^{-\alpha_i}    
     \prod_{j=1}^m\left(\frac{\tau_j}{\tau_{j-1}} \right)^{-\beta_j}\\
    & \times\exp\left(-\frac{1}{ t_{1,1}}-\sum_{i=2}^n\sum_{j=i}^{2n+m-i+1} \frac{t_{i-1, j}}{t_{i, j}} -\sum_{i=1}^n\sum_{j=i+1}^{2n+m-i+1} \frac{t_{i, j-1}}{t_{i, j}} - r \sum_{i=1}^n t_{i,2n+m-i+1}  \right)
     \prod_{(i,j)\in\mathcal{I}}\frac{dt_{i,j}}{t_{i,j}}.
 \end{align*} 
 By the change of variables $t_{i,j} = r^{-1}{s_{i,j}}$, we obtain  
  \begin{align}
     \mathbb{E}&\left[e^{-r Z^{\symflat}(n;m)}\right]  
     ={} C' r^{\sum_{i=1}^n \alpha_i}\int_{\mathbb{R }_{>0}^{n(n+m+1)}} \left(  \prod_{j=1}^{n}(r/s_{j,j})^{(-1)^{n-j}} \right )^{\alpha_{\circ}}  
      \prod_{i=1}^n\left(\frac{\tilde{\tau}_{2n+m-2i+1}^2}{\tilde{\tau}_{2n+m-2i+2}\tilde{\tau}_{2n+m-2i}}\right)^{-\alpha_i} \nonumber \\
        \times &
    \prod_{j=1}^m  \left(\frac{\tilde{\tau}_j}{\tilde{\tau}_{j-1}} \right)^{-\beta_j}
   \exp\left(-\frac{r}{s_{1,1}}-\sum_{i=2}^n\sum_{j=i}^{2n+m-i+1} \frac{s_{i-1, j}}{s_{i, j}} -\sum_{i=1}^n\sum_{j=i+1}^{2n+m-i+1} \frac{s_{i, j-1}}{s_{i, j}} - \sum_{i=1}^n  {s_{i,2n+m-i+1}}  \right)
     \prod_{(i,j)\in\mathcal{I}}\frac{ds_{i,j}}{s_{i,j}} , \label{eq51}
 \end{align}
where as before 
\[\tilde{\tau}_k = \prod_{\substack{j-i=k  \\(i,j)\in\mathcal{I}\cup\mathcal{I}^t}} s_{i,j}.
\]
We observe that the variables $(s_{i,j})_{\substack{1\leq i\leq n \\ i\leq j\leq i+m-1}}$ and $(s_{i,j})_{\substack{1\leq i\leq n \\   i+m+1\leq j\leq 2n+m-i+1}}$ are not coupled in the integrand. Using the change of variables $s_{i,j}=z_{j-i+1,n-i+1}$ for $(s_{i,j})_{\substack{1\leq i\leq n \\ i\leq j\leq i+m-1}}$, we obtain 
 \begin{align}\nonumber
    \mathcal{T}_{\alpha_{\circ},\bsb{\beta};r}(s_{n,n+m},\dots,s_{1,1+m})= \int_{\mathbb{R }_{>0}^{nm}} & \prod_{i=1}^n\prod_{j=i}^{i+m-1}\frac{ds_{i,j}}{s_{i,j}}
     \left( \prod_{j=1}^{n}(r/s_{j,j})^{(-1)^{n-j}} \right )^{\alpha_{\circ}}  
     \prod_{j=1}^m\left(\frac{\tilde{\tau}_j}{\tilde{\tau}_{j-1}} \right)^{-\beta_j} \\
     &\times\exp\left(-\frac{r}{s_{1,1}}-\sum_{i=2}^n\sum_{j=i}^{i+m-1}  \frac{s_{i-1, j}}{s_{i, j}} -\sum_{i=1}^n\sum_{j=i}^{i+m-1}  \frac{s_{i, j}}{s_{i, j+1}}   \right). \label{eq110}
 \end{align}
 Similarly, using the change of variables $s_{i,j}=z_{i-j+2n+m,-j+2n+m+1}$ for $(s_{i,j})_{\substack{1\leq i\leq n \\   i+m+1\leq j\leq 2n+m-i+1}}$, we see that
 \begin{align}\nonumber 
 \Psi^{\gso{2n+1}}_{\bsb{\alpha}}(s_{n,n+m},\dots,s_{1,1+m})
   =  &\int_{\mathbb{R}_{>0}^{n^2}}
      \prod_{i=1}^n\prod_{j=i+m+1}^{2n+m-i+1}\frac{ds_{i,j}}{s_{i,j}} 
      \prod_{i=1}^n\left(\frac{\tilde{\tau}_{2n+m-2i+1}^2}{\tilde{\tau}_{2n+m-2i+2}\tilde{\tau}_{2n+m-2i}}\right)^{-\alpha_i}  \\
      \times 
     \exp&\left( -\sum_{i=2}^n\sum_{j=i+m}^{2n+m-i+1} \frac{s_{i-1, j}}{s_{i, j}} -\sum_{i=1}^n\sum_{j=i+m}^{2n+m-i} \frac{s_{i, j}}{s_{i, j+1}} - \sum_{i=1}^n  {s_{i,2n+m-i+1}}  \right). \label{eq111}
 \end{align}
Inserting \eqref{eq110} and \eqref{eq111} into \eqref{eq51}, we obtain
\begin{align*}
     \mathbb{E}\left[e^{-r Z^{\symflat}(n;m)}\right]  
&={} C' r^{\sum_{i=1}^n \alpha_i} \int_{\mathbb{R}_{>0}^{n}} \mathcal{T}_{\alpha_{\circ},\bsb{\beta};r}(\bsb{x})
\Psi_{\boldsymbol{\alpha}}^{\mathfrak{s} \mathfrak{o}_{2 n+1}}(\boldsymbol{x})  \prod_{i=1}^{n}  \frac{{d} x_{i}}{x_{i}} ,
\end{align*}
which completes the proof of Theorem \ref{LapTrfmFl}.
\end{proof}

   \subsection{Whittaker transform of $\mathcal{T}_{\alpha,\bsb{\beta};r}$}\label{subsec3.2}
   
   We prove in this subsection Proposition \ref{prop:whittransformT}, which is required for   transforming \eqref{eq1.5} into \eqref{eq2.5}.  The proof is similar to the proof for Theorem \ref{LapTrfmFl}. But instead of considering a $\symflat$--shape
   domain, we consider here the $\symtrapz$--shape index set $\mathcal{I}\cup\mathcal{I}^t$ where
\begin{equation}\label{eq53}
    \mathcal{I}:=\{(i,j)\mid  1\leq i\leq n,  i\leq j\leq n+m\}.
\end{equation}
Let  the   weight array  $\bsb{W}$  be defined on     $\mathcal{I}\cup\mathcal{I}^t$   via
\begin{equation*}\label{eq8.51}
W_{i,j}=W_{j,i}\quad  \text{and}\quad
W_{i, j}^{-1} \sim\left\{\begin{array}{ll}
2\operatorname{Gamma}\left(\alpha_{i}+\alpha_{\circ}\right) ,& 1 \leq i=j \leq n, \\
\operatorname{Gamma}\left(\alpha_{i}+\alpha_{j}\right), & 1 \leq i<j \leq n, \\
\operatorname{Gamma}\left(\alpha_i+\beta_{j-n}\right) , & 1\leq i\leq n, n<j\leq n+m.
\end{array}\right. 
\end{equation*}
We readily see that the joint probability density function of $(W_{i,j})_{i\leq j}$ at $(w_{i,j})_{i\leq j}$ can be written as $L_1\times L_2$, where
\begin{align} \label{eq55}
    L_1 :=  & \prod_{i=1}^n \frac{(2w_{i,i})^{-\alpha_i-\alpha_{\circ}}}{\Gamma(\alpha_i+\alpha_{\circ})}
   \prod_{1\leq i<j\leq n}\frac{w_{i,j}^{-\alpha_i-\alpha_j}}{\Gamma(\alpha_i+\alpha_j)}
   \prod_{i=1}^n\prod_{j=n+1}^{n+m} \frac{w_{i,j}^{-\alpha_i-\beta_{j-n}}}{\Gamma(\alpha_i+\beta_{j-n})},
      \\ 
      \label{eq57}
   L_2 := &  \exp\left(-\frac{1}{2}\sum_{i=1}^n\frac{1}{w_{i,i}}-\sum_{i<j}\frac{1}{w_{i,j}}
     \right)
     \prod_{i\leq j}\frac{1}{w_{i,j}}.
 \end{align}
 In \eqref{eq57}, we have written $i<j $ under summation symbol instead of $ (i,j)\in \mathcal{I}\cap\{i<j\}$ for simplicity of notation.
 Since $L_1L_2$ is a probability density, 
 \begin{equation}\label{eq18}
 \int_{\mathbb{R}_{>0}^{n(m+(n+1)/2)}} L_1 L_2 \prod_{i,j} dw_{i,j} =1.
 \end{equation}
 Observe that, by definition of Gamma functions, (\ref{eq18}) still holds  when the parameters are complex numbers chosen with real parts as in the hypothesis of Proposition \ref{prop:whittransformT}, even though the probabilistic interpretation is lost. In the following we work directly with (\ref{eq18}). 
 
\begin{proof}[Proof of Proposition \ref{prop:whittransformT}]

Let $\bsb{w}$ be a symmetric array   defined on $\mathcal{I}\cup\mathcal{I}^t$ with $\mathcal{I}$ being as in (\ref{eq53}). Let $L_1,L_2$ be as in (\ref{eq55}), (\ref{eq57}). We rewrite  $L_1$ by  gathering the $w_{i,j}$'s  according to their exponent   
 \begin{align*}
    L_1={} &C\prod_{i=1}^n  w_{i,i}^{-\alpha_{\circ}} 
     \prod_{i=1}^n \prod_{j= i}^{n+m} w_{i,j} ^{-\alpha_i} \prod_{i=1}^n\prod_{j=i+1}^n w_{i,j}^{-\alpha_j}
     \prod_{i=1}^n\prod_{j=n+1}^{n+m} w_{i,j}^{ -\beta_{j-n}}  \\
     ={} &C
     \prod_{i=1}^n  w_{i,i}^{-\alpha_{\circ}} 
     \prod_{i=1}^n\prod_{j= i}^{n+m} w_{i,j} ^{-\alpha_i} \prod_{j=1}^n\prod_{i=j+1}^n w_{j,i}^{-\alpha_i}
     \prod_{j=1}^{m}\prod_{i=1}^n w_{i,j+n}^{ -\beta_{j}}  \\ 
     ={}&C
     \left(\prod_{i=1}^n   w_{i,i}\right)^{-\alpha_{\circ}} 
     \prod_{i=1}^n\left(\prod_{j=1}^{n+m}w_{i,j}\right)^{-\alpha_i}
     \prod_{j=1}^{m}\left(\prod_{i=1}^n w_{i,j+n}\right)^{ -\beta_{j}}
       ,
 \end{align*}
 where
 \[
 C = \left(\prod_{i=1}^n2^{\alpha_i+\alpha_{\circ}}\Gamma(\alpha_i+\alpha_{\circ}) \prod_{ i=1}^n\prod_{j=i+1}^n\Gamma(\alpha_i+\alpha_j)\prod_{ i=1}^n\prod_{k=1}^m \Gamma(\alpha_i+\beta_k)\right)^{-1}.
 \]
 By symmetry of $\bsb{w}$ we have
 \[
 L_2 = \exp\left(-\frac{1}{2}\sum_{(i,j)\in\mathcal{I}\cup\mathcal{I}^t}\frac{1}{{w_{i,j}}}      \right)
     \prod_{(i,j)\in\mathcal{I}}\frac{1}{w_{i,j}}.
 \] 
 Now we make a change of variables via gRSK. Let   
 $\bsb{T}:=\operatorname{gRSK}(\bsb{W} )$
 and   
 $\bsb{t}:=\operatorname{gRSK}(\bsb{w} )$.
 By Proposition \ref{prop2}, 
 \begin{equation}\label{eq11.11}
     \prod_{i=1}^n   w_{i,i} = 4^{-\lfloor n/2\rfloor}\prod_{j=1}^{n}t_{j,j}^{(-1)^{n-j}}.
 \end{equation}
 By  Proposition \ref{gRSKprop} , 
 \begin{equation}\label{eq11.21}
     \prod_{j=1}^{n+m }w_{i,j} = \frac{\tau_{n+m-i}}{\tau_{n+m-i+1}}\quad  \text{and}\quad  \prod_{i=1}^n w_{i,j+n} = \frac{\tau_j}{\tau_{j-1}} \quad \text{where}\quad   \tau_k = \prod_{\substack{j-i=k  \\(i,j)\in\mathcal{I}\cup\mathcal{I}^t}} t_{i,j}.
 \end{equation} 
 As in (\ref{eq11.31}),  the sum in $L_2$ can be converted to (here the indices $i,j$ are in $\mathcal{I}\cup\mathcal{I}^t$)
 \begin{align}
     \sum_{i,j}w_{i,j}^{-1}   = \frac{1}{t_{1,1}}+2\sum_{1<i\leq j} \frac{t_{i-1, j}}{t_{i, j}} +2\sum_{i<j} \frac{t_{i, j-1}}{t_{i, j}}.
    \label{eq11.3}
 \end{align}
 Inserting   (\ref{eq11.11}), (\ref{eq11.21}), (\ref{eq11.3}), (\ref{eq55}),   (\ref{eq57})  into (\ref{eq18}), and recalling that  $(\log w_{i,j}) \mapsto (\log t_{i,j}) $ is of Jacobian $\pm 1$ (Proposition \ref{jacob}),  we obtain  
 \begin{align*}
  \nonumber 1= &\int_{\mathbb{R}_{>0}^{n(m+(n+1)/2)}} L_1 L_2 \prod_{i,j} dw_{i,j}\\ \nonumber
   =  {}&C\int_{\mathbb{R}_{>0}^{n(m+(n+1)/2)}}  \left(4^{-\lfloor n/2\rfloor}\prod_{j=1}^{n}t_{j,j}^{(-1)^{n-j}}\right )^{-\alpha_{\circ}} \prod_{i=1}^n\left(\frac{\tau_{n+m-i}}{\tau_{n+m-i+1}}\right)^{-\alpha_i}  \prod_{j=1}^m\left(\frac{\tau_j}{\tau_{j-1}} \right)^{-\beta_j}\\
     &\qquad\times\exp\left(-\frac{1}{2 t_{1,1}}-\sum_{1<i\leq j} \frac{t_{i-1, j}}{t_{i, j}} -\sum_{i<j} \frac{t_{i, j-1}}{t_{i, j}}   \right)
     \prod_{(i,j)\in\mathcal{I}}\frac{dt_{i,j}}{t_{i,j}}. 
 \end{align*}
  By the change of variables $2t_{i,j}=r^{-1}s_{i,j} $,
  \begin{align}
 \int_{\mathbb{R}_{>0}^{n(m+(n+1)/2)}} &\left( \prod_{j=1}^{n}(r/s_{j,j})^{(-1)^{n-j}}\right )^{\alpha_{\circ}} 
 \prod_{i=1}^n\left(\frac{\tilde{\tau}_{n+m-i}}{\tilde{\tau}_{n+m-i+1}}\right)^{-\alpha_i} 
  \prod_{j=1}^m\left(\frac{\tilde{\tau}_j}{\tilde{\tau}_{j-1}} \right)^{-\beta_j}  \nonumber\\ 
    &\times\exp\left(-\frac{r}{s_{1,1}}-\sum_{1<i\leq j} \frac{s_{i-1, j}}{s_{i, j}} -\sum_{i<j} \frac{s_{i, j-1}}{t_{i, j}}   \right)
     \prod_{(i,j)\in\mathcal{I}}\frac{ds_{i,j}}{s_{i,j}}  \nonumber\\
     ={}&r^{\sum_{i=1}^n -\alpha_i}\prod_{i=1}^n \Gamma(\alpha_i+\alpha_{\circ}) \prod_{ i=1}^n\prod_{j=i+1}^n\Gamma(\alpha_i+\alpha_j)\prod_{ i=1}^n\prod_{k=1}^m \Gamma(\alpha_i+\beta_k), \label{eq62} 
 \end{align}
 where 
 $$\tilde{\tau}_q:=\prod_{
 \substack{j-i=q \\ (i,j)\in\mathcal{I}\cup\mathcal{I}^t}
 }s_{i,j}.$$ 
Applying Fubini's theorem to \eqref{eq62}, we obtain that  the partial integral
 \begin{align}  \nonumber
 \int_{\mathbb{R}_{>0}^{nm }}   \left( \prod_{j=1}^{n}(r/s_{j,j})^{(-1)^{n-j}}\right )^{\alpha_{\circ}} &\prod_{j=1}^m\left(\frac{\tilde{\tau}_j}{\tilde{\tau}_{j-1}} \right)^{-\beta_j} \\
  \times \exp &\left(-\frac{r}{s_{1,1}}-\sum_{\substack{1<i \leq j \\ j \leq i+m-1}  } \frac{s_{i-1, j}}{s_{i, j}} -\sum_{i<j\leq i+m} \frac{s_{i, j-1}}{s_{i, j}}   \right)
  \prod_{i=1}^n \prod_{j=i}^{i+m-1} \frac{ds_{i,j}}{s_{i,j}}
  \label{eq100}
 \end{align}
 is convergent  for almost every $(s_{i,j})_{ \substack{(i,j)\in \mathcal{I} \\ j\geq i+m}}\in\mathbb R^{n(n+1)/2}_{>0}$. By definition of $\mathcal{T}_{\alpha_{\circ},\bsb{\beta};r}$, we recognize that \eqref{eq100} is  in fact equal to $ \mathcal{T}_{\alpha_{\circ},\bsb{\beta};r}(s_{n,n+m},\dots,s_{1,1+m})$ (cf. \eqref{eq110}). Thus $\mathcal{T}_{\alpha_{\circ},\bsb{\beta};r}(\bsb x)$ is well-defined for a.e.  $\bsb x \in\mathbb R^n_{>0}$, which concludes the proof of the first claim of Proposition \ref{prop:whittransformT}.
 
 By definition of   $\Psi_{\boldsymbol{\alpha}}^{\mathfrak{gl}_n}$ and using the change of variables $s_{i,j}=z_{i-j+m+n,-j+m+n+1}$, we obtain (cf. \eqref{eq111})
 \begin{align*}
   & \Psi_{\boldsymbol{\alpha}}^{\mathfrak{gl}_n}(s_{n,n+m},\dots,s_{1,1+m})=\\
     & \quad \int_{\mathbb{R}_{>0}^{n((n-1)/2)}} \prod_{i=1}^n\prod_{j=i+m+1}^{n+m}\frac{ds_{i,j}}{s_{i,j}}
 \prod_{i=1}^n\left(\frac{\tilde{\tau}_{n+m-i}}{\tilde{\tau}_{n+m-i+1}}\right)^{-\alpha_i} 
 \exp\left(-\sum_{1<i\leq j-m} \frac{s_{i-1, j}}{s_{i, j}} -\sum_{i+m<j} \frac{s_{i, j-1}}{s_{i, j}}   \right) .
 \end{align*} 
 It thus follows from Fubini's theorem  that   the left hand side of \eqref{eq62} is equal to  
 \[\int_{\mathbb{R}^n_{>0}} \mathcal{T}_{\alpha_{\circ},\bsb{\beta};r}(\bsb{x}) \Psi_{\boldsymbol{\alpha}}^{\mathfrak{gl}_n}(\boldsymbol{x}) \prod_{i=1}^n \frac{dx_i}{x_i},\]
 which, with \eqref{eq62}, concludes the proof of   the second claim of Proposition \ref{prop:whittransformT}.
\end{proof}

  \subsection{Contour integral formula}

    We deduce from Theorem \ref{LapTrfmFl} and Proposition \ref{prop:whittransformT} a contour integral formula for the Laplace transform of $Z^{\rflat}(n;m)$.
 \begin{corollary}\label{LapTransFla} 
 Assume that parameters $n,m,\alpha_{\circ },\bsb \alpha, \bsb \beta$ and the partition function $ Z^{\rflat}(n;m)$ are   as in Definition \ref{nota}. Suppose  $\mu> \max_{1\leq i\leq n}  \alpha_{i} $.
Suppose further that $m\geq n-1$.
We have
   \begin{align}\nonumber
  \mathbb{E}&\left[e^{-r Z^{\rflat}(n;m)}\right]\\
  &=\int_{(\mu+\mathrm{i} \mathbb{R})^{n}} r^{\sum_{i=1}^{n}\left(\alpha_{i}-\lambda_{i}\right)}\prod_{1 \leq i, j \leq n} \Gamma\left(\lambda_{i}-\alpha_{j} \right) 
  \prod_{i=1}^{n} \frac{\Gamma\left(\lambda_{i}+{\alpha}_{\circ} \right)}{\Gamma\left(\alpha_{i}+{\alpha}_{\circ}\right)}\prod_{ i=1}^n\prod_{j=1}^{n} \frac{\Gamma\left(\lambda_{i}+{\alpha}_{j} \right)}{\Gamma\left(\alpha_{i}+{\alpha}_{j}\right)} 
   \prod_{ i=1}^n\prod_{k=1}^m\frac{\Gamma(\lambda_i+\beta_k)}{\Gamma(\alpha_i+\beta_k)}
   s_{n}(\bsb{\lambda}) d \bsb{\lambda}.\label{eq2.5}
\end{align}
\end{corollary}
\begin{proof}

Let
\begin{equation*}
    f_1(\bsb{x}):= \Psi^{\mathfrak{so}_{2n+1}}_{\bsb{\alpha}}(\bsb{x})\left(\prod_{i=1}^nx_i\right)^{\mu} \quad \text{and} \quad f_2(\bsb{x}):=\left(\prod_{i=1}^nx_i\right)^{-\mu} \mathcal{T}_{\alpha_{\circ},\bsb{\beta};r}(\bsb{x}).
\end{equation*}
Then by Theorem \ref{LapTrfmFl},
 \begin{align}  \label{eq21}
\mathbb{E}\left[e^{-rZ^{\rflat}(n;m)}\right]=&{} \frac{r^{\sum_{k=1}^{n}\alpha_{k}} \int_{\mathbb{R}_{>0}^{n}} f_1(\boldsymbol{x}) \overline{f_2(\boldsymbol{x})} \prod_{i=1}^{n} {{d} x_{i}}/{x_{i}}}{ \prod_{ i=1}^n \Gamma\left(\alpha_{i}+\alpha_{\circ} \right) \prod_{ i=1}^n\prod_{j=1}^{n} \Gamma\left(\alpha_{i}+\alpha_{j} \right)\prod_{ i=1}^n\prod_{k=1}^m\Gamma(\alpha_i+\beta_k)}.
\end{align}
By Proposition \ref{TransSO}, $f_1\in L^{2}\left(\mathbb{R}_{>0}^{n}, \prod_{i=1}^{n}{~d} x_{i} / x_{i}\right)$  and 
\begin{equation*}
    \widehat{f}_1(\boldsymbol{\lambda}):=\frac{\prod_{1 \leq i, j \leq n} \Gamma\left(\mu-\lambda_{i}+\alpha_{j}\right) \Gamma\left(\mu-\lambda_{i}-\alpha_{j}\right)}{\prod_{1 \leq i<j \leq n} \Gamma\left(2 \mu-\lambda_{i}-\lambda_{j}\right)}.
\end{equation*}
 Now we calculate the Whittaker transform of  $f_2 $, 
\begin{align}\label{eq97}
     \widehat{f}_2(\boldsymbol{\lambda})&= \int_{\mathbb{R}_{>0}^{n}} \left(\prod_{i=1}^nx_i\right)^{-\mu}  \mathcal{T}_{\alpha_{\circ},\bsb{\beta};r}(\bsb{x}) \Psi_{\boldsymbol{\lambda}}^{\mathfrak{g l}_n}(\boldsymbol{x}) \prod_{i=1}^{n} \frac{{d} x_{i}}{x_{i}} 
     = \int_{\mathbb{R}_{>0}^{n}}  \mathcal{T}_{\alpha_{\circ},\bsb{\beta};r}(\bsb{x}) \Psi_{\boldsymbol{\lambda}+\mu}^{\mathfrak{g l}_n}(\boldsymbol{x}) \prod_{i=1}^{n} \frac{{d} x_{i}}{x_{i}},
\end{align} 
where the last equality holds due to \eqref{eq:whitransprop}.
Recall that $\bsb \lambda \in (\mathrm{i}\mathbb{R})^n$, 
the hypothesis $\mu> \max_{1\leq i\leq n}  \alpha_{i}$  thus enables us to  apply Proposition \ref{prop:whittransformT} to the RHS of \eqref{eq97} 
\begin{align*}
     \widehat{f}_2(\boldsymbol{\lambda}) &=  {}r^{-\sum_{i=1}^{n} (\lambda_{i}+\mu)} \prod_{i=1}^n \Gamma\left(\lambda_{i}+\alpha_{\circ}+\mu\right) \prod_{1\leq i<j\leq n} \Gamma\left(2\mu+\lambda_{i}+\lambda_{j}\right)\prod_{i=1}^n\prod_{k=1}^m\Gamma(\lambda_i+\mu+\beta_k).
\end{align*} 
 Now we show that $\widehat{f}_2$ is square integrable, which will imply that ${f}_2$  also is square integrable and hence enable us to apply the Plancherel formula to $ \int_{\mathbb{R}_{>0}^{n}} f_1(\boldsymbol{x}) \overline{f_2(\boldsymbol{x})} \prod_{i=1}^{n} {{d} x_{i}}/{x_{i}}$.
 Recall the following equality involving  the Gamma function
 \begin{align*}
& |\Gamma(\mathrm{i}y)|^{2}=\frac{\pi}{y \sinh \pi y}\quad \text{for}\quad y\in \mathbb{R}\setminus \{0\}, 
\end{align*}
which yields a bound for  the Sklyanin measure  (see \eqref{eq63})
\[
  | s_{n}(\bsb{\lambda})|\leq   \prod_{i<j } |\lambda_{i}-\lambda_{j}| e^{\pi|\lambda_{i}-\lambda_{j}|}.
    \] 
    In order to derive upper bounds for terms in  $\widehat{f}_2$, we use   the following properties of the Gamma function
 \begin{align*}
 &|\Gamma(x+\mathrm{i} y)| \leq |\Gamma(x)| \quad \text{for}\quad  x\in\mathbb{R}_{>0},\, y\in \mathbb{R}, \\
&  |\Gamma(x+\mathrm{i} y)| = \sqrt{2 \pi}|y|^{x-\frac{1}{2}} e^{-\frac{\pi}{2}|y|+O(1/|y|)} \quad
 \text{as} \, \, |y| \rightarrow \infty \, \,\text{with fixed}\, \,x\in\mathbb{R}.
\end{align*}
 It follows
\[\left|\prod_{1\leq i<j\leq n} \Gamma\left(2\mu+\lambda_{i}+\lambda_{j}\right)\right| \leq \Gamma\left(2\mu\right)^{n(n-1)/2},
\]
\[\left|\prod_{i=1}^n \Gamma\left(\lambda_{i}+\alpha_{\circ}+\mu\right)
\prod_{ i=1}^n \prod_{k=1}^m\Gamma(\lambda_i+\mu+\beta_k)\right|
= (2\pi)^{\frac{n(m+1)}{2}}e^{-\frac{\pi}{2}\sum_{i=1}^n (m+1)(|\lambda_i|+O(\frac{1}{|\lambda_i|}))}\prod_{i=1}^n|\lambda_i|^{\alpha_{\circ}+\mu-\frac{1}{2}+\sum_{k=1}^m(\beta_k+\mu-\frac{1}{2})} 
.\]
 Therefore,  as $|\lambda_i|\rightarrow\infty$ for all $1\leq i\leq n$,
 \begin{align}\nonumber
     |\widehat{f}_2|^2|s_n(\bsb{\lambda}) |&\leq P(\bsb{\lambda})\exp\left(-\pi\sum_{i=1}^n(m+1)|\lambda_i|
     \right)
     \exp\left(     \pi \sum_{i<j}\left|\lambda_{i}-\lambda_{j}\right|
     \right) \\ \nonumber
    & \leq P(\bsb{\lambda})\exp\left(
    -\pi\sum_{i=1}^n(m+1)|\lambda_i|  +  \pi \sum_{i<j}|\lambda_{i}|+|\lambda_{j}|
     \right) \\
      & =P(\bsb{\lambda})\exp\left(
    -\pi\sum_{i=1}^n(m-n+2)|\lambda_i|   
     \right) , \label{eq98}
 \end{align}
 with $P(\bsb{\lambda})$ taking the form 
 \[
 P(\bsb{\lambda}) =C_1 \left(\prod_{i=1}^n |\lambda_i|\right)^{C_2}\prod_{i<j} |\lambda_i-\lambda_j|
 \]
 where $C_1$ and $C_2$ are constants that do not depend on $\bsb{\lambda}$.
 Since $m\geq n-1$ by hypothesis, the   estimate \eqref{eq98} shows that $\widehat{f_2}$ is square-integrable.
 
By the Plancherel formula (\ref{plancherel}),
\begin{align*}
    &\int_{\mathbb{R}_{>0}^{n}} f_1(\boldsymbol{x}) \overline{f_2(\boldsymbol{x})} \prod_{i=1}^{n} \frac{{d} x_{i}}{x_{i}}\\
    =& \int_{(\mathrm{i}\mathbb {R})^{n}} r^{\sum_{i=1}^{n} (\lambda_{i}-\mu)}  
    \prod_{1 \leq i, j \leq n} \Gamma\left(\mu-\lambda_{i}+\alpha_{j}\right) \Gamma\left(\mu-\lambda_{i}-\alpha_{j}\right) \\
   & \qquad \qquad   \times\prod_{i=1}^n \Gamma\left(-\lambda_{i}+\alpha_{\circ}+\mu\right)
    \prod_{i=1}^n\prod_{k=1}^m\Gamma(-\lambda_i+\mu+\beta_k)
    s_{n}(\boldsymbol{\lambda}) {d} \boldsymbol{\lambda} \\
    =&\int_{(\mu+\mathrm{i} \mathbb { R })^{n}} r^{-\sum_{i=1}^{n} \lambda_{i}}  \prod_{1 \leq i, j \leq n}   \Gamma\left(\lambda_{i}-\alpha_{j}\right)  \prod_{i=1}^n 
    \Gamma\left(\lambda_{i}+\alpha_{\circ}\right)  \prod_{ i=1}^n\prod_{j=1}^n
    \Gamma\left(\lambda_{i}+\alpha_j\right) 
    \prod_{i=1}^n\prod_{k=1}^m\Gamma(\lambda_i+\beta_k)
    s_{n}(\boldsymbol{\lambda}) {d} \boldsymbol{\lambda}.
\end{align*}
where the change of   variables $\bsb{\lambda}$ to $-\bsb{\lambda}+\mu$ was used in the last equality. Now it suffices to insert this back to (\ref{eq21}) to complete the proof.
\end{proof}

   \section{Proof of  Theorem \ref{main}}\label{sec4}
   To prove Theorem \ref{main}, we compare   the Laplace transform of ${Z}^{\rflat}(n;m)$, given by Corollary \ref{LapTransFla},  to that of the point-to-point partition function $Z(n,n+m+1)$ (as in Definition \ref{nota}), which was computed  in \cite{COSZ}. We will see that the Laplace transforms of the two partition functions are the same, which yields the identity in law when $m\geq n-1$. We then  prove the general case by a probabilistic argument.

   Consider the    weight array  
   \begin{equation}\label{param_nv}
   W_{i,j}^{-1}\sim \operatorname{Gamma}({\alpha}_i+\widehat{\alpha}_j)\quad
   \text{on}\quad \{(i,j)\mid 1\leq i\leq n, 1\leq j\leq N\}
   \end{equation}
   and the point-to-point partition function $Z(n,N)$   defined by this weight array via \eqref{defZ}. The Laplace transform of $Z(n,N)$ can be expressed as a contour integral.
   
\begin{proposition}[{\cite[Theorem 3.9]{COSZ}}] \label{lapTrans} 
Suppose $n\leq  N$. Suppose    ${\alpha}_1,\dots,{\alpha}_n$, $\widehat{\alpha}_1,\dots,\widehat{\alpha}_N$ are such that ${\alpha}_i+\widehat{\alpha}_j>0$ for all $i,j$. Let $Z(n,N)$ be  the point-to-point partition function for the log-gamma polymer with weight array   (\ref{param_nv}). Let $\mu>\max_{i,j}\{-\widehat{\alpha}_j,{\alpha}_i\}$. Then for all $r>0$, the following formula holds
   \begin{equation*}
   \mathbb{E}\left[e^{-r Z(n,N)}\right]=\int_{(\mu+\mathrm{i} \mathbb{R})^{n}}  r^{\sum_{i=1}^{n}\left({\alpha}_{i}-\lambda_{i}\right)} \prod_{1 \leq i, j \leq n} \Gamma\left(\lambda_{i}-{\alpha}_{j}\right)  \prod_{i=1}^{n} \prod_{j=1}^{N} \frac{\Gamma\left(\lambda_{i}+\widehat{\alpha}_{j}\right)}{\Gamma\left({\alpha}_{i}+\widehat{\alpha}_{j}\right)}s_{n}(\bsb{\lambda}) d \bsb{\lambda}.
\end{equation*} 
\end{proposition}
\begin{proof} 
In \cite{COSZ}, it is proved that if
${\alpha}_1,\dots,{\alpha}_n<0$ and 
$\widehat{\alpha}_1,\dots,\widehat{\alpha}_N>0$,  then
   \begin{equation*}
   \mathbb{E}\left[e^{-r Z(n,N)}\right]=\int_{(\mathrm{i} \mathbb{R})^{n}}  r^{\sum_{i=1}^{n}\left({\alpha}_{i}-\lambda_{i}\right)} \prod_{1 \leq i, j \leq n} \Gamma\left(\lambda_{i}-{\alpha}_{j}\right)  \prod_{i=1}^{n}\prod_{j=1}^{N} \frac{\Gamma\left(\lambda_{i}+\widehat{\alpha}_{j}\right)}{\Gamma\left({\alpha}_{i}+\widehat{\alpha}_{j}\right)}s_{n}(\bsb{\lambda}) d \bsb{\lambda}.
\end{equation*} 
In the general case, note that ${\alpha}_i+\widehat{\alpha}_j=({\alpha}_i-\mu)+(\widehat{\alpha}_j+\mu)$. We may therefore apply the preceding formula  to $({\alpha}_i-\mu)_i$ and $(\widehat{\alpha}_j+\mu)_j$
\begin{equation*}
   \mathbb{E}\left[e^{-r Z(n,N)}\right]=\int_{(\mathrm{i} \mathbb{R})^{n}}  r^{\sum_{i=1}^{n}\left({\alpha}_{i}-\mu-\lambda_{i}\right)} \prod_{1 \leq i, j \leq n} \Gamma\left(\lambda_{i}+\mu-{\alpha}_{j}\right)   \prod_{i=1}^{n}\prod_{j=1}^{N} \frac{\Gamma\left(\lambda_{i}+\widehat{\alpha}_{j}+\mu\right)}{\Gamma\left({\alpha}_{i}+\widehat{\alpha}_{j}\right)}s_{n}(\bsb{\lambda}) d \bsb{\lambda}.
\end{equation*} 
The change of variable $\bsb{\lambda}\mapsto\bsb{\lambda}-\mu$ then completes the proof. 
\end{proof} 

\begin{proposition}\label{prop_4}
Theorem \ref{main} holds if $m\geq n-1$.
\end{proposition}
\begin{proof}
  We apply  Proposition \ref{lapTrans} to $Z(n,n+m+1)$, the point-to-point partition function of the log-gamma polymer with weight array (\ref{paramFull}).
  To this end, it suffices to take $N=1+n+m$ and let $\widehat{\bsb{\alpha}}$ be the concatenation of $(\alpha_0), \bsb{\alpha}, \bsb{\beta}$, i.e., $\widehat{\bsb{\alpha}}=(\alpha_0)\sqcup \bsb{\alpha}\sqcup \bsb{\beta}$. We thus obtain the Laplace transform of $Z(n,n+m+1)$, which is the same as the   Laplace transform of ${Z}^{\rflat}(n;m)$  given by Corollary \ref{LapTransFla} under the hypothesis $m\geq n-1$. This yields that the two partition functions  are identical in law.  
\end{proof}

To complete the proof for  Theorem \ref{main}, we only need to treat the $m<n-1$ case. 
  Let $\ell$   be any fixed integer larger or equal to   $n-m-1$. The strategy of proof is to reduce the problem to the previous $m\geq n-1$ case. This will be done in two steps.

In Step 1, 
we will show that $Z(n,n+m+1)=\sum_{ \pi:(1,1)\rightarrow(n,n+m+1)}\prod_{(i,j)\in\pi}W_{i,j}$ can be expressed as the scaled limit of 
a sequence of partition functions (denoted by $(Z_{k}(n,n+m+\ell+1))_{k\in\mathbb N}$) for  full-space log-gamma polymers whose weight array is constructed    by adding $\ell$ auxiliary  columns next to  $(W_{i,j})$.

In Step 2, we will show that $Z^{\rflat}(n;m)=\sum_{\substack{p+q=2n+m+1 \\ \pi:(1,1)\rightarrow(p,q)}}\prod_{(i,j)\in\pi}W_{i,j}$ can be expressed as the scaled limit of a sequence of   partition functions (denoted by $(Z^{\rflat}_{k}(n;m+\ell))_{k\in\mathbb N}$) of trapezoidal log-gamma polymers whose weight array is constructed  by inserting $\ell$ auxiliary columns into  $(W_{i,j})$.

These two steps will enable us to conclude. Indeed,    $Z^{\rflat}_{k}(n;m+\ell)\overset{(d)}{=}Z_{k}(n,n+m+\ell+1)$ holds   by  Proposition \ref{prop_4} as we have chosen $\ell\geq n-m-1$, and the desired identity in distribution  $Z^{\rflat}(n;m)\overset{(d)}{=} Z(n,n+m+1)$   follows by letting $k$ go to infinity.

 \begin{proof}[Proof of Theorem \ref{main}] 
 Let $\ell\geq n-m-1$ be a fixed integer.
 
\emph{Step 1.}
Let  $(w_{i,j}^{k}:1\leq i\leq n, n+m+1< j\leq n+m+\ell +1,k>0)$ be a family of  random variables whose distribution is given by $w_{i,j}^k\sim \operatorname{Gamma}^{-1}(\alpha_i+k)$ for each triple $k,i,j$. 
We require that $(w^k_{i,j})$ be such that  $(W_{i,j})\cup(w^k_{i,j})$ is a  family of  independent random variables.
Define $(W_{i,j}^{k})$   by (see Figure \ref{img_aug1_1})
\begin{equation}\label{eq79}
\text{for $1\leq i\leq n$ and $k\geq 1$,}\quad 
    W_{i,j}^{k}:=\left\{
    \begin{array}{ll}
        W_{i,j} &\text{if} \quad 1\leq j\leq n+m+1,  \\
        w_{i,j}^{k} &\text{if}\quad  n+m+1<j\leq n+m+\ell+1.
    \end{array}
    \right.
\end{equation}
Let 
$$Z_k(n,n+m+\ell+1):=\sum_{ \pi:(1,1)\rightarrow(n,n+m+\ell+1)}\prod_{(i,j)\in\pi}W_{i,j}^k$$
be the corresponding point-to-point partition function on the augmented  log-gamma polymer weight array \eqref{eq79}.
 Observe that for any up-right  path $\pi$ from $(1,1)$ to $(n,n+m+\ell+1)$, the associated product $\prod_{(i,j)\in\pi}W_{i,j}^k $ can be decomposed into two parts
\begin{equation}
    \label{eq91}
\prod_{(i,j)\in\pi}W_{i,j}^k = \prod_{\substack{(i,j)\in\pi \\ j\leq n+m+1}} W_{i,j}
 \prod_{\substack{(i,j)\in\pi \\  j>n+m+1}} w^k_{i,j}  .
\end{equation} 
The number of factors in the product  $ \prod_{\substack{(i,j)\in\pi \\  j>n+m+1}} w^k_{i,j} $ is at least $\ell$, i.e., $$\left|\{(i,j)\in\pi\mid j>n+m+1\}\right|\geq \ell.$$
Observe that $\left|\{(i,j)\in\pi\mid j>n+m+1\}\right|= \ell$  if and only if $(n,n+m+1)\in\pi$.  Let $\Pi_1$ denote the collection of all  up-right  paths $\pi$ that pass by $(n,n+m+1)$ (see Figure \ref{img_aug1_2})
\[\Pi_1:=\{\pi\mid (n,n+m+1)\in\pi\},\]
that is, $\pi\in\Pi_1$ if and only if $\left|\{(i,j)\in\pi\mid j>n+m+1\}\right|= \ell$.

Now we show that as $k$ tends to infinity, $\sum_{\pi\in\Pi_1}\prod_{(i,j)\in\pi}W^k_{i,j}$ dominates $\sum_{\pi\notin\Pi_1}\prod_{(i,j)\in\pi}W^k_{i,j}$. 
For an inverse-gamma random variable  $X\sim \operatorname{Gamma}^{-1}( k+a)$, we have 
\[\mathbb{E}[k X]=\frac{k}{k+a-1} =1+O\left({k}^{-1}\right)
\,
\text{ and } \,
\operatorname{Var}[k X]=\left(\frac{k}{k+a-1}\right)^2\frac{1}{k+a-2}=O\left({k}^{-1}\right) \text{ as } k\rightarrow\infty.
\]
Hence for all $i,j$, $\lim_{k\rightarrow\infty} kw_{i,j}^k=1$   holds  in $L^2$.  Using \eqref{eq91},  we have    the following convergence in $L^2$
\begin{align*}
\lim_{k\rightarrow\infty}k^{\ell}
\prod_{(i,j)\in\pi}W_{i,j}^k
&= \prod_{\substack{(i,j)\in\pi \\ j\leq n+m+1}} W_{i,j}  
\lim_{k\rightarrow\infty} k^{\ell-\left|\{(i,j)\in\pi\mid j>n+m+1\}\right|} 
 \prod_{\substack{(i,j)\in\pi \\  j>n+m+1}} kw^k_{i,j}  \\
&=
\left\{
\begin{aligned}
    \prod_{\substack{(i,j)\in\pi \\ j\leq n+m+1}} &W_{i,j} & \text{ if } \pi\in\Pi_1,   \\
   & 0     & \text{ if } \pi\notin\Pi_1.
\end{aligned}
\right. 
\end{align*} 
It follows that 
\begin{equation}\label{eq92}
\lim_{k\rightarrow\infty} 
k^\ell Z_k(n,n+m+\ell+1)
=\sum_{  \pi\in\Pi_1}\prod_{\substack{(i,j)\in\pi \\ j\leq n+m+1}}W_{i,j} \quad \text{in $L^2$.}
\end{equation} 
 Since $\pi\in\Pi_1$ if and only if $(n,n+m+1)\in\pi$, the RHS of \eqref{eq92} is in fact equal to $Z(n,n+m+1)$
\begin{equation}\label{eq86}
\lim_{k\rightarrow\infty} 
k^\ell Z_k(n,n+m+\ell+1)
=Z(n,n+m+1)\quad \text{in $L^2$.}
\end{equation}
 
 \begin{figure}

     \centering

\begin{subfigure}{.48\textwidth}
\begin{center} 
 
     \begin{tikzpicture}[scale=0.75]

	 \draw (-0.5, -0.4) node{{\footnotesize $(1,1)$}};
	 \draw (8, 3.5) node{{\footnotesize $(n,n+m+\ell+1)$}};
	 
		\draw[->] (-0.3,1.5) node[anchor=east]{$W_{i,j}$} to[bend left] (1.2,1.5);  
		  
		\draw[->] (7.5,1.5) node[anchor=west]{$w^k_{i,j}$} to[bend left] (6.5,1.5);  
		  
 \coordinate (1) at (0,0);
 \coordinate (2) at (1,0);
 \coordinate (3) at (2,0);
 \coordinate (4) at (3,0);
 \coordinate (5) at (0,1);
 \coordinate (6) at (1,1);
 \coordinate (7) at (2,1);
 \coordinate (8) at (3,1);
 \coordinate (9) at (0,2);
 \coordinate (10) at (1,2);
 \coordinate (11) at (2,2);
 \coordinate (12) at (3,2);
 \coordinate (13) at (0,3);
 \coordinate (14) at (1,3);
 \coordinate (15) at (2,3);
 \coordinate (16) at (3,3);
 
 \coordinate (17) at (4,0);
 \coordinate (18) at (4,1);
 \coordinate (19) at (4,2);
 \coordinate (20) at (4,3);
 \foreach \n in {1,5,9,13} \fill [green] (\n)
   circle (2pt) node [below] {};  
  \foreach \n in {2,3,4,6,7,8,10,11,12,14,15,16,17,18,19,20} \fill [red] (\n) circle (2pt) node [below] {};

 \coordinate (1) at (5,0);
 \coordinate (2) at (5,1);
 \coordinate (3) at (5,2);
 \coordinate (4) at (5,3); 
 \foreach \n in {1,2,3,4}  \fill [blue] (\n) circle (2pt) node [below] {};

 \coordinate (1) at (7,0);
 \coordinate (2) at (7,1);
 \coordinate (3) at (7,2);
 \coordinate (4) at (7,3);
 \coordinate (5) at (6,0);
 \coordinate (6) at (6,1);
 \coordinate (7) at (6,2);
 \coordinate (8) at (6,3); 
 \foreach \n in {1,2,3,4,5,6,7,8} \draw (\n) circle(2pt);
    
		\end{tikzpicture}
		
	\caption{}
	\label{img_aug1_1}
	 \end{center}
 \end{subfigure} 
 \hfill
\begin{subfigure}{.45\textwidth}
\begin{center} 
 
     \begin{tikzpicture}[scale=0.75] 
      
	 \fill [white] (-0.5, -0.4) circle (2pt) node [below] {};  
	 \fill [white] (8,3.8) circle (2pt) node [below] {};  
\draw[thick] (0,0)--(0,1)--(1,1)--(1,2)--(2,2)--(5,2)--(5,3)--(7,3);

\draw[thick][dashed] (0,0)--(2,0)--(2,1)--(6,1)--(7,1)--(7,2)--(7,3);		 

 \coordinate (1) at (0,0);
 \coordinate (2) at (1,0);
 \coordinate (3) at (2,0);
 \coordinate (4) at (3,0);
 \coordinate (5) at (0,1);
 \coordinate (6) at (1,1);
 \coordinate (7) at (2,1);
 \coordinate (8) at (3,1);
 \coordinate (9) at (0,2);
 \coordinate (10) at (1,2);
 \coordinate (11) at (2,2);
 \coordinate (12) at (3,2);
 \coordinate (13) at (0,3);
 \coordinate (14) at (1,3);
 \coordinate (15) at (2,3);
 \coordinate (16) at (3,3);
 
 \coordinate (17) at (4,0);
 \coordinate (18) at (4,1);
 \coordinate (19) at (4,2);
 \coordinate (20) at (4,3);
 \foreach \n in {1,5,9,13} \fill [green] (\n)
   circle (2pt) node [below] {};  
  \foreach \n in {2,3,4,6,7,8,10,11,12,14,15,16,17,18,19,20} \fill [red] (\n) circle (2pt) node [below] {};

 \coordinate (1) at (5,0);
 \coordinate (2) at (5,1);
 \coordinate (3) at (5,2);
 \coordinate (4) at (5,3); 
 \foreach \n in {1,2,3,4}  \fill [blue] (\n) circle (2pt) node [below] {};

 \coordinate (1) at (7,0);
 \coordinate (2) at (7,1);
 \coordinate (3) at (7,2);
 \coordinate (4) at (7,3);
 \coordinate (5) at (6,0);
 \coordinate (6) at (6,1);
 \coordinate (7) at (6,2);
 \coordinate (8) at (6,3); 
 \foreach \n in {1,2,3,4,5,6,7,8} \draw (\n) circle(2pt);
    
		\end{tikzpicture}		
	\caption{}
	\label{img_aug1_2}
	\end{center}
 \end{subfigure} 
 \label{img_aug1}
	\caption[]{Example of weights $(W_{i,j}^k)$ and  paths considered in Step 1, in the case  $n=4,m=1,\ell=2$. (A) The weight array $(W_{i,j}^k)$ is obtained by adding $\ell$ columns of weight $(w^k_{i,j})$ (represented by $\personalcirc$) to the right of $(W_{i,j})$. (B) A path in $\Pi_1$ (the solid line) and a path not in $\Pi_1$ (the dashed line).}
 \end{figure}

 \emph{Step 2.}
Similarly to Step 1, let $(w_{i,j}^{k})_{
\substack{1\leq i\leq n \\ n+m< j\leq n+m+\ell \\ k>0}}$ be a family of independent random variables that is independent of $(W_{i,j})$, whose distribution is given by $w_{i,j}^k\sim \operatorname{Gamma}^{-1}(\alpha_i+k)$ for each triple of $k,i,j$.
Define $(W_{i,j}^{k})$  by (see Figure \ref{img_aug2_1})
\begin{equation}\label{eq80}
\text{for $1\leq i\leq n$ and $k\geq 1$,}\quad 
    W_{i,j}^{k}:=\left\{
    \begin{array}{ll}
        W_{i,j} &\text{if} \quad i\leq j\leq n+m,  \\
        w_{i,j}^{k} &\text{if}\quad  n+m<j\leq n+m+\ell, \\
        W_{i,j-\ell} & \text{if}\quad  n+m +\ell<j\leq 2n+m+\ell-i+1.
    \end{array}
    \right.
\end{equation}
Let 
$$Z^{\rflat}_k(n;m+\ell):=\sum_{\substack{ \pi:(1,1)\rightarrow(p,2n+m+\ell-p+1) \\ 1\leq p\leq n }}\prod_{(i,j)\in\pi}W_{i,j}^k$$
be the corresponding point-to-line partition function on the augmented trapezoidal log-gamma polymer with weight array \eqref{eq80}.

For any up-right path $\pi:(1,1)\rightarrow (p,2n+m+\ell-p+1)$, the associated product $\prod_{(i,j)\in\pi}W_{i,j}^k $ can be decomposed into three parts
\begin{equation}\label{eq93}
\prod_{\substack{(i,j)\in\pi}}W_{i,j}^k = \prod_{\substack{(i,j)\in\pi \\j\leq n+m}} W_{i,j} \prod_{\substack{(i,j)\in\pi \\n+m<j\leq n+m+\ell}} w^k_{i,j}   \prod_{\substack{(i,j)\in\pi \\ j>n+m+\ell}} W_{i,j-\ell}.
\end{equation}
We see that   the number of factors in the product  $\prod_{\substack{(i,j)\in\pi \\n+m<j\leq n+m+\ell}} w^k_{i,j}$ is at least $\ell$, i.e.,
$$
|\{(i,j)\in\pi\mid n+m<j\leq n+m+\ell\}|\geq \ell .
$$
The preceding inequality becomes an equality  if and only if $\pi$ stays horizontal on the columns $j=n+m+1,\dots,j=n+m+\ell$. To be precise, this means   there exists an integer  $i$ such that $(i,n+m)\in\pi$ and $(i,n+m+\ell+1)\in\pi$.  Let $\Pi_2$ denote the collection of all such paths $\pi$  (see Figure \ref{img_aug2_2}).

As explained in Step 1, for all $i,j$, $\lim_{k\rightarrow\infty} kw_{i,j}^k=1$   in $L^2$. 
Thus, for   $\pi\notin\Pi_2$ we have 
$$
\lim_{k\rightarrow\infty}k^{\ell}
\prod_{(i,j)\in\pi}W_{i,j}^k
= \prod_{\substack{(i,j)\in\pi \\j\leq n+m}} W_{i,j} \left(\lim_{k\rightarrow\infty}k^{\ell}\prod_{\substack{(i,j)\in\pi \\n+m<j\leq n+m+\ell}} w^k_{i,j}   
\right)
\prod_{\substack{(i,j)\in\pi \\ j>n+m+\ell}} W_{i,j-\ell}=0
\text{ in $L^2$,} $$
where we used \eqref{eq93} and $|\{(i,j)\in\pi\mid n+m<j\leq n+m+\ell\}|> \ell $.
And  for  $\pi\in\Pi_2$ we have $$\lim_{k\rightarrow\infty}k^{\ell}
\prod_{(i,j)\in\pi}W_{i,j}^k=\prod_{\substack{(i,j)\in\pi \\ j\leq n+m}}W_{i,j}\prod_{\substack{(i,j)\in\pi \\ j> n+m+\ell}}W_{i,j-\ell}\quad \text{in $L^2$} $$
since $|\{(i,j)\in\pi\mid n+m<j\leq n+m+\ell\}|= \ell $ holds for $\pi\in\Pi_2$.
It follows that 
\begin{equation}\label{eq81}
\lim_{k\rightarrow\infty} 
k^\ell Z^{\rflat}_k(n;m+\ell)
=\sum_{  \pi\in\Pi}\prod_{\substack{(i,j)\in\pi \\ j\leq n+m}}W_{i,j}\prod_{\substack{(i,j)\in\pi \\ j> n+m+\ell}}W_{i,j-\ell}\quad \text{in $L^2$.}
\end{equation}
We claim that the RHS of \eqref{eq81} is   $Z^{\rflat}(n;m)$. To see this, let  $\Pi^{\prime}_2$ be the set of all up-right paths from $(1,1)$ to $(p,2n+m-p+1)$ for some $1\leq p\leq n$. Then there is a bijetion between $\Pi_2$ and $\Pi^{\prime}_2$. Indeed, since   every  path   $\pi\in\Pi_2$ is   horizontal on the columns carrying weights $w_{i,j}$,  we can remove  the horizontal part from $\pi$ and obtain a path in $\Pi^{\prime}_2$. Inversely, for $\pi^{\prime}\in\Pi^{\prime}_2$ we add a horizontal part to it to obtain a path in $\Pi_2$.
To be precise, we construct this bijection as follows. For $\pi\in\Pi_2$, let 
\begin{equation}\label{eq77}
    \pi^{\prime}:=\{(i,j)\mid j\leq n+m, (i,j)\in \pi\}\cup\{(i,j-\ell)\mid j> n+m+\ell, (i,j)\in\pi\}.
\end{equation}
Then $\pi^{\prime}\in\Pi^{\prime}_2$.
Inversely, for $\pi^{\prime}\in\Pi^{\prime}_2$, let 
$i_0$ be such that $(i_0,n+m)\in\pi^{\prime}$ and $(i_0,n+m+1)\in \pi^{\prime}$. Let 
$$\pi:=\{(i,j)\mid (i,j)\in\pi^{\prime}, j\leq n+m\}\cup \{(i_0,j)\mid n+m<j\leq n+m+\ell\}\cup \{(i,j+\ell)\mid (i,j)\in\pi^{\prime},j>n+m\}.$$ 
 Then $\pi \in\Pi_2 $. Using this bijection, for $\pi \in\Pi_2 $, let $\pi^{\prime}\in\Pi^{\prime}_2$ be  constructed as in (\ref{eq77}), we have 
 $$ \prod_{\substack{(i,j)\in\pi \\ j\leq n+m}}W_{i,j}\prod_{\substack{(i,j)\in\pi \\ j> n+m+\ell}}W_{i,j-\ell} = \prod_{(i,j)\in\pi^{\prime}}W_{i,j}.
 $$
 Inserting this into (\ref{eq81}), 
 \begin{equation}\label{eq85}
\lim_{k\rightarrow\infty} 
k^\ell Z^{\rflat}_k(n;m+\ell) 
 =\sum_{\pi^{\prime}\in\Pi^{\prime}_2} \prod_{(i,j)\in\pi^{\prime}}W_{i,j}
 =Z^{\rflat}(n;m) \quad \text{in $L^2$,} 
 \end{equation}
 as claimed.
 
 \begin{figure}

\begin{subfigure}{.46\textwidth}
\begin{center} 
 
     \begin{tikzpicture}[scale=0.65] 
	 
		\draw[->] (1,1.5) node[anchor=east]{$W_{i,j}$} to[bend left] (2.2,1.5);  
		  
		\draw[->] (7.3,2.6) node[anchor=west]{$w^k_{i,j}$} to[bend right] (5.5,2.3);  
		
		\draw[->] (8.8,1.2) node[anchor=south]{$W_{i,j}$} to[bend left] (7.5,0.5);  
		  
 \coordinate (1) at (0,0);
 \coordinate (2) at (1,0);
 \coordinate (3) at (2,0);
 \coordinate (4) at (3,0);
 \coordinate (5) at (1,1); 
 \coordinate (7) at (2,1);
 \coordinate (8) at (3,1);
 \coordinate (9) at (2,2); 
 \coordinate (12) at (3,2);
 \coordinate (13) at (3,3);  
 
 \coordinate (17) at (4,0);
 \coordinate (18) at (4,1);
 \coordinate (19) at (4,2);
 \coordinate (20) at (4,3);
 \foreach \n in {1,5,9,13} \fill [green] (\n)
   circle (2pt) node [below] {};  
  \foreach \n in {2,3,4,7,8,12} \fill [red] (\n) circle (2pt) node [below] {};   
 \foreach \n in {17,18,19,20}  \fill [blue] (\n) circle (2pt) node [below] {};

 \coordinate (1) at (5,0);
 \coordinate (2) at (5,1);
 \coordinate (3) at (5,2);
 \coordinate (4) at (5,3);
 \coordinate (5) at (6,0);
 \coordinate (6) at (6,1);
 \coordinate (7) at (6,2);
 \coordinate (8) at (6,3); 
 \foreach \n in {1,2,3,4,5,6,7,8} \draw (\n) circle(2pt);
   
 \coordinate (1) at (7,0);
 \coordinate (2) at (7,1);
 \coordinate (3) at (7,2);
 \coordinate (4) at (7,3);
 \coordinate (5) at (8,0);
 \coordinate (6) at (8,1); 
 \coordinate (7) at (8,2); 
 \coordinate (8) at (9,0); 
 \coordinate (9) at (9,1); 
 \coordinate (10) at (10,0); 
 \foreach \n in {1,2,3,4,5,6,7,8,9,10}  \fill [red] (\n)  circle (2pt) node [below] {};  
    
		\end{tikzpicture}
		
	\caption{}
	
 \label{img_aug2_1}
	 \end{center}
	 
 \end{subfigure} 
 \hfill
\begin{subfigure}{.46\textwidth}
\begin{center} 
 
     \begin{tikzpicture}[scale=0.65] 
\draw[thick] (0,0)--(4,0)--(4,1)--(8,1)--(8,2);

\draw[thick][dashed] (0,0)--(1,0)--(1,1)--(3,1)--(3,2)--(5,2)--(5,3)--(7,3);
		  
 \coordinate (1) at (0,0);
 \coordinate (2) at (1,0);
 \coordinate (3) at (2,0);
 \coordinate (4) at (3,0);
 \coordinate (5) at (1,1); 
 \coordinate (7) at (2,1);
 \coordinate (8) at (3,1);
 \coordinate (9) at (2,2); 
 \coordinate (12) at (3,2);
 \coordinate (13) at (3,3);  
 
 \coordinate (17) at (4,0);
 \coordinate (18) at (4,1);
 \coordinate (19) at (4,2);
 \coordinate (20) at (4,3);
 \foreach \n in {1,5,9,13} \fill [green] (\n)
   circle (2pt) node [below] {};  
  \foreach \n in {2,3,4,7,8,12} \fill [red] (\n) circle (2pt) node [below] {};   
 \foreach \n in {17,18,19,20}  \fill [blue] (\n) circle (2pt) node [below] {};

 \coordinate (1) at (5,0);
 \coordinate (2) at (5,1);
 \coordinate (3) at (5,2);
 \coordinate (4) at (5,3);
 \coordinate (5) at (6,0);
 \coordinate (6) at (6,1);
 \coordinate (7) at (6,2);
 \coordinate (8) at (6,3); 
 \foreach \n in {1,2,3,4,5,6,7,8} \draw (\n) circle(2pt);
   
 \coordinate (1) at (7,0);
 \coordinate (2) at (7,1);
 \coordinate (3) at (7,2);
 \coordinate (4) at (7,3);
 \coordinate (5) at (8,0);
 \coordinate (6) at (8,1); 
 \coordinate (7) at (8,2); 
 \coordinate (8) at (9,0); 
 \coordinate (9) at (9,1); 
 \coordinate (10) at (10,0); 
 \foreach \n in {1,2,3,4,5,6,7,8,9,10}  \fill [red] (\n)  circle (2pt) node [below] {};  
    
		\end{tikzpicture}		
	\caption{}
 \label{img_aug2_2}
	\end{center}
	
 \end{subfigure} 
 \label{img_aug2}
	\caption[]{Example for weights $(W^k_{i,j})$ and paths considered in Step 2, in the case $n=4,m=1,\ell=2$. (A) The weight array $(W^k_{i,j})$ is obtained by inserting $\ell$ columns (represented by 
	\begin{tikzpicture}[scale=0.65]
	 \draw   circle(2pt);
	\end{tikzpicture}) into $(W_{i,j})$. (B) A path in $\Pi_2$ (the solid line) and a     path not in $\Pi_2$ (the dashed line). }
 \end{figure}

 To conclude, observe that by setting $ \widetilde{\bsb{\beta}}=\bsb{\beta}\sqcup(k,\dots,k )$, where $k$ repeats $\ell$ times, the weight array (\ref{paramFull}) with parameters $\alpha_{\circ},\bsb{\alpha},\widetilde{\bsb{\beta}}$ is identical to (\ref{eq79}).  And the weight array (\ref{paramHal}) with parameters $\alpha_{\circ},\bsb{\alpha},\widetilde{\bsb{\beta}}$ is identical to (\ref{eq80}). Since the length of $\widetilde{\bsb{\beta}}$ is $m+\ell$, which is at least $ n-1$,  by Proposition \ref{prop_4} we have $Z_{k}(n,n+m+\ell+1)\overset{(d)}{=}Z^{\rflat}_k(n;m+\ell)$   for all $k$.  
 Using (\ref{eq86})  and (\ref{eq85}), this implies 
 $Z(n,n+m+1)\overset{(d)}{=}Z^{\rflat}(n;m)$, which completes the proof.
 
 \end{proof}

   \section{Asymptotics  results for trapezoidal log-gamma polymers}\label{sec:asymptotics}
   
  In this section, we consider  homogeneous trapezoidal log-gamma polymers with boundary perturbations and prove a phase transition as the boundary parameter varies.

\begin{definition}
The GUE Tracy-Widom distribution function \cite{TW94} is given by 
\begin{equation}\label{eq:defGUE}
    F_{\mathrm{GUE}}(t):=\operatorname{det}\left(\mathbbm{1}+{K}\right)_{L^{2}\left(\mathcal{C}\right)},
\end{equation}
where the RHS is  a  Fredholm determinant and the operator $K$ is given by its integral kernel, also denoted by $K$,
\[{K} \left(v, v^{\prime}\right)=\frac{1}{2 \pi \mathrm{i}} \int_{\mathcal{D}} \frac{e^{w^{3} / 3-v^{3} / 3} e^{t v-t w}}{(v-w)\left(w-v^{\prime}\right)}  dw .\]
The contour $\mathcal{C}$ goes from $e^{
5\pi \mathrm{i} /4}\infty $ to $ e^{3\pi \mathrm{i} /4}\infty $.  The contour $\mathcal{D}$ goes from $e^{-\pi \mathrm{i} / 4} \infty$ to $e^{\pi \mathrm{i} / 4} \infty$ and lies to the right of the contour $\mathcal{C}$, i.e., the two contours do not intersect.

For $\bsb{b}\in\mathbb{R}^m$, we define the Baik-Ben Arous-P\'ech\'e distribution with parameter $\bsb{b}$ \cite{BBP05} via   
\begin{equation}\label{eq:defBBP}
F_{\mathrm{BBP},\bsb{b}}(t):=\operatorname{det}\left(\mathbbm{1}+{K}_{\bsb{b}}\right)_{L^{2}\left(\mathcal{C}\right)},
\end{equation}
where 
\[{K}_{ \bsb{b}}\left(v, v^{\prime}\right)=\frac{1}{2 \pi \mathrm{i}} \int_{\mathcal{D}} \frac{e^{w^{3} / 3-v^{3} / 3} e^{t v-t w}}{(v-w)\left(w-v^{\prime}\right)} \prod_{k=1}^{m} \frac{w-b_{k}}{v-b_{k}}dw .\]
The contours $\mathcal{C}$ and $\mathcal{D}$ are as above, but we require furthermore that $\mathcal{C}$ lie to the right of all the $b_i$'s. 

\end{definition}

Recall from Definition \ref{notation2} that $Z(n,m)$ is the point-to-point partition function for the log-gamma polymer with weight array  \eqref{paramGUE} and  $Z^{\rflat}(n;m-n-1)$ is the point-to-line partition function for the log-gamma polymer with weight array \eqref{paramTrapGUE}. The two partition functions have the same distribution due to Theorem \ref{main}.
Therefore, to prove  Theorem \ref{main2}, which states the convergence in law of the  free energy $\log Z^{\rflat}(n;m-n-1)$ appropriately scaled, we  may prove the corresponding convergence in law of   $\log Z(n,m)$ under the same scaling.
 
 In \cite{BCD}, the authors performed an asymptotic analysis of a Fredholm determinant formula from \cite{borodin2015height} for the Laplace transform of $Z(n,m)$, and obtained limiting distribution for the free energy $\log Z(n,m)$. We record the result of \cite{BCD} as Proposition \ref{asymptoticsGUE} below. It immediately yields part (2) of Theorem \ref{main2}. Using a coupling argument, we will deduce part (1) of Theorem \ref{main2} as well. 
 We then  proceed to prove  the asymptotics in the Gaussian phase of Theorem \ref{main2}   with a probabilistic argument, relying on Sepp{\"a}l{\"a}inen's fluctuation estimates on partition functions for stationary log-gamma polymers   \cite{Seppa}.

\subsection{Proof of asymptotics in the GUE Tracy-Widom and BBP phases} 

\begin{proposition}[{\cite[Theorem 1.2,\,Theorem 1.7]{BCD}}]\label{asymptoticsGUE}

Assume parameters $\theta,\theta_0,n,m,p$ and the partition function $Z(n,m)$ are as in Definition \ref{notation2}.   
\begin{itemize}
    \item  If $\theta_0=\theta$,  then   for all $t$,
    \begin{equation}\label{eq:gue}
       \lim_{n\rightarrow\infty} \mathbb{P}\left( \frac{\log Z(n,m)-nf_{\theta,p}}{n^{1 / 3} \sigma_{\theta,p}} \leq t\right)= F_{\mathrm{GUE}}(t).
    \end{equation} 
    \item  If for some $y\in\mathbb{R}$, $\lim_{n\rightarrow\infty} n^{1/3}(\theta_0-\theta_c)\sigma_{\theta,p}=y  $,  then for all $t$,
    \begin{equation} \label{eq:bbp}
       \lim_{n\rightarrow\infty} \mathbb{P}\left( \frac{\log Z(n,m)-nf_{\theta,p}}{n^{1 / 3} \sigma_{\theta,p}}  \leq t\right)= F_{\mathrm{BBP} ; -y}(t).
    \end{equation} 
     
\end{itemize}
\end{proposition}

  First observe that in light of Theorem \ref{main},   the  part (2) of Theorem \ref{main2} is implied by \eqref{eq:bbp} of Proposition \ref{asymptoticsGUE}.
  
  Now we   turn to the proof of  the asymptotics in the  GUE Tracy-Widom phase, i.e. the part (1) of Theorem \ref{main2}. Recall that the assumption here is  $\liminf_{n\rightarrow\infty} n^{1/3}(\theta_0-\theta_c)=+\infty $.  We will exhibit   an upper bound and a lower bound for $Z(n,m)$ and then show that in the $n\rightarrow \infty $ limit,  these bounds  match and  yield the  limiting behavior for $Z(n,m)$.  Thus the same limiting behavior follows for $Z^{\rflat}(n;n-m-1)$ by Theorem \ref{main}.

\subsubsection*{Upper bound}
Let  $U>0$ be a positive real number.
By the hypothesis $\limsup_{n\rightarrow\infty}n^{1/3}(\theta_0-\theta_c)=+\infty$  , for some sufficiently large $n_0$, it holds that $\theta_0-\theta_c-Un^{-1/3}>0$   for all $n>n_0$.
Now, for $n>n_0$, let  $\widetilde{Z}(n,m)$ denote  the   point-to-point partition function with   weights $(\widetilde{W}_{i,j})$ defined on the same probability space as the weights  $(W_{i,j})$, and satisfying 
\begin{equation*}
W_{i,1}\leq \widetilde{W}_{i,1}\sim \operatorname{Gamma}^{-1}(\theta_c+U n^{-1/3})\quad\text{and}\quad \widetilde{W}_{i,j}=W_{i,j} \text{ if } j\geq 2.
\end{equation*}
We may indeed construct such weights $(\widetilde{W}_{i,j})$. We first take    random variables   
$$X_i\sim\operatorname{Gamma}^{-1}(\theta_0-\theta_c-Un^{-1/3}), \quad 1\leq i\leq n,$$
in such a way that all the $X_i$'s and the $W_{i,j}$'s  are independent. Then define
$${\widetilde{W}_{i,1} }:=\left( \frac{1}{W_{i,1}}-\frac{1}{X_i}\right)^{-1} \quad \text{for } 1\leq i\leq n.$$
Since $W_{i,1}\sim\operatorname{Gamma}^{-1}(\theta_0)$, we have  $\widetilde{W}_{i,1}\sim\operatorname{Gamma}^{-1}(\theta_c+Un^{-1/3})$ and by construction $W_{i,1}\leq \widetilde{W}_{i,1}$ for all $1\leq i\leq n$.
It follows that $\widetilde{Z}(n,m)$ is an upper bound for $Z(n,m)$
\begin{equation} \label{eq74}
Z(n,m)\leq \widetilde{Z}(n,m).
\end{equation} 
Furthermore,  \eqref{eq:bbp}   yields the limiting behavior of  $\widetilde{Z}(n,m)$ 
    \begin{equation}\label{eq75}
       \lim_{n\rightarrow\infty} \mathbb{P}\left( \frac{\log \widetilde{Z}(n,m) -nf_{\theta,p}}{n^{1 / 3} \sigma_{\theta,p}} \leq t\right)= F_{\mathrm{BBP} ; -U \lim _{n} \sigma_{\theta, p}}(t).
    \end{equation} 
    By  (\ref{eq74}) and (\ref{eq75}), 
\begin{align*}
    \liminf_{n\rightarrow\infty}{} &\mathbb{P}\left( \frac{\log {Z}(n,m) -nf_{\theta,p} }{n^{1 / 3} \sigma_{\theta,p}} \leq t\right)\geq
    F_{\mathrm{BBP} ; -U}(t) .
    \end{align*} 
It is known that (see \cite[(2.36)]{BaikRains2})
$$
\lim_{U\rightarrow\infty}F_{\mathrm{BBP} ; -U}(t)=F_{\mathrm{GUE}}(t).
$$
Since $U$ can be arbitrary large, it follows that
\begin{align} \label{eq103}
    \liminf_{n\rightarrow\infty}  {}\mathbb{P}\left( \frac{\log {Z}(n,m) -nf_{\theta,p} }{n^{1 / 3} \sigma_{\theta,p}} \leq t\right) \geq
       F_{\mathrm{GUE}}(t).
    \end{align}

\subsubsection*{Lower bound}
Let $Z(k,\ell|n,m)$ denote the point-to-point partition function with starting point $(k,\ell)$ and endpoint $(n,m)$.
We have 
\begin{equation}\label{eq42}
W_{1,1}Z(1,2| n,m)\leq Z(n,m).
\end{equation} 
Observe that $Z(1,2| n,m)$ can be regarded as the partition function on a log-gamma polymer of size $n$ times $m-1$ with homogeneous inverse-gamma weights of parameter $\theta$. Thus  \eqref{eq:gue} can be applied to $Z(1,2|n,m)$
 \[
       \lim_{n\rightarrow\infty} \mathbb{P}\left( \frac{ \log Z(1,2|n,m)-nf_{\theta,p}}{n^{1 / 3} \sigma_{\theta,p}} \leq t\right)=\lim_{n\rightarrow\infty} \mathbb{P}\left( \frac{ \log Z(1,2|n,m)-nf_{\theta,(m-1)/n}}{n^{1 / 3} \sigma_{\theta,(m-1)/n}} \leq t\right)= F_{\mathrm{GUE}}(t).
    \]
    Note that $n^{-1/3}\log W_{1,1}\overset{(d)}{\longrightarrow}0$, by Slutsky's theorem we obtain
 \begin{equation}\label{eq76}
       \lim_{n\rightarrow\infty} \mathbb{P}\left( \frac{\log W_{1,1}+\log Z(1,2|n,m)-nf_{\theta,p}}{n^{1 / 3} \sigma_{\theta,p}} \leq t\right)= F_{\mathrm{GUE}}(t).
    \end{equation} 
    By  (\ref{eq42}) and (\ref{eq76}),
\begin{align}\label{eq104}
   \limsup_{n\rightarrow\infty} \mathbb{P}\left( \frac{\log {Z}(n,m) -nf_{\theta,p} }{n^{1 / 3} \sigma_{\theta,p}} \leq t\right) \leq F_{\mathrm{GUE}}(t).
    \end{align} 
    Combine \eqref{eq103} with \eqref{eq104}, we obtain
\begin{align*}
      \lim_{n\rightarrow\infty} \mathbb{P}\left( \frac{\log {Z}(n,m) -nf_{\theta,p} }{n^{1 / 3} \sigma_{\theta,p}} \leq t\right) = F_{\mathrm{GUE}}(t),
    \end{align*} 
    which completes  the proof for asymptotics in the GUE Tracy-Widom phase in Theorem \ref{main2}.

\subsection{Proof of asymptotics in the Gaussian phase} 
Similarly to the proof in the GUE Tracy-Widom   phase, we will also show that $ Z(n,m)$ has an upper and a lower bound which, under the same scaling, converge in law to the same limiting distribution.
Recall that $\theta_0$ may depend on  $n$ and  the assumption here is that   for some $\alpha\in\left(\frac{2}{3},1\right]$, the   limit  \begin{equation}\label{eq:hypo}
    \lim_{n\rightarrow \infty}
{m^{-\alpha}}\left({n\psi^{\prime}(\theta_0)-m\psi^{\prime}(\theta-\theta_0)}\right)
\end{equation} exists and is positive.  We make a technical remark here, which will be used in the proof of lower bound.
\begin{remark}\label{rmk_compact}
 The parameter $\theta_0$ takes value in a compact subset of $(0,\theta)$.  Indeed, since $n\leq m\leq \delta^{-1}n$ from Definition \ref{notation2}, $\limsup_{n\to\infty} \theta_0<\theta $ because otherwise the limit in \eqref{eq:hypo} would not be positive. Now suppose $\liminf_{n\to\infty}  \theta_0=0$ for a contradiction. This would imply $$\limsup_{n\rightarrow \infty}
{m^{-\alpha}}\left({n\psi^{\prime}(\theta_0)-m\psi^{\prime}(\theta-\theta_0)}\right)=\limsup_{n\rightarrow \infty}
{m^{-\alpha}}n\psi^{\prime}(\theta_0)=\infty,$$
which contradicts our assumption that the limit is finite. This contradiction proves  $\liminf_{n\to\infty}  \theta_0>0$. Since $\theta_0$ is a function of $n$, a variable that takes a countable number of values, we have   shown $\theta_0$ takes value in a compact subset of $(0,\theta)$.
\end{remark}
It also follows from \eqref{eq:hypo} that there exists $n_0$ such that $\theta_0<\theta_c$ holds for all $n>n_0$.
In the following proof we suppose without loss of generality that $\theta_0<\theta_c$ for all $n$.
\subsubsection*{Upper bound}
In \cite{Seppa}, Sepp{\"a}l{\"a}inen studied a   log-gamma polymer whose distribution along down-right paths is stationary.
The weight array of this log-gamma polymer is as follows  
\begin{equation*}
    W_{1,1}^{\mathrm{stat}}=1 \quad \text{ and }\quad 
    W_{i,j}^{\mathrm{stat}}\sim\left\{
    \begin{array}{ll}
         \operatorname{Gamma}^{-1}(\theta_0),&i>j=1 , \\
         \operatorname{Gamma}^{-1}(\theta-\theta_0),& j>i=1,\\
         \operatorname{Gamma}^{-1}(\theta),& i,j\geq 2.
    \end{array}
    \right.
\end{equation*} 
We denote the corresponding point-to-point partition function by $Z^{\mathrm{stat}}(n,m)$.
Observe that the stationary partition function provides  an upper bound for $Z(n,m)$ 
\begin{equation}\label{eq48}
    Z^{\mathrm{stat}}(n+1,m)\geq Z^{\mathrm{stat}}(2,1|n+1,m)\overset{(d)}{=}Z(n,m).
\end{equation}
A central limit theorem for  $  Z^{\mathrm{stat}}$
was given in \cite{Seppa}.
\begin{proposition}[{\cite[(2.5) and Corollary 2.2]{Seppa}}]\label{prop_stat} 
Suppose that for some $\alpha>\frac{2}{3}$, the   limit 
$\lim_{n\rightarrow \infty}
{m^{-\alpha}}\left({n\psi^{\prime}(\theta_0)-m\psi^{\prime}(\theta-\theta_0)}\right)$ exists and is positive.
Then  as $n\rightarrow\infty$,
\[
\frac{\log Z^{\mathrm{stat}}(n,m)-n\bar f_{\theta,p} }{n^{1/2}\sqrt{ {\psi^{\prime}(\theta_0)}-p\psi^{\prime}(\theta-\theta_0)}} \overset{(d)}{ \longrightarrow} \mathcal{N}(0,1).
\] 
\end{proposition} 
Combining \eqref{eq48} with Proposition \ref{prop_stat},
\begin{align}\nonumber
\liminf_{n\rightarrow\infty}\mathbb{P} &\left( \frac{\log Z(n,m)-n\bar f_{\theta,p} }{n^{1/2}\sqrt{  {\psi^{\prime}(\theta_0)}-p\psi^{\prime}(\theta-\theta_0) }} \leq t  \right) \\
\label{eq94}
&\geq
\liminf_{n\rightarrow\infty}\mathbb{P}\left( \frac{\log Z^{\mathrm{stat}}(n+1,m)-(n+1)\bar f_{\theta,m/(n+1)} }{ \sqrt{ (n+1) \psi^{\prime}(\theta_0)-m\psi^{\prime}(\theta-\theta_0) }} \leq t  \right)= \Phi(t).
\end{align} 
\subsubsection*{Lower bound} 

For each $n$, fix some  $s\in (0,p^{-1})$ (i.e. we allow $s$ to depend on $n$). The value of $s$   will be precised    later. In the summation appearing in   the point-to-point partition function $Z(n,m+1)$, instead of summing over all up-right paths, we may sum over only up-right paths that pass through the point $(n-\lfloor ms\rfloor +1,1)$, which yields a lower bound for $Z(n,m+1)$ 
\begin{equation}\label{eq47}
\log Z(n,m+1)\geq \log Z(n-\lfloor ms\rfloor +1,1)+ \log Z(n-\lfloor ms\rfloor +1, 2|n,m+1).
\end{equation}
Note that    the first term in the RHS of (\ref{eq47}) is a sum of i.i.d.  random variables
\[\log Z(n-\lfloor ms\rfloor +1,1)=\sum_{i=1}^{n-\lfloor ms\rfloor +1} \log W_{i,1},\]
and the CLT may be applied to this sum provided $n-\lfloor ms \rfloor \rightarrow\infty$.
The second term in the RHS of (\ref{eq47}) can be identified with the partition function for a homogeneous log-gamma polymer.
In fact, let  $Z^{\mathrm{h}}(n,m)$ (where the letter $\mathrm{h}$ stands for homogeneous) denote the partition function on the log-gamma polymer whose  weight array is given via
 $$W^{\mathrm{h}}_{i,j}\sim\operatorname{Gamma}^{-1}(\theta)\text{ for all }1\leq i\leq n,1\leq j\leq m.$$
 Then we have  
\[
Z(n-\lfloor ms\rfloor +1, 2|n,m+1)\overset{(d)}{=} Z^{\mathrm{h}}(\lfloor ms\rfloor, m).
\] 
The following estimate for the free energy $\log Z^{\mathrm{h}}(n,m)$ of homogeneous log-gamma polymers is proved in \cite[Theorem 2.4]{Seppa} (see also the remark that follows \cite[Theorem 2.1]{Seppa}). 
\begin{proposition}[{\cite[Theorem 2.4]{Seppa}}]\label{prop_gauss}
Let $Z^{\mathrm{h}}(n,m)$ be the point-to-point partition function for the log-gamma polymer   whose weights are independent and are all distributed as $\operatorname{Gamma}^{-1}(\theta)$.
Let $X\subset (0,\theta)$ be a compact set.
For all $\theta_s\in X$, define   $s:=\frac{\psi^{\prime}(\theta-\theta_s)}{\psi^{\prime}(\theta_s)}$.
 Then there exist constants $C,M$ depending on $X$, $\theta$ and $\delta$ such that 
 \[
 \mathbb{E}\left[ \left| \frac{\log Z^{\mathrm{h}}(\lfloor ms \rfloor , m)+ms\psi(\theta_s)+m\psi(\theta-\theta_s)}{m^{1/3}} \right|\right] \leq C
 \]
 holds for all $m\geq M, \theta_s\in X$.
\end{proposition}
Now we perform a heuristic calculation so as to choose the value of $s $ that approximately optimizes the expectation of the RHS of (\ref{eq47}), which will   give us the most tight lower bound with \eqref{eq47}. Using Proposition \ref{prop_gauss} and the fact that $\mathbb{E}[\log W_{i,1}]=-\psi(\theta_0)$ for $1\leq i\leq n- \lfloor ms\rfloor+1$,
\begin{multline}
\mathbb{E}\left[\log Z(n-\lfloor ms\rfloor +1,1)+\log Z(n-\lfloor ms\rfloor +1, 2|n,m+1)\right] \\ 
=-(n-\lfloor ms\rfloor +1)\psi(\theta_0)
-ms\psi(\theta_s)-m\psi(\theta-\theta_s) +O(m^{1/3}) \\
=  -(n- ms)\psi(\theta_0)
-ms\psi(\theta_s)-m\psi(\theta-\theta_s) +O(m^{1/3}) . \label{eq49}
\end{multline}
Neglecting the $m^{1/3}$ order term, its derivative with respect to $s$ is given by
\begin{equation*}
    m\left(\psi(\theta_0)- \psi(\theta_s)-(s\psi^{\prime}(\theta_s)-\psi^{\prime}(\theta-\theta_s))\frac{d\theta_s}{d s}\right)=  m\psi(\theta_0)- m\psi(\theta_s).
\end{equation*} 
As $s$ increases from $0$ to $p^{-1}$, $\theta_s$ increases from $0$ to $\theta_c$ and thus $m\psi(\theta_0)- m\psi(\theta_s)$ decreases from $+\infty $ to 
 $m\psi(\theta_0)- m\psi(\theta_c)<0$ (because we assumed  $\theta_0<\theta_c$). 
 Now, let  $s$ be such that $\theta_s=\theta_0$, that is, 
 $$
 s=\frac{\psi^{\prime}(\theta-\theta_0)}{\psi^{\prime}(\theta_0)}.
$$ Then, approximately, the derivative of  (\ref{eq49}) is positive on $(0,s)$ and negative on $(s,p^{-1})$. This indicates that (\ref{eq49}) is approximately maximized by choosing $s$ as such.
 
 Observe that  with this choice of $s$,  the hypothesis \eqref{eq:hypo} implies
 $$n-ms= \frac{n\psi^{\prime}(\theta_0)-m\psi^{\prime}(\theta-\theta_0)}{\psi^{\prime}(\theta_0)} \longrightarrow+\infty,$$
 which enables us to apply the CLT  to $\sum_{i=1}^{n-\lfloor ms \rfloor +1}\log W_{i,1}$,
\begin{align}\label{eq70}
\frac{\log Z(n-\lfloor ms\rfloor +1,1)+(n-ms)\psi(\theta_0) }{\sqrt{n ( {\psi^{\prime}(\theta_0)}-p\psi^{\prime}(\theta-\theta_0))}} =\frac{\sum_{i=1}^{n-\lfloor ms \rfloor +1}\log W_{i,1}+(n-ms)\psi(\theta_0) }{\sqrt{ (n- ms )
\operatorname{Var} (\log W_{1,1})}} \overset{(d)}{\longrightarrow}\mathcal{N}(0,1),
\end{align} 
where we have used the fact that $\operatorname{Var}(\log X)=\psi^{\prime}(a)$ if $X$ is an inverse gamma  variable with parameter $a$. Recall from Remark \ref{rmk_compact} that  $\theta_0$ takes value in a compact set. We may therefore apply Proposition \ref{prop_gauss} to $Z(n-\lfloor ms\rfloor +1, 2|n,m+1)$. It follows that 
 \begin{align*}
      &\frac{ \log Z(n-\lfloor ms\rfloor +1, 2|n,m+1)+ms\psi(\theta_0)+m\psi(\theta-\theta_0)}{\sqrt{n ( {\psi^{\prime}(\theta_0)}-p\psi^{\prime}(\theta-\theta_0)) }}\\
     =& \sqrt{\frac{m^{2/3}}{{n ( {\psi^{\prime}(\theta_0)}-p\psi^{\prime}(\theta-\theta_0)) }} }\frac{ \log Z(n-\lfloor ms\rfloor +1, 2|n,m+1)+ms\psi(\theta_0)+m\psi(\theta-\theta_0)}{m^{1/3}}\overset{L^1}{\longrightarrow}0.
 \end{align*} 
 Combine this with (\ref{eq70}) and apply Slutsky's theorem,
\begin{align}\nonumber
\frac{\log Z(n-\lfloor ms\rfloor +1,1)  Z(n-\lfloor ms\rfloor +1, 2|n,m+1)+n\psi(\theta_0)+m\psi(\theta-\theta_0)}{{n }^{1/2}\sqrt{  {\psi^{\prime}(\theta_0)}-p\psi^{\prime}(\theta-\theta_0) }}  \overset{(d)}{\longrightarrow}\mathcal{N}(0,1).
\end{align}
Using \eqref{eq47} this implies that  for all $t$,
\begin{align}\nonumber
\limsup_{n\rightarrow\infty}\mathbb{P}\left(
\frac{\log Z(n,m+1)+n\psi(\theta_0)+m\psi(\theta-\theta_0)}{{n }^{1/2}\sqrt{  {\psi^{\prime}(\theta_0)}-p\psi^{\prime}(\theta-\theta_0) }} \leq t  \right)\leq \Phi(t).
\end{align} 
Substituting $m$ for $m+1$, we obtain
\begin{align}\label{eq95}
\limsup_{n\rightarrow\infty}\mathbb{P}\left(
\frac{\log Z(n,m)+n\psi(\theta_0)+m\psi(\theta-\theta_0)}{{n }^{1/2}\sqrt{  {\psi^{\prime}(\theta_0)}-p\psi^{\prime}(\theta-\theta_0) }} \leq t  \right)\leq \Phi(t).
\end{align} 
To conclude, we combine \eqref{eq94} with  \eqref{eq95}, and we obtain that 
   \[ \lim_{n\rightarrow\infty}\mathbb{P}\left(
\frac{\log Z(n,m)-n\bar f_{\theta,p}}{n^{1/2}\sqrt{ {\psi^{\prime}(\theta_0)}-p\psi^{\prime}(\theta-\theta_0)}} \leq t\right)=\Phi(t).
\]  
    

\bibliographystyle{mybibstylealpha}
\bibliography{main}
\end{document}